\RequirePackage{etoolbox}
\csdef{input@path}{{style/}{graphics/}}
\documentclass[seceqn,rotating,MSNbibl,number,citesort,dvips]{arxbj}
\usepackage{mathbh,upgreek,dcolumn}
\usepackage{graphicx}

\aid{0}
\volume{21}
\issue{1}
\pubyear{2015}
\firstpage{83}
\lastpage{143}
\doi{10.3150/13-BEJ562} %kopijuoti is PTS

\makeatletter

\newcommand{\rright}{\right}
\newcommand{\lleft}{\left}
\newcommand{\rrVert}{\Vert}
\newcommand{\rrvert}{\vert}
\newcommand{\llVert}{\Vert}
\newcommand{\llvert}{\vert}

\def\argmin{\operatorname{\arg\min}}
\newcommand{\indic}{\mathbh{1}}
\newtheorem{Th}{Theorem}
\newtheorem{Prop}{Proposition}
\newtheorem{Lemma}{Lemma}
\newtheorem{Cor}{Corollary}
\newproclaim{Experiment}{Experiment}

\newcommand{\E}{\mathbb{E}}

\renewcommand{\P}{\mathbb{P}}

\newcommand{\R}{\mathbb{R}}
\newcommand{\Q}{\mathbb{Q}}

\renewcommand{\L}[1]{\mathbb{L}_{#1}}
\newcommand{\e}{\varepsilon}

\newcommand{\p}{\varphi}

\newcommand{\red}[1]{\mathbf{#1}}

\newcommand{\proc}{\mathrm{proc}}
\newcommand{\eqref}[1]{(\ref{#1})}

\newcolumntype{d}[1]{D{.}{.}{#1}}
\def\sfrac#1#2{#1/#2}

\def\afrac#1#2{#1/(#2)}

\makeatother

\begin{document}
\begin{frontmatter}

\title{Lasso and probabilistic inequalities for multivariate point processes}
\runtitle{Lasso and probabilistic inequalities for multivariate point processes}

\begin{aug}
%%%% inicialai - be tarpu
\author[1]{\inits{N.R.}\fnms{Niels Richard}~\snm{Hansen}\thanksref{1}\ead[label=e1]{Niels.R.Hansen@math.ku.dk}},
\author[2]{\inits{P.}\fnms{Patricia}~\snm{Reynaud-Bouret}\thanksref{2}\ead[label=e2]{Patricia.Reynaud-Bouret@unice.fr}}
\and
\author[3]{\inits{V.}\fnms{Vincent}~\snm{Rivoirard}\corref{}\thanksref{3}\ead[label=e3]{Vincent.Rivoirard@dauphine.fr}}
%%\runauthor{} %% auto
\address[1]{Department of Mathematical
Sciences, University of Copenhagen, Universitetsparken 5, 2100
Copenhagen, Denmark. \printead{e1}}
\address[2]{Univ. Nice Sophia Antipolis, CNRS, LJAD, UMR 7351, 06100
Nice, France.\\ \printead{e2}}
\address[3]{CEREMADE, CNRS-UMR 7534, Universit\'e Paris Dauphine, Place
Mar\'echal de Lattre de Tassigny, 75775 Paris Cedex 16,
France. INRIA Paris-Rocquencourt, projet Classic.\\ \printead{e3}}
\end{aug}

% HISTORY:
\received{\smonth{9} \syear{2012}}
\revised{\smonth{7} \syear{2013}}

% ABSTRACT
%
\begin{abstract}
Due to its low computational cost, Lasso is an attractive
regularization method for high-dimensional statistical
settings. In this paper, we consider multivariate counting
processes depending on an unknown function parameter to be
estimated by linear combinations of a fixed dictionary. To
select coefficients, we propose an adaptive
$\ell_1$-penalization methodology, where data-driven
weights of the penalty are derived from new
Bernstein type inequalities for martingales. Oracle inequalities are
established under assumptions on the Gram matrix of the
dictionary. Nonasymptotic probabilistic results for
multivariate Hawkes processes are proven, which allows us
to check these assumptions by considering general
dictionaries based on histograms, Fourier or wavelet
bases. Motivated by problems of neuronal activity
inference, we finally carry out a simulation study for
multivariate Hawkes processes and compare our methodology
with the \textit{adaptive Lasso procedure} proposed by Zou in
(\textit{J. Amer. Statist. Assoc.} \textbf{101} (2006)
1418--1429). We observe an excellent behavior of our procedure. We rely on theoretical
aspects for the essential question of tuning our
methodology. Unlike adaptive Lasso of (\textit{J. Amer. Statist. Assoc.} \textbf{101} (2006)
1418--1429), our
tuning procedure is proven to be robust with respect to
all the parameters of the problem, revealing its potential
for concrete purposes, in particular in neuroscience.
\end{abstract}

% KEYWORDS
% visi is mazosios raides ir pagal abecele
%

\end{frontmatter}

%s1 #&#
\section{Introduction}
The Lasso, proposed in \cite{Tib}, is a well-established method that
achieves sparsity of an
estimated parameter vector via $\ell_1$-penalization. In this paper, we
focus on using Lasso to select and estimate coefficients in the
basis expansion of intensity processes for multivariate point
processes.

Recent examples of applications of multivariate point processes
include the modeling of multivariate neuron spike data
\cite{Masud:2011,Pillow:2008}, stochastic kinetic modeling
\cite{Bowsher:2010} and the modeling of the distribution of ChIP-seq
data along the genome \cite{CSWH}. In the previous examples, the
intensity of a future occurrence of a point depends on the history of
all or some of the coordinates of the point processes, and it is of
particular interest to estimate this dependence. This can be achieved
using a parametric family of models, as in several of the papers
above. Our aim is to provide a nonparametric method based on the
Lasso.

The statistical properties of Lasso are particularly well understood
in the context of regression with i.i.d. errors or for density
estimation where a range of \emph{oracle inequalities} have been
established. These inequalities, now widespread in the literature,
provide theoretical error bounds that hold on events with a
controllable (large) probability. See, for instance,
\cite{bertin11:_adapt_dantiz,BRT,bunea4,bunea3,bunea2,bunea,geer}.
We refer the reader to \cite{Buhlmann:2011} for an excellent account
on many state-of-the-art results. One main challenge in this context
is to obtain as weak conditions as possible on the design -- or Gram
-- matrix. The other important challenge is to be able to provide an
$\ell_1$-penalization procedure that provides excellent performance
from both theoretical and practical points of view. Standard Lasso
proposed in \cite{Tib} and based on deterministic constant weights
constitutes a major contribution from the methodological point of view,
but underestimation due to its shrinkage nature may lead to poor practical
performance in some contexts. Alternative two step procedures have been
suggested to overcome this drawback (see \cite{Mein,vdGBZ,Zou}). Zou
in \cite{Zou} also discusses problems for standard Lasso to cope with
variable selection and consistency simultaneously. He overcomes these
problems by introducing nonconstant data-driven $\ell_1$-weights
based on preliminary consistent estimates.

%s1.1 #&#
\subsection{Our contributions}

In this paper, we consider an $\ell_1$-penalized least squares
criterion for the estimation of coefficients in the expansion of a
function parameter. As in \cite{bertin11:_adapt_dantiz,HMZ,vdGBZ,Zou}, we consider nonconstant data-driven weights. However the setup
is here that of multivariate point processes
and the function parameter that lives in a Hilbert space determines
the point process intensities.
Even in this unusual context, the least squares criterion also involves a
random Gram matrix as well, and in this respect, we present a fairly
standard oracle inequality
with a strong condition on this Gram matrix (see Theorem~\ref
{th:oracle} in Section~\ref{sec:Lasso}).

One major contribution of this article is to provide probabilistic
results that enable us to calibrate $\ell_1$-weights in the most
general setup (see Theorem~\ref{th:oraclewithd} in Section~\ref{sec:Lasso}).
This is naturally linked to sharp Bernstein type
inequalities for martingales. In the literature, those kinds of
inequalities generally provide upper bounds for the martingale that
are deterministic and unobservable \cite{SW,vdG}. To choose
data-driven weights, we need observable bounds. More recently, there
have been some attempts to use self-normalized processes in order to
provide more flexible and random upper bounds \cite{BT,penaa,Pena,DvZ}. Nevertheless, those bounds are usually not (completely)
observable when dealing with counting processes. We prove a result
that goes further in this direction by providing a completely sharp
random observable upper bound for the martingale in our counting
process framework (see Theorem~\ref{th:weakBer} in Section~\ref{sec:Bernstein}).

The second main contribution is to provide a quite theoretical and
abstract framework to deal with processes whose intensity is (or is
well approximated by) a linear transformation of deterministic
parameters to infer.
%in any situations where the true data may be well approximated by
%intensities with linear shape.
This general framework also allows for different asymptotics in terms
of the number of observed processes or in terms of the duration of the
recording of observations, according to the setup. We focus in this
paper on three main examples: the Poisson model, the Aalen
multiplicative intensity model and the multivariate Hawkes process,
but many other situations can be expressed in the present framework,
which has the advantage of full flexibility. The first two examples
have been extensively studied in the literature as we detail
hereafter, but Hawkes processes are typical of situations where very little
is known from a nonparametric point of view, and where fully
implementable adaptive methods do not exist until the present work, to
the best of our knowledge. They also constitute processes that are
often used
in practice -- in particular in neuroscience -- as explained below.

It is also notable that we, in
each of these three previous examples, can verify explicitly if the
strong condition
on the Gram matrix mentioned previously is fulfilled with probability
close to 1 (see Section~\ref{sec:oraclegeneral} for the Poisson and Aalen cases and Section~\ref{sec:hawkes}
for the Hawkes case). For
the multivariate Hawkes process, this involves novel probabilistic
inequalities. Even though the Hawkes processes have been studied
extensively in the literature, see \cite{BM,DVJ}, very little is known
about exponential inequalities and nonasymptotic tail
control. Besides the univariate case \cite{RBRoy},
no exponential inequality controlling the number of points
per interval is known to us. We derive such results and other sharp
controls on the convergence in the ergodic theorem to obtain
control on the Gram matrix.

Finally, we carry out a simulation study in Section~\ref{sec:simu} for the most
intricate process, namely the multivariate Hawkes process, with a main
aim: to convince practitioners, for instance in neuroscience, that
this method is indeed fully implementable and gives good results in
practice. Data-driven weights for practical purposes are slight
modifications of theoretical ones. These modifications essentially aim
at reducing the number of tuning parameters to one. Due to
nonnegligible shrinkage that is unavoidable, in particular for large
coefficients, we propose a two step procedure where estimation of
coefficients is handled by using ordinary least squares estimation on
the support preliminary determined by our Lasso methodology. Tuning
issues are extensively investigated in our simulation study, and
Table~\ref{tableN2} in Section~\ref{results} shows that our
methodology can easily and robustly be tuned by using limit values
imposed by assumptions of Theorem~\ref{th:oraclewithd}. We naturally
compare our procedure with \textit{adaptive Lasso} of \cite{Zou} for
which weights are proportional to the inverse of ordinary least
squares estimates. The latter is very competitive for estimation
aspects since shrinkage becomes negligible if the preliminary
OLS estimates are large. But adaptive Lasso does not incorporate
random fluctuations of coefficient estimators. So it is most of the
time outperformed by our procedure. In particular, tuning adaptive
Lasso in the Hawkes setting is a difficult task, which cannot be
tackled by using standard cross-validation. Our simulation study shows
that the performance of adaptive Lasso is very sensitive to the choice
of the tuning
parameter. %which highly depends on $T$ in a complicated manner.
Robustness with respect to tuning is another advantage of our method
over adaptive Lasso. For simulations, the framework of neuronal
networks is used. Our short study proves that our methodology can be
used for solving concrete problems in neuroscience such as the
inference of functional connectivity graphs.

%%%%%%%%%%%%%%%%%%%%%%%%%%%
%s1.2 #&#
\subsection{Multivariate counting process}\label{general}
The framework introduced here and used throughout the paper aims at
unifying several situations, making the reading easier. Main examples
are then shortly described, illustrating the use of this setup.

We consider an $M$-dimensional counting process
$(N_t^{(m)})_{m=1,\ldots,M}$, which can also be seen as a random point
measure on $\R$ with marks in $\{1, \ldots, M\}$, and corresponding
\emph{predictable}
intensity processes $(\lambda_t^{(m)})_{m = 1, \ldots, M}$ under a
probability measure $\mathbb{P}$ (see \cite{Bre} or \cite{DVJ} for
precise definitions).

Classical models assume that the intensity $\lambda_t^{(m)}$ can be
written as
a \emph{linear predictable transformation} of a deterministic function
parameter $f_0$ belonging to a Hilbert space $\mathcal{H}$ (the
structure of $\mathcal{H}$, and then of $f_0$, will differ according to
the context, as illustrated below). We denote this
linear transformation by
%
%e1.1 #&#
\begin{equation}
\label{defpsi} \psi(f) = \bigl(\psi^{(1)}(f),\ldots,\psi^{(M)}(f)
\bigr).
\end{equation}
Therefore, for classical models, for any $t$,
%
%e1.2 #&#
\begin{equation}
\label{egalite} \lambda^{(m)}_t = \psi^{(m)}_t(f_0).
\end{equation}
The main goal in classical settings is to estimate $f_0$ based on observing
$(N_t^{(m)})_{m=1,\ldots,M}$ for $t\in[0,T]$. Actually, we do not
require in Theorems~\ref{th:oracle} and \ref{th:oraclewithd} that (\ref{egalite}) holds. Our aim is
mainly to furnish an estimate of the best linear approximation $\psi^{(m)}
_t(f_0)$ of the underlying intensity $\lambda^{(m)}_t$. %From
%estimation of $f_0$, we can then deduce estimation of intensity
%processes.

Let us illustrate the general setup with three main examples: First,
the case
with i.i.d. observations of an inhomogeneous Poisson process on
$[0,1]$ and unknown intensity, second, the well known Aalen
multiplicative intensity model and third,
the central example of the multivariate Hawkes process. For the first
two models, asymptotics is with respect to $M$ ($T$ is fixed). For the
third one, $M$ is fixed and asymptotics is with respect to $T$.
%In Section
%our setup.

%%%%%%%%%%
%s1.2.1 #&#
\subsubsection{The Poisson model}\label{sec:PoiIntro}
Let us start with a very simple example which will be somehow a toy
problem here compared to the other examples.
In this example, we take $T = 1$ and assume that we observe $M$
i.i.d. Poisson processes on $[0,1]$ with common intensity
$f_0\dvtx [0,1]\longmapsto\R_+$. Asymptotic properties are obtained when
$M$ tends to infinity. In this case, the intensity $\lambda^{(m)}$ of the
$m$th process $N^{(m)}$ is $f_0$, which does not depend on $m$: Therefore,
for any $m\in\{1,\ldots,M\}$ and any $t$, we set
\[
\psi^{(m)}_t(f_0):=f_0(t),
\]
and $ \mathcal{H} = \L{2}([0,1])$ is equipped with the classical norm
defined by
\[
\llVert f \rrVert = \biggl(\int_0^1
f^2(t)\,\mathrm{d}t \biggr)^{1/2}.
\]
%
%In this case, the support of $f$, namely $[0,1]$, does not play a
%fundamental role.
This framework has already been extensively studied from an adaptive
point of view: see
for instance \cite{Reynaud-Bouret,WN} for model selection methods,
\cite{RBR} for wavelet thresholding estimation or \cite{RBTMRG} for
kernel estimates. In this context, our present general result matches
with existing minimax adaptation results where asymptotics is with
respect to $M$.
%%%%%%%%%%%
%s1.2.2 #&#
\subsubsection{The Aalen multiplicative intensity model}\label{sec:AalIntro}
This is one of the most popular counting processes because of its
adaptivity to various situations (from Markov models to censored
lifetime models) and its various applications to biomedical data (see
\cite{andersen}). Given $\mathcal{X}$ a Hilbert space,
we consider $f_0\dvtx [0,T]\times\mathcal{X}\longmapsto\R_+$, and we set for
any $t\in\R$,
\[
\lambda^{(m)}_t = \psi^{(m)}_t(f_0):=
f_0\bigl(t,X^{(m)}\bigr)Y_t^{(m)},
\]
where $Y^{(m)}$ is an observable predictable process and $X^{(m)}$
represents covariates. In this case, $\mathcal{H}=\L{2}([0,T]\times
\mathcal{X})$. Our goal is to estimate $f_0$ and not to select
covariates. So, to fix ideas one sets $\mathcal{X}=[0,1]$ and
$T=1$. Hence $\mathcal{H}$ can be identified with
$\L{2}([0,1]^2)$. For right-censored data, $f_0$ usually represents
the hazard rate. The presence of covariates in this pure nonparametric
model is the classical generalization of the semi-parametric
model proposed by Cox (see \cite{Letue}, for instance). Note that the
Poisson model is a special case of the Aalen model.

The classical framework consists in assuming that $(X^{(m)}, Y^{(m)},
N^{(m)})_{m=1,\ldots,M}$ is an i.i.d. $M$-sample and as for the
Poisson model, it is natural to investigate asymptotic properties when
$M\to+\infty$. If there are no covariates, several adaptive approaches
already exist: See \cite{BCa,BCb,RB} for various penalized
least-squares contrasts and \cite{chagny} for kernel methods in
special cases of censoring. In the presence of covariates, one can
mention \cite{aalen,andersen} for a parametric approach,
\cite{CGG,Letue} for a model selection approach and \cite{GG} for a
Lasso approach. Let us also cite \cite{BmK} where covariate selection
via penalized MLE has been studied. Once again, our present general
result matches with existing oracle results. In \cite{chagny},
exponential control of random fluctuations leading to adaptive
results are derived without using the martingale theory. In more
general frameworks (as in \cite{CGG}, for instance), martingales are
required and this even when i.i.d. processes are involved.
%%%%%%%%%%%%%%%%%%%%
%s1.2.3 #&#
\subsubsection{Hawkes processes}\label{sec:IntroHawkes}
Hawkes processes are the point processes equivalent to autoregressive
models. In seismology, Hawkes processes can model earthquakes and
their aftershocks \cite{VJO}. More recently they have been used to
model favored or avoided distances between occurrences of motifs
\cite{GS} or Transcription Regulatory Elements \cite{CSWH} on the
DNA.\vadjust{\goodbreak} We can also mention the use of Hawkes processes as models of
social interactions \cite{MC} or financial phenomena \cite{BDHM}.

In the univariate setting, with $M=1$, the intensity of a nonlinear
Hawkes process $(N_t)_{t>0}$ is given by
\[
\lambda_t=\phi \biggl(\int_{-\infty}^{t-}h(t-u)
\,\mathrm{d}N_u \biggr),
\]
where $\phi\dvtx \R\mapsto\R_+$ and $h\dvtx \R_+\mapsto\R$ (see \cite{BM}). A
particular case is Hawkes's self exciting point process, for which $h$
is nonnegative and $\phi(x)=\nu+x$ where $\nu>0$ (see
\cite{BM,DVJ,hawkes}). For instance, for seismological purposes, $\nu$
represents the spontaneous occurrences of real original
earthquakes. The function $h$ models self-interaction: after a shock
at time $u$, we observe an aftershock at time $t$ with large
probability if $h(t-u)$ is large.

These notions can be easily extended to the multivariate setting and in
this case the intensity of $N^{(m)}$ takes the form:
\[
\lambda_t^{(m)}=\phi^{(m)} \Biggl(\sum
_{\ell=1}^M\int_{-\infty}^{t-}h_\ell^{(m)}
(t-u)\,\mathrm{d}N^{(\ell)} (u) \Biggr).
\]
Theorem~7 of \cite{BM} gives conditions on the functions $\phi^{(m)}$
(namely Lipschitz properties) and on the functions $h_\ell^{(m)}$ to obtain
existence and uniqueness of a stationary version of the associated
process. Throughout this paper, we assume that for any $m\in\{1,\ldots
,M\}$,
\[
\phi^{(m)}(x)=\bigl(\nu^{(m)}+x\bigr)_+,
\]
where $\nu^{(m)}>0$ and $(\cdot)_+$ denotes the positive part. Note that in
\cite{CSWH,chornoboy}, the case $\phi^{(m)}(x)=\exp(\nu^{(m)}+x)$ was studied.
However, Lipschitz properties required in \cite{BM} are not satisfied
in this case. By introducing, as previously, the linear predictable
transformation $\psi(f) = (\psi^{(1)}(f),\ldots,\psi^{(M)}(f))$
with for any $m$ and any $t$
%
%e1.3 #&#
\begin{equation}
\label{eq:intensity} \psi^{(m)}_t(f_0):=
\nu^{(m)}+ \sum_{\ell=1}^M \int
_{-\infty}^{t-}h_\ell ^{(m)}(t-u)
\,\mathrm{d}N^{(\ell)} (u),
\end{equation}
with $f_0=(\nu^{(m)}, (h_\ell^{(m)})_{\ell=1,\ldots,M})_{m=1,\ldots,M}$,
we have $\lambda_t^{(m)}=(\psi^{(m)}_t(f_0))_+$. Note that the upper
integration limits in
(\ref{eq:intensity}) are $t-$, that is, the integrations are all over the
open interval $(-\infty, t)$. This assures predictability of the intensity
disregarding the values of $h_\ell^{(m)}(0)$. Alternatively, it can be
assumed throughout that $h_\ell^{(m)}(0) = 0$, in which case the
integrals in (\ref{eq:intensity}) can be over $(-\infty,t]$ without
compromising predictability.
The parameter $\nu^{(m)}$ is called the \textit{spontaneous rate}, whereas
the function $h_\ell^{(m)}$ is called the \textit{interaction function}
of $N^{(\ell)}$ on $N^{(m)}$. The goal is to estimate $f_0$ by using Lasso
estimates. In the sequel, we will assume that the support of $h_\ell
^{(m)}$ is bounded. By
rescaling, we can then assume that the support is in $[0,1]$, and we
will do so throughout. Note
that in this case we will need to observe the process on $[-1,T]$ in
order to compute $\psi^{(m)}_t(f_0)$ for $t \in[0,T]$. The Hilbert space
$\mathcal{H}$ associated with this setup is
\begin{eqnarray*}
\mathcal{H}&=&\bigl(\R\times\L{2}\bigl([0,1]\bigr)^M
\bigr)^M\\
&=& \biggl\{f= \bigl( \bigl(\mu^{(m)},
\bigl(g_\ell^{(m)}\bigr)_{\ell=1,\ldots,M}\bigr)_{m=1,\ldots,M}
\bigr):
\\
&&\hphantom{\biggl\{}g_\ell^{(m)} \mbox{ with support in } [0,1] \mbox{
and }\llVert f \rrVert ^2= \sum_m
\bigl(\mu^{(m)}\bigr)^2 + \sum_m
\sum_\ell \int_0^1
g_\ell^{(m)}(t)^2 \,\mathrm{d}t <\infty \biggr\}.
\end{eqnarray*}
Some theoretical results are established in this general setting but to
go further, we shall consider in Section~\ref{sec:hawkes} the case
where the functions $h_\ell^{(m)}$ are nonnegative and then $\lambda_t^{(m)}$
is a linear function of $f_0$, as in Sections~\ref{sec:PoiIntro} and
\ref{sec:AalIntro}:
\[
\lambda_t^{(m)}=\psi^{(m)}_t(f_0).
\]
The multivariate point process associated with this setup is called
the \textit{multivariate Hawkes self exciting point process} (see \cite
{hawkes}). In this
example, $M$ is fixed and asymptotic properties are obtained when $T$
tends to infinity.

To the best of our knowledge, nonparametric estimation for Hawkes
models has
only been proposed in \cite{RBS} in the univariate setting where the
method is based on $\ell_0$-penalization of the least-squares
contrast. However, due to $\ell_0$-penalization, the criterion is not
convex and the computational cost, in particular for the memory
storage of all the potential estimators, is huge. Therefore, this
method has never been adapted to the multivariate setting. Moreover, the
penalty term in this method is not data-driven and ad-hoc tuning
procedures have been used for simulations. This motivates the present
work and the use of a convex Lasso criterion combined with data-driven
weights, to provide a fully implementable and theoretically valid
data-driven method, even in the multivariate case.

\subsubsection*{Applications in neuroscience} Hawkes processes can
naturally be applied to model neuronal activity. Extracellular action
potentials can be recorded by electrodes and the recorded data for the
neuron $m$ can be seen as a point process, each point corresponding to
the peak of one action potential of this neuron (see \cite{brette}, for
instance, for precise definitions). When $M$ neurons are simultaneously
recorded, one can assume that we are faced with a realization of $N=(N^{(m)}
)_{m=1,\ldots,M}$ modeled by a multivariate Hawkes process.
% and the action potentials of each neuron $m$ are represented by
%realizations of $N^{(m)}$, where $N=(N^{(m)})_{m=1,\ldots,M}$ is a
%multivariate Hawkes process.
We then assume that
the intensity associated with the activity of the neuron $m$ is given
by $\lambda_t^{(m)}=(\psi^{(m)}_t(f_0))_+$, where $\psi^{(m)}_t(f_0)$ is given in
(\ref{eq:intensity}). %In the neurobiological terminology, the
%parameter $\nu^{(m)}$ is called the \textit{spontaneous rate}, whereas the
%function $h_\ell^{(m)}$ is called the \textit{interaction function} of
%$N^{(\ell)}$ on $N^{(m)}$.
At any occurrence $u<t$ of $N^{(\ell)}$, $\psi^{(m)}_t(f_0)$ increases
(excitation) or decreases (inhibition) according to the sign of $h_\ell
^{(m)}(t-u)$. Modeling inhibition is essential from the neurobiological
point of view. So, we cannot assume that all interaction functions are
nonnegative, and we cannot omit the positive part. More details are
given in Section~\ref{sec:simu}.

In neuroscience, Hawkes processes combined with maximum likelihood
estimation have been used in the seminal paper \cite{chornoboy}, but
the application of the method requires a %huge number of observations,
%model with few parameters is known.
too huge number of observations for realistic practical purposes.
%they have emerged in the 1980's with \cite{chornoboy}, once again with
%parametric MLE methods. However, up to our knowledge, this method was
%not used intensively in practice because of the absence of parametric
%models with few parameters in Neurosciences. Indeed a huge number of
%parameters was used to estimate functions by piecewise constant
%functions on a fine regular grid. The MLE method was therefore
%performant only with a huge number of observations that are usually
%not obtained in practice on neuronal data.
Models based on Hawkes processes have nevertheless been recently
discussed in neuroscience, since they constitute one of the few simple
models able to produce dependence graphs between neurons, that may be
interpreted in neuroscience as functional connectivity graphs \cite
{PSCR,PSCR2}. However, many nonparametric statistical questions arise
that are not solved yet in order to furnish a fully applicable tool for
real data analysis \cite{KRS}. We think that the Lasso-based
methodology presented in this paper may furnish the first robust tool
in this direction.
\subsection{Notation and overview of the paper}
Some notation from the general theory of stochastic integration is
useful to simplify the otherwise quite heavy notation.
If $H=(H^{(1)},\ldots,H^{(M)})$ is a multivariate process with locally
bounded coordinates, say, and $X=(X^{(1)},\ldots,X^{(M)})$ is a
multivariate semi-martingale, we define the real valued process
$H\bullet X$ by
\[
H\bullet X_t :=\sum_{m=1}^M
\int_0^t H_s^{(m)}
\,\mathrm{d}X_s^{(m)}.
\]
Given $\mathfrak{F}\dvtx  \R\longmapsto\R$ we use $\mathfrak{F}(H)$ to
denote the
coordinatewise application of $\mathfrak{F}$, that is $\mathfrak{F}(H) =
(\mathfrak{F}(H^{(1)}),\ldots,\mathfrak{F}(H^{(M)}))$. In particular,
\[
\mathfrak{F}(H)\bullet X_t =\sum_{m=1}^M
\int_0^t \mathfrak {F}\bigl(H_s^{(m)}
\bigr) \,\mathrm{d}X_s^{(m)}.
\]
We also define the following scalar product on the space of
multivariate processes. For any multivariate processes
$H=(H^{(1)},\ldots,H^{(M)})$ and $K=(K^{(1)},\ldots,K^{(M)})$, we set
\[
\langle H,K\rangle_{\proc}:=\sum_{m=1}^M
\int_0^T H^{(m)}_s
K^{(m)}_s \,\mathrm{d}s,
\]
the corresponding norm being denoted $\llVert  H  \rrVert _{\proc}$.
Since $\psi$ introduced in (\ref{defpsi}) is linear, the Hilbert space
$\mathcal{H}$ inherits a bilinear form from the previous scalar
product, that we denote, for all $f,g$ in~$\mathcal{H}$,
\[
\langle f, g\rangle_T := \bigl\langle \psi(f),\psi(g)\bigr
\rangle_{\proc} = \sum_{m=1}^M
\int_0^T \psi^{(m)}_s(f)
\psi^{(m)}_s(g) \,\mathrm{d} s,
\]
and the corresponding quadratic form is denoted $\llVert  f  \rrVert _T^2$.

The compensator $\Lambda=(\Lambda^{(m)})_{m=1,\ldots,M}$ of
$N=(N^{(m)})_{m=1,\ldots,M}$ is finally defined for all $t$ by
\[
\Lambda^{(m)}_t=\int_0^t
\lambda^{(m)}_s \,\mathrm{d}s.
\]

Section~\ref{sec:Lasso} gives our main oracle inequality and the choice
of the $\ell_1$-weights in the general framework of counting processes.
Section~\ref{sec:Bernstein} provides the fundamental Bernstein-type
inequality. Section~\ref{sec:oraclegeneral} details the meaning of the
oracle inequality in the Poisson and Aalen setups. The probabilistic
results needed for the Hawkes processes as well as the interpretation
of the oracle inequality in this framework is done in Section~\ref{sec:hawkes}. Simulations on multivariate Hawkes processes are
performed in Section~\ref{sec:simu}. The last Section is dedicated to
the proofs of our results.

%%%%%%%%%%%%%%%%%%%%%%%%%
%%%%%%%%%%%%%%%%%%%%%%%%%
%s2 #&#
\section{Lasso estimate and oracle inequality}\label{sec:Lasso}
We wish to estimate the true underlying intensity so our main goal
consists in estimating the parameter $f_0$. For this purpose, we assume
we are given $\Phi$ a dictionary of functions (whose cardinality is
denoted $\llvert  \Phi \rrvert $) and we define $f_a$ as a linear
combination of the functions of $\Phi$, that is,
\[
f_a:=\sum_{\varphi\in\Phi} a_{\varphi}
\varphi,
\]
where $a=(a_\varphi)_{\varphi\in\Phi}$ belongs to $\R^{\Phi}$. %, where
%$\left| \Lambda \right|$ denotes the cardinal of $\Lambda$.
Then, since $\psi$ is linear, we get
\[
\psi(f_a)=\sum_{\varphi\in\Phi} a_{\varphi}
\psi(\varphi).
\]
We use the following least-squares contrast $\mathcal{C}$ defined on
$\mathcal{H}$ by
%
%e2.1 #&#
\begin{equation}
\label{contrast} \mathcal{C}(f):=-2 \psi(f) \bullet N_T + \llVert f
\rrVert _T^2, \qquad \forall f\in \mathcal{H}.
\end{equation}
This contrast, or some variations of $\mathcal{C}$, have already been
used in particular setups (see, for instance, \cite{RBS} or \cite{GG}).
The main heuristic justification lies in following arguments. Since
$\psi(f)$ is a predictable process, the compensator at time $T$ of
$\mathcal{C}(f)$ is given by
\[
\tilde{\mathcal{C}}(f):=-2 \psi(f) \bullet\Lambda_T + \llVert f
\rrVert _T^2,
\]
which can also be written as $\tilde{\mathcal{C}}(f)= -2 \langle \psi
(f),\lambda\rangle _{\proc} + \llVert  \psi(f)  \rrVert _{\proc}^2$. Note that $\tilde
{\mathcal{C}}$ is minimum when $\llVert  \psi(f)-\lambda  \rrVert _{\proc}$ is
minimum. If $\lambda=\psi(f_0)$ and if $\llVert  \cdot  \rrVert _T$ is a norm on the
Hilbert space $\mathcal{H,}$ then the unique minimizer of $\tilde
{\mathcal{C}}$ is $f_0$. Therefore, to get the best linear
approximation of $\lambda$ of the form $\psi(f)$, it is natural to look
at minimizers of $\mathcal{C}(f)$.
Restricting to linear combinations of functions of $\Phi$, since $\psi$
is linear, we obtain
\[
\mathcal{C}(f_a)=-2 a' b + a'G a,
\]
where $a'$ denotes the transpose of the vector $a$ and for $\varphi_1,
\varphi_2 \in\Phi$,
%
%e2.2 #&#
\begin{equation}
\label{bbG} b_{\varphi_1}=\psi(\varphi_1)\bullet
N_T,\qquad  G_{\varphi_1,\varphi_2}= \langle \varphi_1,
\varphi_2\rangle _T.
\end{equation}
Note that both the vector $b$ of dimension $\llvert  \Phi \rrvert $ and the Gram matrix
$G$ of dimensions $\llvert  \Phi \rrvert \times\llvert  \Phi \rrvert $ are random but
nevertheless observable.

To estimate $a$ we minimize the contrast, $\mathcal{C}(f_a)$, subject
to an
$\ell_1$-penalization on the $a$-vector. That is, we introduce the
following $\ell_1$-penalized estimator
%
%e2.3 #&#
\begin{equation}
\label{Lasso} \hat a\in\mathop{\argmin}\limits_{a\in\R^\Phi}\bigl\{-2a'b+
a'G a+ 2d'\llvert a \rrvert \bigr\},
\end{equation}
where $\llvert  a \rrvert =(\llvert  a_\varphi \rrvert )_{\varphi\in\Phi}$ and $d \in\R_+^\Phi$. With
a good choice of $d$ the solution of (\ref{Lasso}) will achieve both
sparsity and good statistical properties. Finally, we let $\hat
f=f_{\hat a}$ denote the Lasso estimate associated with $\hat a$.
%$$\hat s=\sum_{\lambda\in\Lambda} \hat a_\lambda\phi_\lambda.$$

Our first result establishes theoretical properties of $\hat f$ by
using the classical oracle approach. More precisely, we establish a
bound on the risk of $\hat f$ if some conditions are true. This is a
nonprobabilistic result that only relies on the definition of $\hat
{a}$ by
(\ref{Lasso}). In the next section we will
deal with the probabilistic aspect, which is to prove that the
conditions are fulfilled with large probability.

%th1 #&#
\begin{Th}\label{th:oracle}
Let $c>0$. If
%
%e2.4 #&#
\begin{equation}
\label{normlower} G\succeq c I%\inf_{x\in\R^{\left| \Lambda \right|}_*} \frac{x'Gx}{\left\Vert x  \right\Vert^2_{
\end{equation}
and if for all $\varphi\in\Phi$
%
%e2.5 #&#
\begin{equation}
\label{Coeffcontrol} \llvert b_\varphi-\bar b_\varphi \rrvert \leq
d_\varphi,
\end{equation}
where
\[
\bar b_\varphi=\psi(\varphi)\bullet\Lambda_T,
\]
then there exists an absolute constant $C$, independent of $c$, such that
%
%e2.6 #&#
\begin{equation}
\label{oraclerand} \bigl\llVert \psi(\hat{f})-\lambda \bigr\rrVert
_{\proc}^2\leq C\inf_{a\in\R^{\Phi}} \biggl\{\bigl
\llVert \lambda-\psi(f_a) \bigr\rrVert _{\proc}^2
+ c^{-1}\sum_{\varphi\in
S(a)}d_\varphi^2
\biggr\},
\end{equation}
where $S(a)$ is the support of $a$.
If $\lambda=\psi(f_0)$, the oracle inequality (\ref{oraclerand}) can
also be rewritten as
%
%e2.7 #&#
\begin{equation}
\label{oracle} \llVert \hat{f}-f_0 \rrVert _{T}^2
\leq C\inf_{a\in\R^{\Phi}} \biggl\{\llVert f_0-f_a
\rrVert _{T}^2 + c^{-1}\sum
_{\varphi\in S(a)}d_\varphi^2 \biggr\}.
\end{equation}
\end{Th}

The proof of Theorem~\ref{th:oracle} is given in
Section~\ref{sec:proof-oracle}. Note that assumption \eqref{normlower}
ensures that $G$ is invertible and then coordinates of $\hat a$ are
finite almost surely.
Assumption \eqref{normlower} also ensures that $\llVert  f  \rrVert _T$ is a real
norm on $f$ at least when $f$ is a linear combination of the functions
of $\Phi$.

Two terms are involved on the right-hand sides of (\ref{oraclerand})
and \eqref{oracle}. The first one is an approximation term
and the second one can be viewed as a variance term providing a control
of the random fluctuations of the $b_\p$'s around the $\bar
b_\p$'s. Note that $b_\p-\bar b_\p=\psi(\varphi)\bullet(N-\Lambda)_T$
is a martingale (see also the comments after
Theorem~\ref{th:oraclewithd} for more details). The approximation term
can be small but the price to pay may be a large support of $a$,
leading to large values for the second term. Conversely, a sparse
$a$ leads to a small second term. But in this case the
approximation term is potentially larger. Note that if the function
$f_0$ can be approximated by a sparse linear combination of the
functions of $\Phi$, then we obtain a sharp control of $\llVert  \hat
f-f_0 \rrVert _{T}^2$. In particular, if $f_0$ can be decomposed on the
dictionary, so we can write $f_0=f_{a_0}$ for some $a_0\in\R^\Phi$,
then (\ref{oracle}) gives
\[
\llVert \hat{f}-f_0 \rrVert _{T}^2\leq C
c^{-1}\sum_{\varphi\in S(a_0)}d_\varphi^2.
\]
In this case, the right-hand side can be viewed as the sum of the
estimation errors made by estimating the components of $a_0$.

Such oracle inequalities are now classical in the huge literature of
Lasso procedures. See, for instance, \cite{bertin11:_adapt_dantiz,BRT,bunea4,bunea3,bunea2,bunea,KLT,geer}, who established oracle
inequalities in the same spirit as in Theorem~\ref{th:oracle}. We
bring out the paper \cite{CT}, which gives technical and heuristic
arguments for justifying optimality of such oracle inequalities (see
Section~1.3 of \cite{CT}). Most of these papers deal with independent data.
%Most of these papers, that deal with independent data, aim at
%establishing
%oracle inequalities under assumptions as weak as possible on the
%design matrix.
%We refer the reader to \cite{CondLasso} or \cite{Buhlmann:2011} for a
%good review
%and a hierarchy of these assumptions. Assumption \eqref{normlower},
%that can also be found in \cite{bunea}, is not the weakest one since
%it involves \textit{simultaneously} all columns of $G$ unlike
%assumptions based on restricted isometry constants. Remember that for
%any positive integer $S$, the $S$-restricted isometry constant
%associated with a matrix $G$ is the smallest quantity $\delta_S$
%satisfying
%$$(1-\delta_S)\|x\|_{\ell_2}\leq\|Gx\|_{\ell_2}\leq(1+\delta_S)\|x\|_{
%for any $S$-sparse vector $x$ (see the seminal paper \cite{CT05}).
%As mentioned, the main
%contributions of this paper is not to obtain assumptions as weak as
%possible on the matrix $G$, but rather to

In the sequel, we prove that assumption
\eqref{normlower} is satisfied with large probability by using the
same approach as \cite{RV2,RV3} and to a lesser extent as Section~2.1
of \cite{CT} or \cite{RV1}. Section~\ref{sec:hawkes} is in particular mainly
devoted to show that \eqref{normlower} holds with large probability for
the multivariate Hawkes processes.

For Theorem~\ref{th:oracle} to be of interest, the condition on the
martingale, condition \eqref{Coeffcontrol}, needs to hold with large
probability as well. From this control, we deduce convenient
data-driven $\ell_1$-weights that are the key parameters of our
estimate. Note that our estimation procedure does not depend on the
value of $c$ in \eqref{normlower}. So knowing the latter is not
necessary for implementing our procedure. Therefore, one of the main
contributions of the paper
is to provide new sharp concentration inequalities that are satisfied by
multivariate point processes. This is the main goal of
Theorem~\ref{th:weakBer} in Section~\ref{sec:Bernstein} where we
establish Bernstein type inequalities for martingales. We apply it to
the control of \eqref{Coeffcontrol}. This allows us to derive the
following result, which specifies the choice of the $d_\p$'s needed to
obtain the oracle inequality with large probability.

%th2 #&#
\begin{Th}
\label{th:oraclewithd}
Let $N=(N^{(m)})_{m=1,\ldots,M}$ be a multivariate counting process with
predictable intensities $\lambda^{(m)}_t$ and almost surely finite
corresponding compensator $\Lambda^{(m)}_t$. %We assume that the $M$ sets
%of points corresponding to each $N^{(m)}$ are almost surely disjoint sets.
Define
\[
\Omega_{V,B}= \Bigl\{ \mbox{for any } \p\in\Phi, \sup
_{t\in
[0,T],m}\bigl\llvert \psi^{(m)}_t(\p) \bigr
\rrvert \leq B_\p\mbox{ and } \bigl(\psi(\p)\bigr)^2
\bullet N_T \leq V_\p \Bigr\},
\]
for positive deterministic constants $B_\p$ and $V_\p$ and
\[
\Omega_c= \{ G\succeq cI \}.
\]
Let $x$ and $\e$ be strictly positive constants and define for all $\p
\in\Phi$,
%
%e2.8 #&#
\begin{equation}
\label{d} d_\p=\sqrt{2(1+\e)\hat V_\p^\mu
x} + \frac{B_\p x}{3},
\end{equation}
with
\[
\hat{V}_\p^\mu= \frac{\mu}{\mu-\phi(\mu)} \bigl(\psi(\p)
\bigr)^2\bullet N_T + \frac{B_\p^2x}{\mu-\phi(\mu)}
\]
for a real number $\mu$ such that $\mu>\phi(\mu)$, where $\phi(u)=\exp(u)-u-1$.
Let us consider the Lasso estimator $\hat f$ of $f_0$ defined in
Section~\ref{sec:Lasso}.
Then, with probability larger than
\[
1-4\sum_{\p\in\Phi} \biggl(\frac{\log (1+\afrac{\mu V_\p}{B_\p^2
x} )}{\log(1+\e)} +1
\biggr)\mathrm{e}^{-x}-\P\bigl(\Omega_{V,B}^c\bigr)- \P
\bigl(\Omega_c^c\bigr),
\]
inequality (\ref{oracle}) is satisfied, that is,
\[
\bigl\llVert \psi(\hat{f})-\lambda \bigr\rrVert _{\proc}^2
\leq C\inf_{a\in\R^{\Phi}} \biggl\{\bigl\llVert \lambda-
\psi(f_a) \bigr\rrVert _{\proc}^2 +
c^{-1}\sum_{\p\in S(a)}d_\p
^2 \biggr\}.
\]
If moreover $\lambda=\psi(f_0)$, then
\[
\llVert \hat{f}-f_0 \rrVert _{T}^2\leq C
\inf_{a\in\R^{\Phi}} \biggl\{\llVert f_0-f_a
\rrVert _{T}^2 + c^{-1}\sum
_{\p\in S(a)}d_\p^2 \biggr\},
\]
where $C$ is a constant independent of $c$, $\Phi$, $T$ and $M$.
\end{Th}

The first oracle inequality gives a control of the difference between
the true intensity and $\psi(\hat{f})$. The equality $\lambda=\psi
(f_0)$ is not required and we can apply this result, for instance, with
$\lambda=(\psi(f_0))_+$.
%the multivariate counting process to have an intensity of the precise
%linear shape $\psi(f)$. Hence the first oracle inequality basically
%says that one can approximate at best any intensity by linear
%predictable processes built on a finite dictionary. This is of
%importance for instance for inhibitory cases in Hawkes processes where
%the relation $\lambda=\psi(f_0)$ does not hold on the whole
%probability space since $\lambda=(\psi(f_0))_+$ (see also Section~5).}

Of course, the smaller the $d_\p$'s the better the oracle inequality.
Therefore, when $x$ increases, the probability bound and the $d_\p$'s
increase and we have to realize a compromise to obtain a
meaningful oracle inequality on an event with large probability. The
choice of $x$ is deeply discussed below, in Sections~\ref{sec:oraclegeneral} and \ref{sec:hawkes} for theoretical purposes and
in Section~\ref{sec:simu} for practical purposes.

Let us first discuss more deeply the definition of $d_\p$ (derived from
subsequent Theorem~\ref{th:weakBer}) which seems intricate. Up to a
constant depending on the choice of $\mu$ and $\e$, $d_\p$ is of same
order as $\max ( \sqrt{x(\psi(\p))^2\bullet N_T },B_\p x )$. To
give more insight on the values of $d_\p$, let us consider the very
special case where for any $m\in\{1,\ldots,M\}$ for any $s$,
$\psi_s^{(m)}(\p)=c_m1_{\{s\in A_m\}}$, where $c_m$ is a positive
constant and $A_m$ a compact set included into $[0,T]$. In this case,
by naturally choosing $B_\p=\max_{1\leq m\leq M} c_m$, we have:
\[
\sqrt{x\bigl(\psi(\p)\bigr)^2\bullet N_T }\geq
B_\p x\quad \iff\quad \sum_{m=1}^Mc_m^2N^{(m)}_{A_m}
\geq x\max_{1\leq m\leq M} c_m^2,
\]
where $N^{(m)}_{A_m}$ represents the number of points of $N^{(m)}$
falling in $A_m$. For more general vector functions $\psi(\p)$, the
term $\sqrt{x(\psi(\p))^2\bullet N_T }$ will dominate $B_\p x$ if the
number of points of the process lying where $\psi(\p)$ is large,
%lying in the support of \red{the main part} of functions of $\psi(\p)$
is significant. In this case, the leading
term in $d_\p$ is expected to be the quadratic term
$\sqrt{2(1+\e)\frac{\mu}{\mu-\phi(\mu)}x(\psi(\p))^2\bullet N_T }$ and
the linear
terms in $x$ can be viewed as residual terms. Furthermore, note that
when $\mu$ tends to $0$,
\[
\frac{\mu}{\mu-\phi(\mu)}=1+\frac{\mu}{2}+\mathrm{o}(\mu), \qquad \frac{x}{\mu-\phi(\mu
)}\sim
\frac{x}{\mu}\to+\infty
\]
since $x>0$. So, if $\mu$ and $\e$ tend to $0$, the quadratic term
tends to $\sqrt{2 x(\psi(\p))^2\bullet N_T }$, but the price to pay is
the explosion of the linear term in $x$. In any case, it is possible
to make the quadratic term as close to $\sqrt{2 x(\psi(\p))^2\bullet
N_T }$ as desired. Basically, this term cannot be improved (see
comments after Theorem~\ref{th:weakBer} for probabilistic arguments).

Let us now discuss the choice of $x$. In more classical contexts such
as density estimation based on an $n$-sample, the choice $x=\gamma\log
(n)$ plugged in the parameters analog to the $d_\p$'s is convenient,
since it both ensures a small probability bound and a meaningful order
of magnitude for the oracle bound (see \cite{bertin11:_adapt_dantiz}
for instance).
See also Sections~\ref{sec:oraclegeneral} and \ref{sec:hawkes} for similar evaluations in our setup. But it
has also been further established that the choice $\gamma=1$ is the
best. Indeed if the components of $d$ are chosen smaller than the
analog of $\sqrt{2 x(\psi(\p))^2\bullet N_T }$ in the density
framework, then the resulting estimator is definitely bad from the
theoretical point of view, but simulations also show that, to some
extent, if the components of $d$ are larger than the analog of $\sqrt{2
x(\psi(\p))^2\bullet N_T }$, then the estimator deteriorates too. A
similar result is out of reach in our setting, but similar conclusions
may remain valid here since density estimation often provides some
clues about what happens for more intricate heteroscedastic models. In
particular, the main heuristic justifying the optimality of this tuning
result in the density setting is that the quadratic term (and in
particular the constant $\sqrt{2}\,$) corresponds to the rate of the
central limit theorem and in this sense, it provides the ``best
approximation'' for the fluctuations. %\red{In our more involved setup,
%we therefore aim at providing a quadratic term as close as possible to
%this one (see Section~3 for more details).}} %, which means at the end
%that we have some chance to have sharp values for the components of $d_
For further discussion, see the simulation study in Section~\ref{sec:simu}.

Finally, it remains to control $\P(\Omega_{V,B})$ and $\P(\Omega_c)$.
These are the goals of Section~\ref{sec:oraclegeneral} for Poisson and
Aalen models and Section~\ref{sec:hawkes} for multivariate Hawkes processes.

%%%%%%%%%%%%%%%%%%%%%%%%%%%%%%%
%%%%%%%%%%%%%%%%%%%%%%%%%%%%%%%
%s3 #&#
\section{Bernstein type inequalities for multivariate point
processes}\label{sec:Bernstein}
We establish a Bernstein type concentration inequality based on
boundedness assumptions. This result, which has an interest per se from
the probabilistic point of view, is the key result to derive the
convenient values for the vector $d$ in Theorem~\ref{th:oraclewithd}
and so is capital from the statistical perspective.

%th3 #&#
\begin{Th}
\label{th:weakBer}
Let $N=(N^{(m)})_{m=1,\ldots,M}$ be a multivariate counting process with
predictable intensities $\lambda^{(m)}_t$ and corresponding compensator
$\Lambda^{(m)}_t$ with respect to some given filtration.
Let $B>0$. Let $H=(H^{(m)})_{m=1,\ldots,M}$ be a multivariate predictable
process such that for all $\xi\in(0,3)$, for all~$t$,
%
%e3.1 #&#
\begin{equation}
\label{expcond} \exp(\xi H/B)\bullet\Lambda_t<\infty\qquad  \mbox{a.s.}\quad  \mbox{and}\quad
\exp\bigl(\xi H^2/B^2\bigr)\bullet\Lambda_t<
\infty \qquad \mbox{a.s.}
\end{equation}
Let us consider the martingale defined for all $t\geq0$ by
\[
M_t=H\bullet(N-\Lambda)_t.
\]
Let $v>w$ and $x$ be positive constants and let $\tau$ be a bounded
stopping time.
Let us define
\[
\hat{V}^\mu= \frac{\mu}{\mu-\phi(\mu)}H^2\bullet
N_\tau+ \frac{B^2x}{\mu
-\phi(\mu)}
\]
for a real number $\mu\in(0,3)$ such that $\mu>\phi(\mu)$, where $\phi
(u)=\exp(u)-u-1$.
Then, for any $\e>0$,
%
%e3.2 #&#
\begin{eqnarray}
\label{bracketT} &&\P \biggl(M_\tau\geq\sqrt{2 (1+\e)
\hat{V}^\mu x} + \frac{Bx}{3} \mbox{ and } w\leq
\hat{V}^\mu\leq v \mbox{ and } \sup_{m,t\leq\tau} \bigl
\llvert H^{(m)}_t \bigr\rrvert \leq B \biggr)\nonumber
\\[-8pt]\\[-8pt]
&&\quad \leq2 \biggl(\frac{\log(v/w)}{\log(1+\e)}+1 \biggr) \mathrm{e}^{-x}.\nonumber
\end{eqnarray}
\end{Th}

This result is based on the exponential martingale for counting
processes, which has been used for a long time in the context of the
counting process theory. See, for instance, \cite{Bre,SW} or
\cite{vdG}. This basically gives a concentration inequality taking the
following form (the result is stated here in its univariate form for
comparison purposes): for any $x>0$,
%
%e3.3 #&#
\begin{equation}
\label{sho2} \P \biggl(M_\tau\geq\sqrt{2\rho x} + \frac{Bx}{3}
\mbox{ and } \int_0^\tau H_s^2
\lambda(s) \,\mathrm{d}s\leq\rho \mbox{ and } \sup_{s\in
[0,\tau]} \llvert
H_s \rrvert \leq B \biggr) \leq \mathrm{e}^{-x}.
\end{equation}
In (\ref{sho2}), $\rho$ is a deterministic upper bound of $v=\int_0^\tau H_s^2 \lambda(s) \,\mathrm{d}s$, the bracket of the martingale,
and consequently the martingale equivalent of the variance term.
Moreover, $B$ is a deterministic upper bound of $\sup_{s\in[0,\tau]} \llvert  H_s \rrvert $.
The leading term for moderate values of $x$ and $\tau$ large enough is
consequently $\sqrt{2\rho x}$.
The central limit theorem for martingales states that, under some
assumptions, a sequence of martingales $(M_n)_n$ with respective
brackets $(v_n)_n$ tending to a deterministic value $\bar{v}$, once
correctly normalized, tends to a Gaussian process with bracket $\bar
{v}$. Therefore, a term of the form $\sqrt{2 \bar{v} x}$ in the upper
bound is not improvable, in particular the constant $\sqrt{2}$. However
the replacement of the limit $\bar{v}$ by a deterministic upper bound
$\rho$ constitutes a loss.
% where the constant $\sqrt{2}$ is not improvable since this coincides
%with the rate of the central limit theorem for martingales.
%sequence of martingales $(M_n)_n$ with corresponding brackets $v_n$,
%it is assumed that $v_n$ tends to a deterministic value, which is the
%asymptotic variance (once everything is correctly renormalized). So if
%the term $\sqrt{2}$ is not improvable, it is likely that a fixed
%deterministic value which deterministically upper bounds $v$
%constitutes a loss.}
In this sense, Theorem~\ref{th:weakBer} improves the bound and
consists of plugging in the unbiased estimate $\hat{v}=\int_0^\tau
H_s^2 \,\mathrm{d}N_s$ instead of a nonsharp deterministic upper bound of $v$.
Note that we are not able to obtain exactly the term $\sqrt{2}$ but any
value strictly larger than $\sqrt{2}$, as close as we want to $\sqrt {2}$ up to some additive terms depending on $B$ that are negligible for
moderate values of $x$.

The proof is based on a peeling argument that was first introduced in
\cite{LS} for Gaussian processes and is given in Section~\ref{sec:weakBer}.
%Note that $\hat{v}$ grows more or less linearly with $\tau$ while $B$
%is fixed.
%Hence, for large $\tau$, in the lower bound of $M_\tau$ in (
%Furthermore, note that when $\mu$ tends to $0$,
%$$\frac{\mu}{\mu-\phi(\mu)}=1+\frac{\mu}{2}+o(\mu), \frac{x}{\mu-\phi(
%since $x>0$. So, if $\mu$ and $\e$ tend to $0$, the quadratic term
%tends to $\sqrt{2\hat{v} x}$ but the price to pay is the explosion of
%the linear term in $x$. In any case it is possible to make the
%quadratic term as close to $\sqrt{2\hat{v} x}$ as desired.

Note that there exist also inequalities that seem nicer than \eqref
{sho2} which constitutes the basic brick for our purpose. For instance,
in \cite{DvZ}, it is established that for any deterministic positive
real number $\theta$, for any $x>0$,
%
%e3.4 #&#
\begin{equation}
\label{Dvz} \P \biggl(M_\tau\geq\sqrt{2\theta x} \mbox{ and } \int
_0^\tau H_s^2
\,\mathrm{d}\Lambda_s + \int_0^\tau
H_s^2 \,\mathrm{d}N_s \leq\theta \biggr)
\leq \mathrm{e}^{-x}.
\end{equation}
At first sight, this seems better than Theorem~\ref{th:weakBer} because
no linear term depending on $B$ appears, but if we wish to use the
estimate $2 \int_0^\tau H_s^2 \,\mathrm{d}N_s$ instead of $\theta$ in the
inequality, we have to bound $\llvert  H_s \rrvert $ by some $B$ in any case. Moreover,
by doing so, the quadratic term will be of order $\sqrt{4 \hat v x}$
which is worse than the term $\sqrt{2 \hat v x}$ derived in Theorem~\ref
{th:weakBer}, even if this constant $\sqrt{2}$ can only be reached
asymptotically in our case.

There exists a better result if the martingale $M_t$ is for instance
conditionally symmetric (see \cite{penaa,BT,DvZ}): for any $x>0$,
%
%e3.5 #&#
\begin{equation}
\label{Dvz2} \P \biggl(M_\tau\geq\sqrt{2\kappa x} \mbox{ and } \int
_0^\tau H_s^2
\,\mathrm{d}N_s \leq\kappa \biggr) \leq \mathrm{e}^{-x},
\end{equation}
which seems close to the ideal inequality. But there are actually two
major problems with this inequality. First, one needs to assume that
the martingale is conditionally symmetric, which cannot be the case in
our situation for general counting processes and general dictionaries.
Second, it depends on the deterministic upper bound $\kappa$ instead of
$\hat{v}$. To replace $\kappa$ by $\hat{v}$ and then to apply peeling
arguments as in the proof of Theorem~\ref{th:weakBer}, we need to
assume the existence of a positive constant $w$ such that $\hat{v}\geq
w$. But if the process is empty, then $\hat{v}=0$, so we cannot
generally find such a positive lower bound, whereas in our theorem, we
can always take $w=\frac{B^2x}{\mu-\phi(\mu)}$ as a lower bound for
$\hat{V}^\mu$.

Finally, note that in Proposition~\ref{weakberLambda} (see
Section~\ref{sec:weakBer}), we also derive a similar bound where $\hat
V^\mu$ is replaced by $\int_0^\tau H_s^2
\,\mathrm{d}\Lambda_s$. Basically, it means that the same type of
results hold for the quadratic characteristic instead of the quadratic
variation. Though this result is of little use here, since the quadratic
characteristic is not observable, we think that it may be of interest
for readers investigating self-normalized results as in \cite{Pena}.

%Indeed, as seen in the previous section, we can write
%$$b_\p-\bar b_\p=\psi(\p)\bullet(N-\Lambda)_T$$
%and Theorem~\ref{th:weakBer} with $H=\psi(\p)$ applies. So, Inequality
%(\ref{bracketT}) with convenient values of $x$, ensures that %
%%%%%%%%%%%%%%%%%%%%%%%%%%%%%%%%%%%%%%%%%%%
%%%%%%%%%%%%%%%%%%%%%%%%%%%%%%%%%%%%%%%%%%%
%s4 #&#
\section{Applications to the Poisson and Aalen
models}\label{sec:oraclegeneral}

We now apply Theorem~\ref{th:oraclewithd} to the Poisson and Aalen
models. The case of the multivariate Hawkes process, which is much more
intricate, will be the subject of the next section.
%%%%%%%%%%%%%%%%%%%%%%%%%%%%%%%%%%%%%
%s4.1 #&#
\subsection{The Poisson model}
Let us recall that in this case, we observe $M$ i.i.d. Poisson
processes with intensity $f_0$ supported by $[0,1]$ (with $M\geq2$)
and that the norm is given by $\llVert  f  \rrVert ^2=\int_0^1 f^2(x) \,\mathrm{d}x$. We
assume that $\Phi$ is an orthonormal system for $\llVert  \cdot  \rrVert $. In this case,
\[
\llVert \cdot \rrVert _{1}^2=M\llVert \cdot \rrVert
^2 \quad \mbox{and}\quad  G=M I,
\]
where $I$ is the identity matrix. One applies Theorem~\ref
{th:oraclewithd} with $c=M$ (so $\P(\Omega_c^c)=0$) and
\[
B_\p=\llVert \p \rrVert _\infty,\qquad  V_\p=\llVert
\p \rrVert _\infty^2 (1+\delta)Mm_1,
\]
for $\delta>0$ and $m_1=\int_0^1f_0(t)\,\mathrm{d}t$. Note that here
$T=1$ and therefore $N^{(m)}_T=N^{(m)}_1$ is the total number of observed
points for the $m$th process. Using
\[
\psi(\p)^2\bullet N_1\leq\llVert \p \rrVert
_\infty^2 \sum_{m=1}^M
N^{(m)}_1
\]
and since the distribution of $\sum_{m=1}^M N^{(m)}_1$ is the Poisson
distribution with parameter $Mm_1$, Cramer--Chernov arguments give:
\[
\P\bigl(\Omega_{V,B}^c\bigr)\leq\P \Biggl(\sum
_{m=1}^M N^{(m)}_1> (1+
\delta)Mm_1 \Biggr)\leq\exp \bigl(-\bigl\{(1+\delta)\ln(1+\delta)-
\delta\bigr\}Mm_1 \bigr).
\]
For $\alpha>0$, by choosing $x=\alpha\log(M)$, we finally obtain the
following corollary derived from Theorem~\ref{th:oraclewithd}.

%co1 #&#
\begin{Cor}\label{Cor:oracle:poisson}
With probability larger than $1-C_1\frac{\llvert  \Phi \rrvert \log(M)}{M^\alpha}-\mathrm{e}^{-C
_2M}$, where $C_1$ is a constant depending on $\mu$, $ \e$, $\alpha$,
$\delta$ and $m_1$ and $C_2$ is a constant depending on $\delta$ and
$m_1$, we have:
\begin{eqnarray*}
&&\llVert \hat{f}-f_0 \rrVert ^2\leq C \inf
_{a\in\R^{\Phi}} \Biggl\{\llVert f_0-f_a\rrVert
^2 \\
&&\hphantom{\llVert \hat{f}-f_0 \rrVert ^2\leq C \inf
_{a\in\R^{\Phi}} \Biggl\{}{}+ \frac{1}{M^2}\sum_{\p\in S(a)}
\Biggl(\log(M)\sum_{m=1}^M \int
_0^1 \p^2(x) \,\mathrm{d}N^{(m)}_x+
\log^2(M) \llVert \p \rrVert _\infty ^2 \Biggr)
\Biggr\},
\end{eqnarray*}
where $C$ is a constant depending on $\mu$, $ \e$, $\alpha$, $\delta$
and $m_1$.
\end{Cor}

To shed some lights on this result, consider an asymptotic perspective
by assuming that $M$ is large. Assume also, for sake of simplicity,
that $f_0$ is bounded from below on $[0,1]$. If the dictionary $\Phi$
(whose size may depend on $M$) satisfies
\[
\max_{\p\in\Phi}\llVert \p \rrVert _\infty=\mathrm{o} \biggl(\sqrt{
\frac{M}{\log M}} \biggr),
\]
then, since, almost surely,
\[
\frac{1}{M}\sum_{m=1}^M \int
_0^1 \p^2(x) \,\mathrm{d}N^{(m)}_x
\stackrel{M\to \infty} {\longrightarrow} \int_0^1
\p^2(x) f_0(x) \,\mathrm{d}x,
\]
almost surely,
\begin{eqnarray*}
&& \frac{1}{M^2}\sum_{\p\in S(a)} \Biggl(\log(M)\sum
_{m=1}^M \int_0^1
\p^2(x) \,\mathrm{d}N^{(m)}_x+
\log^2(M) \llVert \p \rrVert _\infty^2 \Biggr)
\\
& &\quad  =\frac{\log M}{M}\sum_{\p\in S(a)}\int
_0^1 \p^2(x) f_0(x)\,
\mathrm{d}x\times\bigl(1+\mathrm{o}(1)\bigr).
\end{eqnarray*}
The right-hand term corresponds, up to the logarithmic term, to the sum
of variance terms when estimating $\int_0^1 \p(x)f_0(x)\,\mathrm{d}x$
with $\frac{1}{M}\sum_{m=1}^M\int_0^1 \p(x) \,\mathrm{d}N^{(m)}_x$ for $\p\in
S(a)$. This means that the estimator adaptively
achieves the best trade-off between a bias term and a variance term.
The logarithmic term is the price to pay for adaptation. Furthermore,
when $M\to+\infty$, the inequality of Corollary~\ref
{Cor:oracle:poisson} holds with probability that goes to 1 at a
polynomial rate. We refer the reader to \cite{RBR} for a deep
discussion on optimality of such results.
%%%%%%%%%%%%%%%%%%%%%%%%%%%
%s4.2 #&#
\subsection{The Aalen model}\label{sec:aalenvrai}
Results similar to those presented in this paragraph can be found in
\cite{GG} under restricted eigenvalues conditions instead of (\ref
{normlower}). % but this \red{also} expresses some orthogonality
%properties of columns of $G$, that are nonmild conditions as well. %
%?}%Results of \cite{GG} are all the more similar to ours since they
%rely on arguments drawn from this paper (see Section~6 and
%Acknowledgements in \cite{GG}). The main difference lies in leading
%constants of their oracle inequality derived from recent results on
%Lasso estimates obtained by \cite{KLT}.
Recall that we observe an $M$-sample $(X^{(m)}, Y^{(m)},
N^{(m)})_{m=1,\ldots,M}$, with $Y^{(m)}=(Y_t^{(m)})_{t\in[0,1]}$ and
$N^{(m)}=(N_t^{(m)})_{t\in[0,1]}$ (with $M\geq2$). We assume that $X^{(m)}
\in[0,1]$ and that the intensity of $N_t^{(m)}$ is $f_0(t,X^{(m)})Y^{(m)}_t$.
We set for any $f$,
\[
\llVert f \rrVert _e^2:=\E \biggl(\int
_0^1 f^2\bigl(t,X^{(1)}
\bigr) \bigl(Y^{(1)}_t \bigr)^2 \,\mathrm {d}t
\biggr).
\]
%
%that $\left\Vert \cdot  \right\Vert_e$
%is a true norm. For instance if there are no covariates, it is
%equivalent to assuming that $\E\Bigl(\bigl(Y^{(1)}_t\bigr)^2\Bigr)\not=0$
%on $[0,1]$ i.e. $Y^{(1)}_t$ cannot be zero almost surely and this for all
%$t$ in $[0,1]$. This is natural since of course one cannot estimate
%$f_0(t)$ if $Y^{(1)}_t=0$ almost surely. If $Y^{(1)}_t$ is deterministic and
%nonzero on $[0,1]$ then we are in the case of a Cox process ($N^{(m)}$ is
%a Poisson process given the covariates $X^{(m)}$) and it is natural to say
%that we will be able to measure $f_0$ only on the support of the
%variables $X^{(m)}$.}
%
We assume that $\Phi$ is an orthonormal system for $\llVert  \cdot  \rrVert _2$,
the classical norm on ${\mathbb L}_2([0,1]^2)$, and we assume that
there exists a positive constant $r$ such that
%
%e4.1 #&#
\begin{equation}
\label{condr} \forall f\in{\mathbb L}_2\bigl([0,1]^2
\bigr), \qquad \llVert f \rrVert _e\geq r \llVert f \rrVert _2,
\end{equation}
so that $\llVert  \cdot  \rrVert _e$
is a norm. If we assume, for instance, that the density of the $X^{(m)}$'s
is lower bounded by a positive constant $c_0$ and there exists $c_1>0$
such that for any $t$,
\[
\E\bigl[\bigl(Y_t^{(1)}\bigr)^2|X^{(1)}
\bigr]\geq c_1
\]
then (\ref{condr}) holds with $r^2=c_0 c_1$. The empirical version of
$\llVert  f  \rrVert _e$, denoted $\llVert  f  \rrVert _{\mathrm{emp}}$, is defined by
\[
\llVert f \rrVert _{\mathrm{emp}}^2:=\frac{1}{M}\llVert f
\rrVert _{T}^2=\frac{1}{M}\sum
_{m=1}^M\int_0^1
f^2\bigl(t,X^{(m)}\bigr) \bigl(Y^{(m)}_t
\bigr)^2 \,\mathrm{d}t.
\]
Unlike the Poisson model, since the intensity depends on covariates $X^{(m)}
$'s and variables $Y^{(m)}$'s, the control of $\P(\Omega_c^c)$ is much
more cumbersome for the Aalen case, even if it is less intricate than
for Hawkes processes (see Section~\ref{sec:hawkes}).
%To avoid another set of tedious computations, we just give here a
%brief sketch of what one could do.
%%Then, as we shall do in Section~\ref{sec:hawkes} for the Hawkes
%process, t
%To control $\Omega_c$, we only need to concentrate the elements of $G$
%around their mean since they are sum of i.i.d. variables and use the
%fact that $\E(G)\succeq Mr^2 I$. Then the probability of $\Omega_c^c$
%can be proved to be smaller than $\frac{\left| \Phi \right|^2}{M^\alpha}$ up to a
%constant if one chooses $c=M r^2 (1-\delta)$ and if one assumes that $\left| \Phi \right|=o(\sqrt{T}\log(T)^{-\beta})$ (where of course $\alpha$, $\beta$
%and $\delta$ are convenient positive constants).
We have the following result proved in Section~\ref{2proofs}.

%pr1 #&#
\begin{Prop}\label{Prop:Aalen}
We assume that (\ref{condr}) is satisfied, the density of the
covariates $X^{(m)}$ is bounded by $D$ and
%
%e4.2 #&#
\begin{equation}
\label{condY} \sup_{t\in[0,1]}\max_{m\in\{1,\ldots,M\}}Y^{(m)}_t
\leq1 \qquad \mbox{almost surely}.
\end{equation}
We consider an orthonormal dictionary $\Phi$ of functions of ${\mathbb
L}_2([0,1]^2)$ that may depend on $M$, and we let $r_\Phi$ denote the
spectral radius of the matrix $\mathfrak{H}$ whose components are
$\mathfrak{H}_{\p,\p'}=\int\!\int\llvert  \p(t,x) \rrvert  \llvert  \p'(t,x) \rrvert \,\mathrm{d}t\,\mathrm{d}x$.
Then, if
%
%e4.3 #&#
\begin{equation}
\label{cond:Aalen} \max_{\p\in\Phi}\llVert \p \rrVert
_\infty^2\times r_\Phi\llvert \Phi \rrvert \times
\frac{\log
M}{M}\to0,
\end{equation}
when $M\to+\infty$ then, for any $\beta>0$, there exists $C_1>0$
depending on $\beta$, $D$ and $f_0$ such that with $c=C_1M$,
we have
\[
\P\bigl(\Omega_c^c\bigr)=\mathrm{O}\bigl(\llvert \Phi \rrvert
^2M^{-\beta}\bigr).
\]
%
%La c'est sale car en vrai c'est $O(\left| \Phi \right|^2M^{-\alpha})$. Mais sous (
%remplacer $\alpha$ par $\alpha-2$. Contrairement a Hawkes ca se voit
%dans la preuve
\end{Prop}

Assumption (\ref{condY}) is usually satisfied in most of the practical
examples where Aalen models are involved. See \cite{andersen} for
explicit examples and see, for instance, \cite{greg,RB} for other
articles where this assumption is made. In the sequel, we also assume that
there exists a positive constant $R$ such that
%
%e4.4 #&#
\begin{equation}
\label{CondN} \max_{m\in\{1,\ldots,M\}}N^{(m)}_1\leq R\qquad
\mbox{a.s.}
\end{equation}
This assumption, considered by \cite{RB}, is obviously satisfied in
survival analysis where there is at most one death per individual.
%(if for instance, $N^{(m)}_1$ is the total number of times the subject $m$
%contracts a specific disease in his life). As (\ref{condY}),
It could have been relaxed in our setting, by considering exponential
moments assumptions, to include Markov cases for instance. However,
this much simpler assumption allows us to avoid tedious and unnecessary
technical aspects since we only wish to illustrate our results in a
simple framework. Under (\ref{condY}) and (\ref{CondN}),
almost surely,
\[
\psi(\p)^2\bullet N_T=\sum_{m=1}^M
\int_0^1 \bigl[Y^{(m)}_t
\bigr]^2\p^2\bigl(t,X^{(m)}\bigr) \,\mathrm
{d}N^{(m)}_t\leq\sum_{m=1}^M
\int_0^1 \p^2\bigl(t,X^{(m)}
\bigr) \,\mathrm{d}N^{(m)}_t\leq M R \llVert \p \rrVert
_\infty^2.
\]
So, we apply Theorem~\ref{th:oraclewithd} with $B_\p=\llVert  \p  \rrVert _\infty$,
$V_\p=M R \llVert  \p  \rrVert _\infty^2$ (so $\P(\Omega_{V,B})=1$) and $x=\alpha
\log(M)$ for $\alpha>0$. We finally obtain the following corollary.

%co2 #&#
\begin{Cor}
Assume that (\ref{condY}) and (\ref{CondN}) are satisfied. With
probability larger than $1-C_1\frac{\llvert  \Phi \rrvert \log(M)}{M^\alpha}-\P(\Omega_c^c)$,
where $C_1$ is a constant depending on $\mu$, $ \e$, $\alpha$ and $R$,
we have:
\begin{eqnarray*}
\llVert \hat{f}-f_0 \rrVert _{\mathrm{emp}}^2&\leq& C
\inf_{a\in
\R^{\Phi}} \Biggl\{\llVert f_0-f_a
\rrVert _{\mathrm{emp}}^2
\\
&&\hphantom{C\inf_{a\in
\R^{\Phi}} \Biggl\{}\!{} +\frac{1}{M^2}\sum_{\p\in S(a)} \Biggl(\log(M)
\sum_{m=1}^M\int_0^1
\p^2\bigl(t,X^{(m)}\bigr) \,\mathrm{d}N^{(m)}_t+
\log^2(M)\llVert \p \rrVert _\infty^2 \Biggr)
\Biggr\},
\end{eqnarray*}
where $C$ is a constant depending on $\mu$, $ \e$, $\alpha$ and $R$.
\end{Cor}

To shed lights on this result, assume that the density of the $X^{(m)}$'s
is upper bounded by a constant $\tilde{R}$.
In an asymptotic perspective with $M\to\infty$,
we have almost surely,
\[
\frac{1}{M} \sum_{m=1}^M\int
_0^1 \p^2\bigl(t,X^{(m)}
\bigr) \,\mathrm{d}N^{(m)}_t\to\E \biggl(\int
_0^1 \p^2\bigl(t,X^{(1)}
\bigr) f_0\bigl(t,X^{(1)}\bigr) Y^{(1)} \,\mathrm{d}t
\biggr).
\]
But\vspace*{-1pt}
\[
\E \biggl(\int_0^1 \p^2
\bigl(t,X^{(1)}\bigr) f_0\bigl(t,X^{(1)}\bigr)
Y^{(1)} \,\mathrm{d}t \biggr)\leq\llVert f_0 \rrVert
_\infty\E \biggl( \int_0^1
\p^2\bigl(t,X^{(1)}\bigr) \,\mathrm{d}t \biggr) \leq \tilde{R}
\llVert f_0 \rrVert _\infty.
\]
So, if the dictionary $\Phi$ satisfies\vspace*{-1pt}
\[
\max_{\p\in\Phi}\llVert \p \rrVert _\infty=\mathrm{O} \biggl(\sqrt{
\frac{M}{\log M}} \biggr),\vspace*{-1pt}
\]
which is true under (\ref{cond:Aalen}) if $r_\Phi\llvert  \Phi \rrvert \geq1$,
then, almost surely, the variance term is asymptotically smaller than
$\log(M)\frac{\llvert  S(a) \rrvert \llVert  f_0  \rrVert _\infty}{M}$ up to constants. So, we can
draw the same conclusions as for the Poisson model. We have not
discussed here the choice of $\Phi$ and Condition (\ref{cond:Aalen}).
This will be extensively done in Section~\ref{LassoH} where we deal
with a similar condition but in a more involved setting.
%If we assume that $\left\Vert \p  \right\Vert_\infty^2\leq M/\log(M)$, then up to
%asymptotically negligible random fluctuations when $M\to\infty$, we
%obtain, as for the Poisson model,
%$$\norm{\hat{f}-f_0}_{\mathrm{emp}}^2\leq C\inf_{a\in\R^{\Phi}}\left\{\left
%%%%%%%%%%%%%%%%%%%%%%%%%%%%%%%
%%%%%%%%%%%%%%%%%%%%%%%%%%%%%%%
%s5 #&#
\section{Applications to the case of multivariate Hawkes process}\label
{sec:hawkes}
%%%%%%%%%%%%%%%%%%%%%%%%%%%%%%%
For a multivariate Hawkes model, the parameter $f_0=(\nu^{(m)}, (h_\ell^{(m)}
)_{\ell=1,\ldots,M})_{m=1,\ldots,M}$ belongs to\vspace*{-1pt}
\[
\mathcal{H}=\mathbb{H}^M= \Biggl\{f=\bigl(\mathbf{f}^{(m)}
\bigr)_{m=1,\ldots,M}\Bigm|\mathbf {f}^{(m)}\in\mathbb{H} \mbox{ and }
\llVert f \rrVert ^2=\sum_{m=1}^M
\bigl\llVert \mathbf {f}^{(m)} \bigr\rrVert ^2 \Biggr\}\vspace*{-1pt}
\]
where\vspace*{-1pt}
\begin{eqnarray*}
\mathbb{H}&=& \Biggl\{\mathbf{f}=\bigl(\mu, (g_\ell)_{\ell=1,\ldots,M}
\bigr)\Bigm| \mu\in \R , g_\ell\mbox{ with support in } [0,1]
\\[-1pt]
&&\hphantom{\Biggl\{}\mbox{and }\llVert \mathbf{f} \rrVert ^2= \mu^2 + \sum
_{\ell=1}^M \int_0^1
g_\ell^2(x) \,\mathrm{d}x <\infty \Biggr\}.
\end{eqnarray*}
If one defines the linear predictable transformation $\kappa$ of
$\mathbb{H}$ by
%
%e5.1 #&#
\begin{equation}
\label{kappa} \kappa_t(\mathbf{f})=\mu+ \sum
_{\ell=1}^M \int_{t-1}^{t-}
g_\ell (t-u)\,\mathrm{d}N^{(\ell)}_u,
\end{equation}
then the transformation $\psi$ on $\mathcal{H}$ is given by
\[
\psi^{(m)}_t(f)=\kappa_t\bigl(
\mathbf{f}^{(m)}\bigr).
\]
The first oracle inequality of Theorem~\ref{th:oraclewithd} provides
theoretical guaranties of our Lasso methodology in full generality and
in particular, even if inhibition takes place (see Section~\ref{sec:IntroHawkes}). % works by giving a linear transformation $\psi(
Since $\Omega_{V,B}$ and $\Omega_c$ are observable events, we know
whether the oracle inequality holds. However we are not able to
determine $\P(\Omega_{V,B})$ and $\P(\Omega_c)$ in the general case.
Therefore, in Sections~\ref{probaHawkes} and \ref{LassoH}, we assume
that all interaction functions are nonnegative and that there exists
$f_0$ in $\mathcal{H}$ so that for any $m$ and any $t$,
\[
\lambda^{(m)}_t=\psi^{(m)}_t(f_0).
\]
We also assume that the process is observed on $[-1,T]$ with $T>1$.
\subsection{Some useful probabilistic results for multivariate Hawkes
processes}\label{probaHawkes}
% Before stating oracle inequalities for Lasso estimates, we need to
%prove some probabilistic results. They will be useful to deal with $\P(
%
In this paragraph, we present some particular exponential results and
tail controls for Hawkes processes. As far as we know, these results
are new: They constitute the generalization of \cite{RBRoy} to the
multivariate case. In this paper, they are used to control
$\P(\Omega_c^c)$ and $\P(\Omega_{V,B}^c)$ but they may be of
independent interest.

Since the functions $h^{(m)}_{\ell}$'s are nonnegative, a cluster
representation exists.
We can indeed construct the Hawkes
process by the Poisson cluster representation (see \cite{DVJ}) as follows:
\begin{itemize}
\item Distribute \emph{ancestral points}
with marks $\ell= 1, \ldots, M$
according to homogeneous Poisson processes with intensities $\nu^{(\ell
)}$ on $\mathbb{R}$.
\item For each ancestral point, form a cluster of descendant
points. More precisely, starting with an ancestral point at time 0 of a
certain type, we successively build new generations as Poisson
processes with intensity $h^{(m)}_{\ell}(\cdot-T)$, where $T$ is the
parent of type $\ell$ (the corresponding children being of type
$m$). We will be in the situation where this process becomes
extinguished and we
denote by $H$ the last children of all generations, which also
represents the length of the cluster. Note that the number of
descendants is a multitype branching process
(and there exists a branching cluster representation (see
\cite{BM,DVJ,hawkes})) with offspring
distributions being Poisson variables with means
\[
\gamma_{\ell, m} = \int_0^1
h^{(m)}_{\ell}(t) \,\mathrm{d} t.
\]
\end{itemize}
%
%The description of the cluster representation above is not complete,
%but
The essential part we need is that the expected number of offsprings of
type $m$ from a point of type $\ell$ is $\gamma_{\ell, m}$. With
$\Gamma= (\gamma_{\ell, m})_{\ell, m = 1, \ldots, M}$, the theory of
multitype branching processes gives that
the clusters are finite almost surely if %and only if
the spectral
radius of $\Gamma$ is strictly smaller than %or equal to
1. In this case, there
is a stationary version of the Hawkes process by the Poisson cluster
representation.

%Below we will need the stronger requirement that
Moreover, if $\Gamma$ has
spectral radius strictly smaller than 1, one can provide a bound on the
number of points in a cluster. We denote by $\P_\ell$ the law of the
cluster whose ancestral point is of type $\ell$, $\E_\ell$ is the
corresponding expectation.

The following lemma is very general and holds even if the function
$h_\ell^{(m)}$ have infinite support as long as the spectral radius $\Gamma
$ is strictly less than 1.

%le1 #&#
\begin{Lemma} \label{lem:expmoment}
If $W$ denotes the total number of points of any type in
the cluster whose ancestral point is of type $\ell$, then if the
spectral radius of $\Gamma$ is strictly
smaller than $1$ there exists $\vartheta_\ell> 0$, only depending on
$\ell$ and on $\Gamma$, such that
\[
\E_{\ell} \bigl(\mathrm{e}^{\vartheta_\ell W}\bigr) < \infty.
\]
\end{Lemma}

This easily leads to the following result, which provides the existence
of the Laplace transform of the total number of points in an arbitrary
bounded interval, when the functions $h_\ell^{(m)}$ have bounded support.
%also modified slightly the proof. An arbitrary interval is maybe more
%useful here.}

%pr2 #&#
\begin{Prop}\label{momentsHawkes} Let $N$ be a stationary multivariate
Hawkes process, with compactly supported nonnegative interactions
functions and such that the spectral radius of $\Gamma$ is strictly
smaller than~$1$. For any $A>0$, let $N_{[-A,0)}$ be the total number of
points of $N$ in $[-A,0)$, all marks included.
Then there exists a constant $\theta> 0$, depending on the
distribution of the process and on~$A$, such that
\[
\mathcal{E}:=\E\bigl(\mathrm{e}^{\theta N_{[-A,0)}}\bigr) < \infty,
\]
which implies that for all positive $u$
\[
\P(N_{[-A,0)}\geq u)\leq\mathcal{E}\mathrm{e}^{-\theta u}.
\]
\end{Prop}

Moreover, one can strengthen the ergodic theorem in a nonasymptotic
way, under the same assumptions.

%pr3 #&#
\begin{Prop}\label{flow}
Under the assumptions of Proposition~\ref{momentsHawkes}, let $A>0$ and
let $Z(N)$ be a function depending on the points of $N$ lying in
$[-A,0)$. % of a stationary multivariate Hawkes process, $N$, with
%parameter $f_0\in\mathcal{H}$.
Assume that there exist $b$ and $\eta$ nonnegative constants such that
\[
\bigl\llvert Z(N) \bigr\rrvert \leq b \bigl(1+N_{[-A,0)}^\eta
\bigr),
\]
where $N_{[-A,0)}$ represents the total number of points of $N$ in
$[-A,0)$, all marks included.
We denote $\mathfrak{S}$ the shift operator, meaning that $Z\circ
\mathfrak{S}_t(N)$ depends now in the same way as $Z(N)$ on some points
that are now the points of $N$ lying in $[t-A,t)$.

We assume $\E[\llvert  Z(N) \rrvert ]<\infty$ and for short, we denote $\E(Z)=\E
[Z(N)]$. Then, for any $\alpha>0$, there exists a constant $\mathcal
{T}(\alpha,\eta, f_0, A)>1$ such that for $T>\mathcal{T}(\alpha,\eta
,f_0,A)$, there exist $C_1$, $C_2$, $C_3$ and $C_4$ positive constants
depending on $\alpha, \eta, A$ and $f_0$ such that
\[
\P \biggl(\int_0^T \bigl[Z\circ
\mathfrak{S}_t(N) - \E(Z)\bigr] \,\mathrm{d}t \geq C_1\sigma
\sqrt{T\log^3(T)} +C_2 b \bigl(\log(T)
\bigr)^{2+\eta} \biggr) \leq\frac
{C_4}{T^\alpha},
\]
with $\sigma^2=\E([Z(N)-\E(Z)]^2\indic_{N_{[-A,0)} \leq\tilde{\mathcal
{N}}})$ and $\tilde{\mathcal{N}}=C_3 \log(T)$.
\end{Prop}

Finally, to deal with the control of $\P(\Omega_c)$, we shall need the
next result. First,
we define a quadratic form $Q$ on $\mathbb{H}$ by
%
%e5.2 #&#
\begin{equation}
\label{Q} Q(\mathbf{f},\mathbf{g}) =\E_{\P} \bigl(
\kappa_1(\mathbf{f})\kappa_1(\mathbf{g}) \bigr)=
\E_{\P
} \biggl(\frac{1}{T} \int_0^T
\kappa_t(\mathbf{f})\kappa_t(\mathbf{g}) \,\mathrm{d}t
\biggr), \qquad \mathbf{f},\mathbf{g}\in\mathbb{H}.
\end{equation}
We have:

%pr4 #&#
\begin{Prop}\label{eqnormH}
Under the assumptions of Proposition~\ref{momentsHawkes}, if the
function parameter $f_0$ satisfies % For a stationary Hawkes process
%with
% intensities given by (\ref{eq:intensity}), which fulfills
%
%e5.3 #&#
\begin{equation}
\label{eq:intbound} \min_{m\in\{1,\ldots,M\}}\nu^{(m)} > 0 \quad \mbox{and}\quad
\max_{l,m\in\{1,\ldots,M\}}\sup_{t \in[0,1]} h^{(m)}_{\ell}(t)
< \infty
\end{equation}
%
%and \pat{if} the spectral radius of $\Gamma$
% is strictly smaller than 1,
then there is a constant $\zeta> 0$ such that for any $\mathbf{f} \in
\mathbb{H}$,
\[
Q(\mathbf{f},\mathbf{f}) \geq\zeta\llVert \mathbf{f}\rrVert ^2.
\]
\end{Prop}

%structure $L^2$ peut-etre definir un espace sans ca ... cf commentaire
%en haut preuve prop \ref{momentsHawkes}}
We are now ready to establish oracle inequalities for multivariate
Hawkes processes.
%%%%%%%%%%%%%%%%%%%%%%
%s5.2 #&#
\subsection{Lasso for Hawkes processes}
\label{LassoH}
In the sequel, we still consider the main assumptions of the previous
subsection: We deal with a stationary Hawkes process whose intensity is
given by (\ref{eq:intensity}) such that the spectral radius of $\Gamma$
is strictly smaller than 1 and (\ref{eq:intbound}) is satisfied. We
recall that the components of $\Gamma$ are the $\gamma_{\ell, m} $'s with
\[
\gamma_{\ell, m} = \int_0^1
h^{(m)}_{\ell}(t) \,\mathrm{d} t.
\]
One of the main results of this section is to link properties of the
dictionary (mainly orthonormality but also more involved assumptions)
to properties of $G$ (the control of $\Omega_c$). To do so, let us
define for all $f\in\mathcal{H}$,
\[
\llVert f \rrVert _\infty=\max \Bigl\{\max_{m=1,\ldots,M}
\bigl\llvert \mu^{(m)} \bigr\rrvert , \max_{m,\ell
=1,\ldots,M}\bigl
\llVert g_\ell^{(m)} \bigr\rrVert _\infty \Bigr\}.
\]
Then, let us set $\llVert  \Phi  \rrVert _\infty:=\max\{\llVert  \p  \rrVert _\infty, \p\in\Phi
\}$. % and recall that $\left| \Phi \right|$ is the cardinality of $\Phi$.
The next result is based on the probabilistic results of Section~\ref{probaHawkes}.

%pr5 #&#
\begin{Prop}\label{Qnorm}
%Let $\alpha>0$.
Assume that the Hawkes process is stationary, that (\ref{eq:intbound})
is satisfied and that the spectral radius of $\Gamma$ is strictly
smaller than $1$. Let $r_\Phi$ be the spectral radius of the matrix
$\mathfrak{H}$ defined by
\[
\mathfrak{H}= \Biggl( \sum_m \Biggl[\bigl
\llvert \mu_\p^{(m)} \bigr\rrvert \bigl\llvert
\mu_\rho^{(m)} \bigr\rrvert +\sum
_{\ell
=1}^M \int_0^1
\bigl\llvert (g_\p)_\ell^{(m)} \bigr\rrvert
\bigl\llvert (g_\rho)_\ell^{(m)} \bigr\rrvert (u)
\,\mathrm{d}u \Biggr] \Biggr)_{\p,\rho\in\Phi}.
\]
Assume that $\Phi$ is orthonormal and that
%
%e5.4 #&#
\begin{equation}
\label{Aphi} %A_\Phi(T):=r_\Phi\left\Vert \Phi  \right\Vert_\infty^2(M+KM^2) [\log(\left\Vert \Phi  \right\Vert_
A_\Phi(T):=r_\Phi
\llVert \Phi \rrVert _\infty^2 \llvert \Phi \rrvert \bigl[
\log\bigl(\llVert \Phi \rrVert _\infty \bigr)+\log\bigl(\llvert \Phi
\rrvert \bigr)\bigr]\frac{\log^5(T)}{T}\to0
\end{equation}
when $T\to\infty$. Then, for any $\beta>0$, there exists $C_1>0$
depending on $\beta$ and $f_0$ such that with $c=C_1T$,
we have
\[
\P\bigl(\Omega_c^c\bigr)=\mathrm{O}\bigl(\llvert \Phi \rrvert
^2T^{-\beta}\bigr).
\]
%
%where $C_1$ and $C_2$ are two constants depending on $\alpha$ and
%$f_0$.
\end{Prop}

Up to logarithmic terms, (\ref{Aphi}) is similar to (\ref{cond:Aalen})
with $M$ replaced with $T$.
The dictionary $\Phi$ can be built via a dictionary $(\Upsilon
_k)_{k=1,\ldots,K}$ of functions of $\L{2}([0,1])$ (that may depend on
$T$) in the following way. A function $\p=(\mu^{(m)}_\p, ((g_\p)_\ell^{(m)}
)_\ell)_m$ belongs to $\Phi$ if and only if only one of its $M+M^2$
components is nonzero and in this case,
\begin{itemize}
\item if $\mu^{(m)}_\p\neq0$, then $\mu^{(m)}_\p=1$,
\item if $(g_\p)_\ell^{(m)}\neq0$, then there exists $k\in\{1,\ldots,K\}$
such that $(g_\p)_\ell^{(m)}=\Upsilon_k$.
\end{itemize}
Note that $\llvert  \Phi \rrvert =M+KM^2$. Furthermore, assume from now on that
$(\Upsilon_k)_{k=1,\ldots,K}$ is orthonormal in $\L{2}([0,1])$. Then $\Phi
$ is also orthonormal in ${\mathcal H}$ endowed with $\llVert  \cdot  \rrVert $.

Before going further, let us discuss assumption (\ref{Aphi}). First,
note that the matrix $\mathfrak{H}$ is block diagonal. The first block
is the identity matrix of size $M$. The other $M^2$ blocks are
identical to the matrix:
\[
\mathfrak{H}_K= \biggl(\int\bigl\llvert \Upsilon_{k_1}(u)
\bigr\rrvert \bigl\llvert \Upsilon_{k_2}(u) \bigr\rrvert \,\mathrm {d}u
\biggr)_{1\leq k_1,k_2\leq K}.
\]
So, if we denote $\tilde r_K$ the spectral radius of $\mathfrak{H}_K$,
we have:
\[
r_\Phi=\max(1,\tilde r_{K}).
\]
We analyze the behavior of $\tilde r_K$ with respect to $K$. Note that
for any $k_1$ and any $k_2$,
\[
(\mathfrak{H}_K)_{k_1,k_2}\geq0.
\]
Therefore,
\[
\tilde r_{K}\leq\sup_{\llVert  x  \rrVert _{\ell_1}=1}\llVert
\mathfrak{H}_K x \rrVert _{\ell
_1}\leq\max
_{k_1}\sum_{k_2} (
\mathfrak{H}_K)_{k_1,k_2}.
\]
We now distinguish three types of orthonormal dictionaries (remember
that $M$ is viewed as a constant here):
\begin{itemize}
\item Let us consider regular histograms. The basis is composed of the
functions $\Upsilon_k=\delta^{-1/2} \indic_{((k-1)\delta,k\delta]}$
with $K\delta=1$. Therefore, $\llVert  \Phi  \rrVert _\infty=\delta^{-1/2}=\sqrt {K}$. But $\mathfrak{H}_K$ is the identity matrix and $\tilde{r}_K=1$.
Hence, (\ref{Aphi}) is satisfied as soon as
\[
\frac{K^2\log(K)\log^5(T)}{T}\to0
\]
when $T\to\infty$, which is satisfied if $K=\mathrm{o} (\frac{\sqrt{T}}{\log
^3(T)} )$.
\item Assume that $\llVert  \Phi  \rrVert _\infty$ is bounded by an absolute
constant (Fourier dictionaries satisfy this assumption). Since $\tilde
r_{K}\leq K$,
(\ref{Aphi}) is satisfied as soon as
\[
\frac{K^2\log(K)\log^5(T)}{T}\to0
\]
when $T\to\infty$, which is satisfied if $K=\mathrm{o} (\frac{\sqrt{T}}{\log
^3(T)} )$.
\item Assume that $(\Upsilon_k)_{k=1,\ldots,K}$ is a compactly supported
wavelet dictionary where resolution levels belong to the set $\{
0,1,\ldots,J\}$. In this case, $K$ is of the same order as $2^J$, $\llVert  \Phi \rrVert _\infty$ is of the same order as $2^{J/2}$ and it can be
established that $\tilde r_K\leq C2^{J/2}$ where $C$ is a constant only
depending on the choice of the wavelet system (see \cite{hkpt} for
further details). Then, (\ref{Aphi}) is satisfied as soon as
\[
\frac{K^{5/2}\log(K)\log^5(T)}{T}\to0
\]
when $T\to\infty$, which is satisfied if $K=\mathrm{o} (\frac{T^{2/5}}{\log
^{12/5}(T)} )$.
\end{itemize}
To apply Theorem~\ref{th:oraclewithd}, it remains to control $\Omega_{V,B}$.
Note that
\[
\psi_t^{(m)}(\p)=\lleft\{ %
\begin{array} {l@{\qquad}l} 1&
\mbox{ if } \mu^{(m)}_\p=1,
\\
\displaystyle \int_{t-1}^{t-}\Upsilon_k(t-u)
\,\mathrm{d}N^{(\ell)}_u& \mbox{ if } (g_\p
)_\ell^{(m)}=\Upsilon_k. \end{array}
\rright.
\]
Let us define
\[
\Omega_\mathcal{N}= \bigl\{\mbox{for all } t \in[0,T], \mbox{for all } m
\in\{1,\ldots,M\}\mbox{ we have } N^{(m)}_{[t-1,t]}\leq\mathcal{N}
\bigr\}.
\]
We therefore set
%
%e5.5 #&#
\begin{equation}
\label{Bphi} B_\p=1 \qquad \mbox{if } \mu^{(m)}_\p=1
\quad \mbox{and}\quad  B_\p=\llVert \Upsilon _k \rrVert
_\infty \mathcal{N}\qquad  \mbox{if }(g_\p)_\ell^{(m)}=
\Upsilon_k.
\end{equation}
Note that on $\Omega_\mathcal{N}$, for any $\p\in\Phi$,
\[
\sup_{t\in[0,T],m}\bigl\llvert \psi^{(m)}_t(\p)
\bigr\rrvert \leq B_\p.
\]
Now, for each $\p\in\Phi$, let us determine $V_\p$ that constitutes an
upper bound of
\[
M_\p=\sum_{m=1}^M\int
_0^T \bigl[\psi^{(m)}_t(
\p)\bigr]^2 \,\mathrm{d}N^{(m)}_t.
\]
Note that only one term in this sum is nonzero. %If $\mu^{(m)}_\p=1$, $M_
%take \red{$V_\p= \left\Vert \Upsilon_k  \right\Vert_\infty^2 \lceil T \rceil
We set
%
%e5.6 #&#
\begin{equation}
\label{Vphi} V_\p=\lceil T\rceil\mathcal{N}\qquad  \mbox{if }
\mu^{(m)}_\p=1 \quad \mbox{and}\quad  V_\p = \llVert
\Upsilon_k \rrVert _\infty^2 \lceil T \rceil
\mathcal{N}^3\qquad  \mbox{if }(g_\p)_\ell^{(m)}=
\Upsilon_k,
\end{equation}
where $\lceil T\rceil$ denotes the smallest integer larger than $T$.
With this choice, one has that $\Omega_\mathcal{N}\subset\Omega
_{V,B}$, which leads to the following result.
%
%By subdivising the interval $[-1,T]$ in $J(T)$ intervals of length
%less or equal to $1$, where $J(T)\leq2+T$, using stationarity and
%applying Proposition~\ref{momentsHawkes}, we easily \red{bof...}
%obtain the following corollary.

%co3 #&#
\begin{Cor}\label{OmegaN}
Assume that the Hawkes process is stationary, that (\ref{eq:intbound})
is satisfied and that the spectral radius of $\Gamma$ is strictly
smaller than $1$. With the choices \eqref{Bphi} and \eqref{Vphi},
\[
\P(\Omega_{V,B})\geq\P(\Omega_\mathcal{N})\geq1-C_1T
\exp(-C_2\mathcal{N}),
\]
where $C_1$ and $C_2$ are positive constants depending on $f_0$.

If $\mathcal{N}\gg\log(T)$, then for all $\beta>0$,
\[
\P\bigl(\Omega_{V,B}^c\bigr)\leq\P\bigl(
\Omega_\mathcal{N}^c\bigr)=\mathrm{o}\bigl(T^{-\beta}\bigr).
\]
%
%Let $\beta>0$. The choice $\mathcal{N}\gg\log(T)$ leads to $$\P(
%take ${\mathcal N}=C\log T$ but the constant $C$ must depend on $f_0$
%as $C_2$, which is a problem for the choice of $B_\p$.}
\end{Cor}

We are now ready to apply Theorem~\ref{th:oraclewithd}.
%(\ref{eq:intbound}) is satisfied and the spectral radius of $\Gamma$
%is strictly than 1. With probability larger than $1-O(K\log(T)^{-
%qu'on prend $\mathcal{N}=\log(T)^2$}
%%$$1- 4 (M+M^2K)\left(\frac{\log\left(1+\frac{\mu T{\mathcal N}}{\alpha
%$$\frac{1}{T}\norm{\hat{f}-f_0}^2_T\leq C \inf_{a\in\R^{\Phi}}\left\{
%where $C$ is a constant depending on $f_0$, $\mu$, $ \e$, and $\alpha$.
%c^{-1}\sum_{\la\in S(a), \lambda\sim(m,\ell,k)} \left[\left(\int_0^T
%where $\lambda\sim m$ means that $\lambda$ is of the form $m$ and that
%we will use $m$ instead of $\lambda$ as dummy variable in the sum. In
%the same way, $\lambda\sim(m,\ell,k)$ means that $\lambda$ is of the
%form $(m,\ell,k)$ and that we will use $(m,\ell,k)$ instead of $
%

%co4 #&#
\begin{Cor}\label{etpourH}
Assume that the Hawkes process is stationary, that (\ref{eq:intbound})
is satisfied and that the spectral radius of $\Gamma$ is strictly
smaller than $1$. Assume that the dictionary $\Phi$ is built as
previously from an orthonormal family $(\Upsilon_k)_{k=1,\ldots,K}$. With
the notations of Theorem~\ref{th:oraclewithd}, let $B_\p$ be defined by
\eqref{Bphi} and $d_\p$ be defined accordingly with $x=\alpha\log(T)$.
Then, with probability larger than
\[
1- 4 \bigl(M+M^2K\bigr) \biggl(\frac{\log (1+\afrac{\mu\lceil T \rceil{\mathcal
N}}{\alpha\log(T)} )}{\log(1+\e)}+1
\biggr)T^{-\alpha}-\P\bigl(\Omega _{\mathcal{N}}^c\bigr)-\P
\bigl(\Omega_c^c\bigr),
\]
\[
\frac{1}{T}\llVert \hat{f}-f_0 \rrVert ^2_T
\leq C \inf_{a\in\R^{\Phi}} \biggl\{ \frac{1}{T}\llVert
f_0-f_a\rrVert ^2_T +\sum
_{\p\in S(a)} \biggl(\frac{\log
(T)(\psi(\p))^2\bullet N_T}{T^2}+\frac{B_\p^2\log^2(T)}{T^2}
\biggr) \biggr\},
\]
where $C$ is a constant depending on $f_0$, $\mu$, $ \e$, and $\alpha$.

From an asymptotic point of view, if the dictionary also satisfies
\eqref{Aphi}, and if $\mathcal{N}=\log^2(T)$ in \eqref{Bphi}, then for
$T$ large enough with probability larger than $1-C_1K\log(T)T^{-\alpha}$
\[
\frac{1}{T}\llVert \hat{f}-f_0 \rrVert ^2_T
\leq C_2 \inf_{a\in\R^{\Phi}} \biggl\{ \frac{1}{T}
\llVert f_0-f_a\rrVert ^2_T +
\frac{\log^3(T)}{T} \sum_{\p\in
S(a)} \biggl[
\frac{1}{T}\llVert \p \rrVert ^2_T+
\frac{\log^{7/2}(T)}{\sqrt{T}}\llVert \Phi \rrVert _\infty^2 \biggr]
\biggr\},
\]
where $C_1$ and $C_2$ are constants depending on $M$, $f_0$, $\mu$, $ \e
$, and $\alpha$.
\end{Cor}

We express the oracle inequality by using $\frac{1}{T}\llVert  \cdot  \rrVert _T$
simply because, when $T$ goes to $+\infty$, by ergodicity of the
process (see, for instance, \cite{DVJ}, and Proposition~\ref{flow} for
a nonasymptotic statement),
\[
\frac{1}{T}\llVert f \rrVert _T^2=\sum
_{m=1}^M\frac{1}{T}\int_0^T
\bigl(\kappa_t\bigl(\mathbf {f}^{(m)}\bigr)
\bigr)^2\,\mathrm{d}t\longrightarrow\sum_{m=1}^M
Q\bigl(\mathbf{f}^{(m)},\mathbf {f}^{(m)}\bigr)
\]
under assumptions of Proposition~\ref{Qnorm}. Note that the right-hand
side is a true norm on ${\mathcal H}$ by Proposition~\ref{eqnormH}.
Note also that
\[
\frac{\log^{7/2}(T)}{\sqrt{T}}\llVert \Phi \rrVert _\infty^2\stackrel{T
\to\infty } {\to} 0,
\]
as soon as \eqref{Aphi} is satisfied for the Fourier basis and
compactly supported wavelets. It is also the case for histograms as
soon as $K=\mathrm{o} (\frac{\sqrt{T}}{\log^{7/2}(T)} )$.
Therefore, this term can be viewed as a residual one. In those cases,
the last inequality can be rewritten as
\[
\frac{1}{T}\llVert \hat{f}-f_0 \rrVert ^2_T
\leq C \inf_{a\in\R^{\Phi}} \biggl\{ \frac{1}{T}\llVert
f_0-f_a\rrVert ^2_T +
\frac{\log^3(T)}{T} \sum_{\p\in
S(a)}\frac{1}{T}
\llVert \p \rrVert ^2_T \biggr\},
\]
for a different constant $C$, the probability of this event tending to
1 as soon as $\alpha\geq1/2$ in the Fourier and histogram cases and
$\alpha\geq2/5$ for the compactly supported wavelet basis. Once again,
as mentioned for the Poisson or Aalen models, the right-hand side
corresponds to a classical ``bias-variance'' trade off and we obtain a
classical oracle inequality up to the logarithmic terms. Note that
asymptotics is now with respect to $T$ and not with respect to $M$ as
for Poisson or Aalen models. So, the same result, namely Theorem~\ref
{th:oraclewithd}, allows to consider both asymptotics.

%$$\max_{k\in\{1,\ldots,K\}}\left\Vert \Upsilon_k  \right\Vert_\infty=O\left(\sqrt{
%then, the variance term is asymptotically smaller than $\log(T)
%
%Note also that on $\Omega_{V,B}$
%$ (\psi(\p))^2\bullet(N_t-\Lambda_t)= [(\psi(\p))^2\indic_{\sup_{m,s
%This last part is a martingale which satisfies the assumption of
%$$ (\psi(\p))^2\bullet N_T \leq(\psi(\p))^2\bullet\Lambda_T + B_\phi
%
%Consequently
%$$\frac{1}{T} (\psi(\p))^2\bullet N_T =O\left(\frac{\mathcal{N}^2}{T}
%
%eventuellement faire une vraie preuve du corollaire}
%
%So, we can draw the same conclusions as for Poisson and Aalen models.
%%
%%\red{Patricia, you probably really wish to make a remark on the
%negative case. The following remark is devoted to this extension.
%Improve just the writing.}
%%\begin{Rem}
%%The last remark concerns $\Omega_=$. If the model is the classical
%one, when all the $\nu^{(m)}$'s and $h^{(m)}_\ell$'s are nonnegative, then of
%course $\P(\Omega_=)=1$. If not, but if the model satisfies
%%for some fixed $\eta\geq0$
%%$$\forall m, \lambda^{(m)}_t=\eta+\left(\nu^{(m)}-\eta+\sum_{\ell=1}^M\int_{-
%%then on $\Omega_=$,
%%$$\nu^{(m)}-\eta+\sum_{\ell=1}^M\int_{-\infty}^{t-}h^{(m)}_\ell(t-u)dN_u^{(
%%But this can be true on $\Omega_\mathcal{N}$ if one can assume that
%%\begin{equation}\label{pos}
%%\forall m\in\{1,\ldots,M\},
%%\nu^{(m)}\geq\norm{h_\ell^{(m)}}_\infty M \mathcal{N}.
%%\end{equation}
%%\end{Rem}

%%%%%%%%%%%%%%%%%%%%%%%%%%%%%%%%%%%
%%%%%%%%%%%%%%%%%%%%%%%%%%%%%%%%%%%
%%%%%%%%%%%%%%%%%%%%%%%%%%%%%%%%%%%
%%%%%%%%%%%%%%%%%%%%%%%%%%%%%%%%%%%
%s6 #&#
\section{Simulations for the multivariate Hawkes process}\label{sec:simu}
This section is devoted to illustrations of our procedure on simulated
data of multivariate Hawkes processes and comparisons with the
well-known adaptive Lasso procedure proposed by \cite{Zou}. We consider
the general case and we do no longer assume that the functions $h_\ell^{(m)}
$ are nonnegative as in Section~\ref{sec:hawkes}.
However, if the parameter $\nu^{(m)}$ is large with respect to the $h_\ell^{(m)}
$'s, then $\psi^{(m)}(f_0)$ is nonnegative with large probability and
therefore $\lambda^{(m)}=\psi^{(m)}(f_0)$ with large probability. Hence,
Theorem~\ref{th:oraclewithd} implies that $\hat{f}$ is close to $f_0$.
%%%%%%%%%%%%%%%%%%%%%%%%%%%%%%%%%%%
%%%%%%%%%%%%%%%%%%%%%%%%%%%%%%%%%%%
%s6.1 #&#
\subsection{Description of the data}
As mentioned in the introduction, Hawkes processes can be used in
neuroscience to model the action potentials of individual neurons. So,
we perform simulations whose parameters are close, to some extent, to
real neuronal data.
For a given neuron $m\in\{1,\ldots,M\}$, we recall that its activity
is modeled by a point process $N^{(m)}$ whose intensity is
\[
\lambda_t^{(m)}= \Biggl(\nu^{(m)}+ \sum
_{\ell=1}^M \int_{-\infty}^{t-}h_\ell
^{(m)}(t-u) \,\mathrm{d}N^{(\ell)} (u) \Biggr)_+.
\]
%
%interaction functions are nonnegative and the intensity if a linear
%function of $f_0$.}
The \textit{interaction function} $h^{(m)}_\ell$ represents the influence
of the past activity of the neuron $\ell$ on the neuron $m$. The \textit
{spontaneous rate} $\nu^{(m)}$ may somehow represent the external
excitation linked to all the other neurons that are not recorded. It is
consequently of crucial importance not only to correctly infer the
interaction functions, but also to reconstruct the spontaneous rates
accurately. Usually, activity up to 10 neurons can be recorded in a
``stationary'' phase during a few seconds (sometimes up to one minute).
Typically, the points frequency is of the order of 10--80 Hz and the
interaction range between points is of the order of a few milliseconds
(up to 20 or 40 ms).
We first lead three experiments in the \textit{pure excitation case}
where all the interaction functions are nonnegative by simulating
multivariate Hawkes processes (two with $M=2$, one with $M=8$) based on
these typical values. More precisely, we take for any $m\in\{1,\dots,M\}
$, $\nu^{(m)}=20$ and the interaction functions $h^{(m)}_\ell$ are defined as
follows (supports of all the functions are assumed to lie in the
interval $[0,0.04]$):

%ex1 #&#
\begin{Experiment}[($M=2$ and piecewise constant functions)]\label{ex1}
\[
h_1^{(1)}=30\times\indic_{(0,0.02]},\qquad
h_2^{(1)}=30\times\indic _{(0,0.01]},\qquad
h_1^{(2)}=30\times\indic_{(0.01,0.02]},\qquad
h_2^{(2)}=0.
\]
In this case, each neuron depends on the other one. The spectral radius
of the matrix $\Gamma$ is 0.725.
\end{Experiment}

%ex2 #&#
\begin{Experiment}[($M=2$ and ``smooth'' functions)]\label{ex2}
 In this
experiment, $h_1^{(1)}$ and $h_1^{(2)}$ are not piecewise constant.
\begin{eqnarray*}
h_1^{(1)}(x)&=&100 \mathrm{e}^{-200x}\times
\indic_{(0,0.04]}(x) ,\qquad  h_2^{(1)}(x)=30\times
\indic_{(0,0.02]}(x),
\\
h_1^{(2)}(x)&=&\frac{1}{0.008\sqrt{2\uppi} }\mathrm{e}^{-\frac
{(x-0.02)^2}{2*0.004^2}}\times
\indic_{(0,0.04]}(x),\qquad  h_2^{(2)}(x)=0.
\end{eqnarray*}
In this case, each neuron depends on the other one as well. The
spectral radius of the matrix $\Gamma$ is 0.711.
\end{Experiment}

%ex3 #&#
\begin{Experiment}[($M=8$ and piecewise constant functions)]\label{ex3}
\[
h_2^{(1)}=h_3^{(1)}=h_2^{(2)}=h_1^{(3)}=h_2^{(3)}=h_8^{(5)}=h_5^{(6)}=h_6^{(7)}=h_7^{(8)}=25
\times \indic_{(0,0.02]}
\]
and all the other $55$ interaction functions are equal to 0. Note in
particular that this leads to 3 independent groups of dependent neurons
$\{1,2,3\}$, $\{4\}$ and $\{5,6,7,8\}$. The spectral radius of the
matrix $\Gamma$ is 0.5.
\end{Experiment}

We also lead one experiment in the \emph{pure inhibition case} where
all the interaction functions are nonpositive:

%ex4 #&#
\begin{Experiment}[($M=2$)]\label{ex4} In this experiment, the interaction
functions are the opposite of the functions introduced in Experiment~\ref{ex2}.
We take for any $m\in\{1,\dots,M\}$, $\nu^{(m)}=60$ so that $\psi_t(f_0)$
is positive with high probability.
\end{Experiment}

For each simulation, we let the process ``warm up'' during 1 second to
reach the stationary state.\footnote{Note that since the size of the
support of the interaction functions is less or equal to $0.04$, the
``warm up'' period is 25 times the interaction range.} Then the data are
collected by taking recordings during the next $T$ seconds. For
instance, we record about 100 points per neuron when $T=2$ and 1000
points when $T=20$. Figure~\ref{raster} shows two instances of data
sets with $T=2$.

%f1 #&#
\begin{figure}

\includegraphics{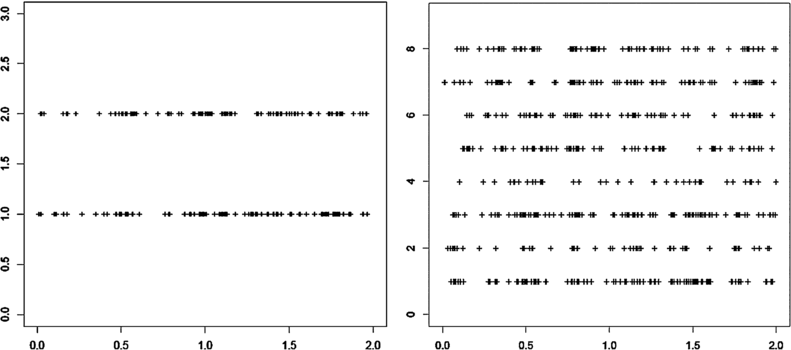}

\caption{Raster plots of two data sets with $T=2$
corresponding to Experiment~\protect\ref{ex2} on the left and Experiment~\protect\ref{ex3} on the
right. The $x$-axis correspond to the time of the experiment. Each line
with ordinate $m$ corresponds to the points of the process $N^{(m)}$. From
bottom to top, we observe 124 and 103 points for Experiment~\protect\ref{ex2} and 101,
60, 117, 38, 73, 75, 86 and 86 points for Experiment~\protect\ref{ex3}.}
\label{raster}
\end{figure}
%%%%%%%%%%%%%%%%%%%%%%%%%
%%%%%%%%%%%%%%%%%%%%%%%%%
%s6.2 #&#
\subsection{Description of the methods}
To avoid approximation errors when computing the matrix $G$, we focus
on a dictionary $(\Upsilon_k)_{k=1,\ldots,K}$ whose functions are
piecewise constant. More precisely, we take $\Upsilon_k=\linebreak[4] \delta
^{-1/2}\indic_{((k-1)\delta,k\delta]}$ with $\delta=0.04/K$ and $K$,
the size of the dictionary, is chosen later.

Our practical procedure strongly relies on the theoretical one based on
the $d_\p$'s defined in (\ref{d}), with $x$, $\mu$ and $\e$ to be
specified. First, using Corollary~\ref{etpourH}, we naturally take
$x=\alpha\log(T)$. Then, three hyperparameters would need to be tuned,
namely $\alpha$, $\mu$ and $\e$, if we directly used the Lasso estimate
of Theorem~\ref{th:oraclewithd}. So, for simplifications, we implement
our procedure by replacing the Lasso parameters $d_\p$ with
\[
\tilde d_\p(\gamma)=\sqrt{2\gamma\log(T) \bigl(\psi(\p)
\bigr)^2\bullet N_T} + \frac{\gamma\log(T)}{3} \sup
_{t\in[0,T],m}\bigl\llvert \psi^{(m)}_t(\p) \bigr
\rrvert ,
\]
where $\gamma$ is a constant to be tuned. Besides taking $x=\alpha\log
(T)$, our modification consists in neglecting the linear part $\frac
{B_\p^2x}{\mu-\phi(\mu)}$ in $\hat{V}^\mu$ and replacing $B_\p$ with
$\sup_{t\in[0,T],m}\llvert  \psi^{(m)}_t(\p) \rrvert $.\vspace*{2pt}
Then, note that, up to these modifications, the choice $\gamma=1$
corresponds to the limit case where $\alpha\to1$, $\e\to0$ and $\mu
\to0$ in the definition of the $d_\p$'s (see the comments after
Theorem~\ref{th:oraclewithd}). Note also that, under the slight abuse
consisting in identifying $B_\p$ with $\sup_{t\in[0,T],m}\llvert  \psi^{(m)}_t(\p
) \rrvert $, for every parameter $\mu$, $\e$ and $\alpha$ of Theorem~\ref
{th:oraclewithd} with $x=\alpha\ln(T)$, one can find two parameters
$\gamma$ and $\gamma'$ such that
\[
\tilde d_\p(\gamma)\leq d_\p\leq\tilde
d_\p\bigl(\gamma'\bigr).
\]
Therefore, this practical choice is consistent with the theory and
tuning hyperparameters reduces to only tuning $\gamma$. Our simulation
study will provide sound answers to the question of tuning $\gamma$.

We compute the Lasso estimate by using the shooting method of \cite{Fu}
and the \texttt{R}-package \texttt{Lassoshooting}.
In particular, we need to invert the matrix $G$. In all simulations,
this matrix was invertible,
which is consistent with the fact that $\Omega_c$ happens with large
probability. Note also that the value of
$c$, namely the smallest eigenvalue of $G$, can be very small (about
$10^{-4}$) whereas the largest eigenvalue
is potentially as large as $10^{5}$, both values highly depending on
the simulation and on $T$. Fortunately,
those values are not needed to compute our Lasso estimate. Since it is
based on \textit{Bernstein type inequalities},
our Lasso method is denoted \textbf{B} in the sequel.

Due to their soft thresholding nature, Lasso methods are known to
underestimate the coefficients \cite{Mein,Zou}.
To overcome biases in estimation due to shrinkage, we propose a two
steps procedure, as usually suggested in the
literature: Once the support of the vector has been estimated by
\textbf{B}, we compute the ordinary least-square
estimator among the vectors $a$ having the same support, which
provides the final estimate. This method is denoted \textbf{BO} in the sequel.

Another popular method is \textit{adaptive Lasso} proposed by Zou \cite
{Zou}. This method overcomes the flaws of standard Lasso by taking $\ell
_1$-weights of the form
\[
d^a_\p(\gamma)=\frac{\gamma}{2\llvert  \hat{a}_{\p}^o \rrvert ^p},
\]
where $p>0$, $\gamma>0$ and $\hat{a}_{\p}^o$ is a preliminary
consistent estimate of the true coefficient. Even if the shapes of the
weights are different, the latter are data-driven and this method
constitutes a natural competitive method with ours. The most usual
choice, which is adopted in the sequel, consists in taking $p=1$ and
the ordinary least squares estimate for the preliminary estimate (see
\cite{HMZ,vdGBZ,Zou}). Then, penalization is stronger for coefficients
that are preliminary estimated by small values of the ordinary least
square estimate. In the literature, the parameter $\gamma$ of adaptive
Lasso is usually tuned by cross-validation, but this does not make
sense for Hawkes data that are fully dependent. Therefore, a
preliminary study has been performed to provide meaningful values for~$\gamma$.
Results are given in the next section. This adaptive Lasso
method is denoted \textbf{A} in the sequel and \textbf{AO} when
combined with ordinary least squares in the same way as for \textbf{BO}.

Simulations are performed in \texttt{R}. The computational time is
small (merely a few seconds for one estimate even when $M=8$, $T=20$
and $K=8$ on a classical laptop computer), which constitutes a clear
improvement with respect to existing adaptive methods for Hawkes
processes. For instance, the ``Islands'' method\footnote{This method
developed for $M=1$ could easily be theoretically adapted for larger
values of $M$, but its extreme computational cost prevents us from
using it in practice.} of \cite{RBS} is limited to the estimation
of one or two dozens of coefficients at most, because of an extreme
computational memory cost whereas here when $M=8$ and $K=8$, we can
easily deal with $M+KM^2=520$ coefficients.

%%%%%%%%%%%%%%%%%%%%%%%%%
%%%%%%%%%%%%%%%%%%%%%%%%%
%s6.3 #&#
\subsection{Results}\label{results}
%
%f2 #&#
\begin{figure}

\includegraphics{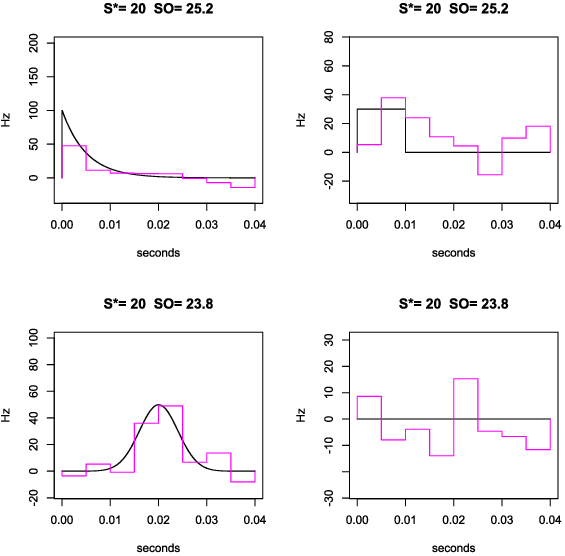}

\caption{Reconstructions corresponding to Experiment~\protect\ref{ex2} with
the OLS estimate with $T=2$ and $K=8$. Each line $m$ represents the
function $h_\ell^{(m)}$, for $\ell=1,2$. The spontaneous rates
associated with each line $m$ are given above the graphs where S$^*$
denotes the true spontaneous rate and its estimator is denoted by SO.
The true interactions functions are plotted in black whereas the OLS
estimates are plotted in magenta.}
\label{OLS}
\end{figure}

First, we provide in Figure~\ref{OLS} reconstructions by using the OLS
estimate on the whole dictionary, which corresponds to the case where
all the weights $d_\p$ are null. As expected, reconstructions are not
sparse and also bad due to a small signal to noise ratio (remember that
$T=2$).%too much noise. %Note also that there can be very high
%estimated coefficients that corresponds to null true coefficients.
%Indeed, the full OLS method corresponds to $\gamma=0$ in our method.
%When $\gamma$ increases, coefficients that correspond to high level of
%noise are thresholded and put to zero. Consequently a high estimated
%coefficient may disappear before a smaller coefficient when $\gamma$
%increases because its intrinsic level of noise is higher. This is due
%to the Berstein weights that automatically scales to the right level
%of noise

Now let us consider methods leading to sparsity. A precise study over
100 simulations has been carried out corresponding to Experiments 1
and 3 for which we can precisely check if the support of the vector
$\hat{a}$ is the correct one. For each method, we have selected 3
values for the hyperparameter $\gamma$ based on results of preliminary
simulations. Before studying mean squared errors, we investigate
the following problems that are stated in order of importance. We wonder
whether our procedure can identify:
\begin{itemize}
\item[-] \textit{the dependency groups}. Recall that two neurons belong
to the same group if and only if they are connected directly or through
the intermediary of one or several neurons. This issue is essential
from the neurobiological point of view since knowing interactions
between two neurons is of capital importance.
\item[-] \textit{the nonzero interaction functions $h_\ell^{(m)}$'s
and nonzero spontaneous rates $\nu^{(m)}$'s}. For $\ell,m\in\{1,\ldots
,M\}$, the neuron $\ell$ has a significative \emph{direct} interaction
on neuron $m$ if and only if $h_\ell^{(m)}\neq0$;
\item[-] \textit{the nonzero coefficients of nonzero interaction
functions}. This issue is more mathematical. However, it may provide
information about the maximal range for direct interactions between two
given neurons or about the favored delay of interaction.
\end{itemize}
Note that the dependency groups are the only features that can be
detected by classical analysis tools of neuroscience, such as the
Unitary Events method \cite{grun}. In particular, to the best of our
knowledge, identification of the nonzero interaction functions inside
a dependency group is a problem that has not been solved yet as far as
we know.
%
%t1 #&#
\begin{sidewaystable}
 \tablewidth=\textwidth
 \tabcolsep=0pt
 \caption{Numerical results of both procedures over 100
runs with $K=4$. Results for Experiment~\protect\ref{ex1} (top) and Experiment~\protect\ref{ex3}
(bottom) are given for $T=2$ (left) and $T=20$ (right). ``DG'' gives the
number of correct identifications of dependency groups over 100 runs.
``S'' gives the median number of nonzero spontaneous rate estimates,
``*'' means that all the spontaneous rate estimates are nonzero over
all the simulations. ``F$+$'' gives the median number of additional
nonzero interaction functions w.r.t. the truth. ``F$-$'' gives the median
number of missing nonzero interaction functions w.r.t. the truth.
``Coeff$+$'' and ``Coeff$-$'' are defined in the same way for the
coefficients. ``SpontMSE'' is the Mean Square Error for the spontaneous
rates with or without the additional ``ordinary least squares step''.
``InterMSE'' is the analog for the interaction functions. In bold, we
give the optimal values}\label{tableN2}
\begin{tabular*}{\textwidth}{@{\extracolsep{\fill}}ld{3.1}d{3.0}d{3.0}d{3.1}d{3.0}d{3.1}ld{3.1}d{3.1}d{3.1}d{3.3}d{3.1}d{3.2}@{}}
\hline
 \multicolumn{1}{l}{$M=2$, $T=2$} & \multicolumn{3}{l}{Our Lasso method}&
\multicolumn{3}{l}{Adaptive Lasso}&\multicolumn{1}{l}{$M=2$, $T=20$} &
\multicolumn{3}{l}{Our Lasso method}&\multicolumn{3}{l}{Adaptive Lasso}\\[-5pt]
 \multicolumn{1}{l}{\hrulefill} & \multicolumn{3}{l}{\hrulefill}&
\multicolumn{3}{l}{\hrulefill}&\multicolumn{1}{l}{\hrulefill} &
\multicolumn{3}{l}{\hrulefill}&\multicolumn{3}{l}{\hrulefill}\\
 \multicolumn{1}{l}{$\gamma$} & \multicolumn{1}{l}{0.5} & \multicolumn{1}{l}{1} & \multicolumn{1}{l}{2} & \multicolumn{1}{l}{2}
 & \multicolumn{1}{l}{200} & \multicolumn{1}{l}{1000} &\multicolumn{1}{l}{$\gamma$} & \multicolumn{1}{l}{0.5} & \multicolumn{1}{l}{1}
  & \multicolumn{1}{l}{2} & \multicolumn{1}{l}{2}
& \multicolumn{1}{l}{200} & \multicolumn{1}{l}{1000} \\
\hline
DG & \red{100} & \red{100} & 98 & \red{100} & \red{100} & 98 &DG &
\red{100}& \red{100} & \red{100} & \red{100}& \red{100}& \red{100}\\
S & \red{*} & \red{*} & \red{*} & 2 & 2 & 1 & S & \red{*} & \red{*}
& \red{*} & \red{*} & \red{*} & \red{*} \\
F$+$ & \red{0} & \red{0} & \red{0} & 1 & \red{0} & \red{0} &F$+$ &
\red{0} & \red{0} & \red{0} & 1 & \red{0} & \red{0} \\
F$-$ & \red{0} & \red{0} & \red{0} & \red{0} & \red{0} & \red{0} &F$-$ &
\red{0} & \red{0} & \red{0} & \red{0} & \red{0} & \red{0} \\
Coeff$+$ & 2 & 1 & \red{0} & 11 & 2 & \red{0} &Coeff$+$ & 1 & \red{0} &
\red{0} & 11 & 2 & \red{0}\\
Coeff$-$ & \red{0} & \red{0} & \red{0} & \red{0} & \red{0} &
\red{0}&Coeff$-$ & \red{0} & \red{0} & \red{0} & \red{0} & \red{0} &
\red{0}\\
SpontMSE &  108 & 140 &214 & 150 & 193 & 564&
SpontMSE &  22 & 37 & 69 & 14 & 12 & 27\\
 \multicolumn{1}{r}{$+$ols} & 104 & 96 & \red{95} & 151 & 154 & 516 & \multicolumn{1}{r}{$+$ols}& 11 & 10 &
\red{9} & 14 & 12 & 10 \\
InterMSE &  \red{7} & 9 & 15 & 13 & 8 & 11&
InterMSE &  2 & 3 & 6 & 1.4 & 0.6 & 0.5\\
\multicolumn{1}{r}{$+$ols}& \red{7} & \red{7} & \red{7} & 14 & 10 & 10& \multicolumn{1}{r}{$+$ols} & 0.6 &
0.5 & \red{0}.\red{4} & 1.4 & 0.9 & \red{0}.\red{4}\\
 \hline
 \multicolumn{1}{l}{$M=2$, $T=2$} & \multicolumn{3}{l}{Our Lasso method}&
\multicolumn{3}{l}{Adaptive Lasso}&\multicolumn{1}{l}{$M=2$, $T=20$} &
\multicolumn{3}{l}{Our Lasso method}&\multicolumn{3}{l}{Adaptive Lasso}\\[-5pt]
 \multicolumn{1}{l}{\hrulefill} & \multicolumn{3}{l}{\hrulefill}&
\multicolumn{3}{l}{\hrulefill}&\multicolumn{1}{l}{\hrulefill} &
\multicolumn{3}{l}{\hrulefill}&\multicolumn{3}{l}{\hrulefill}\\
\multicolumn{1}{l}{$\gamma$} & \multicolumn{1}{l}{0.5} & \multicolumn{1}{l}{1} & \multicolumn{1}{l}{2} & \multicolumn{1}{l}{2}
 & \multicolumn{1}{l}{200} & \multicolumn{1}{l}{1000} &\multicolumn{1}{l}{$\gamma$} & \multicolumn{1}{l}{0.5} & \multicolumn{1}{l}{1}
  & \multicolumn{1}{l}{2} & \multicolumn{1}{l}{2}
& \multicolumn{1}{l}{200} & \multicolumn{1}{l}{1000} \\
\hline
DG & 0 & \red{32} & 24 & 0 & 0 & \red{32} &DG & 63 & 99 & \red{100}
& 0 & 0 & 90 \\
S & \red{*} & \red{*} & \red{*} & 8 & 7 & 5 &S & \red{*} & \red{*} &
\red{*} & \red{*} & \red{*} & \red{*} \\
F$+$ & 17 & 6 & 1 & 55 & 13 & \red{0}.\red{5} &F$+$ & 3 & 1 & \red{0} & 55 &
10 & \red{0} \\
F$-$ & \red{0} & \red{0} & 2 & \red{0} & \red{0} & 2 &F$-$ & \red{0} &
\red{0} & \red{0} & \multicolumn{1}{l}{\phantom{97}{\ref{Bphi}$\mathrm{d}0$}} & \red{0} & \red{0} \\
Coeff$+$ & 22 & 7 & \red{1} & 199.5 & 17 & \red{1}&Coeff$+$ & 4 & 1 &
\red{0} & 197 & 13 & \red{0}\\
Coeff$-$& 0.5 & 2 & 7 & \red{0} & 2 & 7&Coeff$-$ & \red{0} & \red{0} &
\red{0} & \red{0} & \red{0} & \red{0}\\
SpontMSE &  \red{295} & 428 & 768 & 1445 & 1026 &
1835&SpontMSE & 82 & 166 & 355 & 104 & 43 & 64\\
 \multicolumn{1}{r}{$+$ols}& 1327 & 587 & 859 & 1512 & 1058 & 1935& \multicolumn{1}{r}{$+$ols}& 41 & 26 &
\red{24} & 107 & 74 & 26 \\
InterMSE & \red{38} & 51 & 79 & 214 & 49 & 65&
InterMSE &  10 & 19 & 39 & 16 & 2.9 & 3.17\\
  \multicolumn{1}{r}{$+$ols}& 63 & 45 & 61 & 228 & 84 & 70& \multicolumn{1}{r}{$+$ols}& 3 & 2.1 & \red{1}.\red{9} & 17
& 6.3 & 2\\
 \hline
\end{tabular*}
\end{sidewaystable}

Results for our method and for adaptive Lasso can be found in Table~\ref{tableN2}. This preliminary study also provides answers for tuning issues.
The line ``DG'' gives the number of correct identifications of
dependency groups. For instance, for $M=8$, ``DG'' gives the number of
simulations for which the 3 dependency groups $\{1,2,3\}$, $\{4\}$ and
$\{5,6,7,8\}$ are recovered by the methods. When $M=2$, both methods
correctly find that neurons 1 and 2 are dependent, even if $T=2$. When
8 neurons are considered, the estimates should find 3 dependency
groups. We see that even with $T=2$, our method with $\gamma=1$
correctly guesses the dependency groups for 32\% of the simulations.
It's close or equal to 100\% when $T=20$ with $\gamma=1$ or $\gamma=2$.
The adaptive Lasso has to take $\gamma=1000$ for $T=2$ and $T=20$ to
obtain as convincing results. Clearly, smaller choices of $\gamma$ for
adaptive Lasso leads to bad estimations of the dependency groups. Next,
let us focus on the detection of nonzero spontaneous rates. Whatever
the experiment and the parameter $\gamma$, our method is optimal
whereas adaptive Lasso misses some nonzero spontaneous rates when
$T=2$. Under this criterion, for adaptive Lasso, the choice $\gamma
=1000$ is clearly bad when $T=2$ (the optimal value of $S$ is $S=2$
when $M=2$ and $S=8$ when $M=8$) on both experiments, whereas $\gamma
=2$ or $\gamma=200$ is better. Not surprisingly, the number of
additional nonzero functions and additional nonzero coefficients
decreases when $T$ grows and when $\gamma$ grows, whatever the method
whereas the number of missing functions or coefficients increases. We
can conclude from these facts and from further analysis of Table~\ref{tableN2} that the choice $\gamma=0.5$ for our method and the choice
$\gamma=2$ for the adaptive Lasso are wrong choices of the tuning
parameters. In conclusion of this preliminary study, our method with
$\gamma=1$ or $\gamma=2$ seems a good choice and is robust with respect
to $T$. When $T=20$, the optimal choice for adaptive Lasso is $\gamma
=1000$. When $T=2$, the choice is not so clear and depends on the
criterion we wish to favor.

Now let us look at mean squared errors (MSE). Since the spontaneous
rates do not behave like the other coefficients, we split the MSE in
two parts: one for the spontaneous rates:
\[
\mbox{SpontMSE}=\sum_{m=1}^M\bigl(\hat{
\nu}^{(m)}-\nu^{(m)}\bigr)^2,
\]
and one for interactions:
\[
\mbox{InterMSE}=\sum_{m=1}^M\sum
_{\ell=1}^M \int\bigl(\hat{h}_\ell^{(m)}
(t)-h_\ell^{(m)}(t)\bigr)^2 \,\mathrm{d}t.
\]
We still report the results for \textbf{B}, \textbf{BO}, \textbf{A} and
\textbf{AO} in Table~\ref{tableN2}. Our comments mostly focus on cases where
the results for the previous study are good. First, note that results
on such cases are better by using the second step (OLS). Furthermore,
MSE is increasing with $\gamma$ for \textbf{B} and \textbf{A}, since
underestimation is stronger when $\gamma$ increases. This phenomenon
does not appear for two step procedures, which leads to a more stable
MSE. %\pat{a rediscuter, je comprends pas le lien avec R1}\red{A
%supprimer ? One of the main differences between both methods can be
%seen by analyzing SpontMSE. Since adaptive Lasso does not detect all
%nonzero spontaneous rates, the corresponding MSE cannot be good and
%this cannot be improved via the OLS transformation. This comforts us
%in the fact that the choice $\gamma=1000$ is a wrong choice for $T=2$
%and adaptive Lasso.}
For adaptive Lasso, when $T=2$, the choice $\gamma=200$ leads to good
MSE, but the MSE are smaller for \textbf{BO} with $\gamma=1$. When
$T=20$, the choice $\gamma=1000$ for \textbf{AO} leads to results that
are of the same magnitude as the ones obtained by \textbf{BO} with
$\gamma=1$ or $2$. Still for $T=20$, results for the estimate \textbf{B}
are worse than results for \textbf{A}. It is due to the fact that
shrinkage is larger in our method for the coefficients we want to keep
than shrinkage of adaptive Lasso that becomes negligible as soon as
the true coefficients are large enough. However the second step
overcomes this problem.
%
%f3 #&#
\begin{figure}

\includegraphics{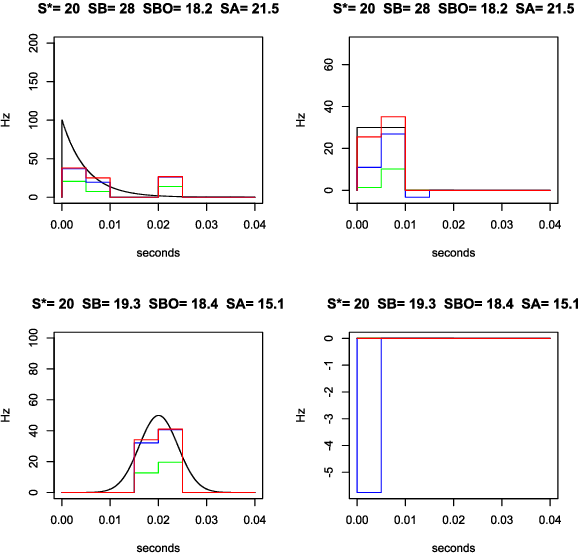}

\caption{Reconstructions corresponding to Experiment~\protect\ref{ex2}
with $T=2$ and $K=8$. Each line $m$ represents the function $h_\ell
^{(m)}$, for $\ell=1,2$. The spontaneous rates estimation associated
with each line $m$ is given above the graphs: S$^*$ denotes the true
spontaneous rate and its estimators computed by using \textbf{B},
\textbf{BO} and \textbf{A} respectively are denoted by SB, SBO and SA.
The true interactions functions (in black) are reconstructed by using
\textbf{B}, \textbf{BO} and \textbf{A} providing reconstructions in
green, red and blue respectively. We use $\gamma=1$ for \textbf{B} and
\textbf{BO} and $\gamma=200$ for \textbf{A}.}\vspace*{3pt}
\label{smoothT2}
\end{figure}
%
%f4 #&#
\begin{figure}

\includegraphics{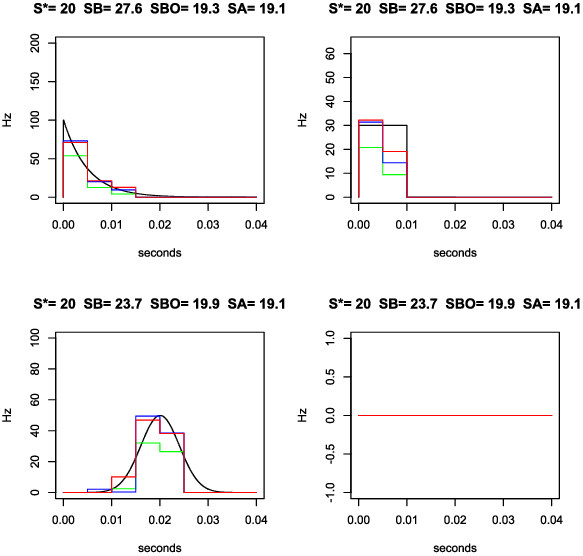}

\caption{Reconstructions corresponding to Experiment~\protect\ref{ex2} with $T=20$ and $K=8$. %Each line $m$ represents the function $h_
%associated with each line $m$ is given above the graphs: S$^*$ denotes
%the true spontaneous rate and its estimators computed by using
%SBO and SA. The true interactions functions (in black) are
%reconstructed by using \textbf{B}, \textbf{BO} and \textbf{A}
%providing reconstructions in green, red and blue respectively.
Same convention as in Figure \protect\ref{smoothT2}. We use $\gamma=1$
for \textbf{B} and \textbf{BO} and $\gamma=1000$ for \textbf{A}.}\label
{smoothT20}
\end{figure}

Note also that a more thorough study of the tuning parameter $\gamma$
has been performed by \cite{bertin11:_adapt_dantiz} where it is
mathematically proved that the choice $\gamma<1$ leads to very
degenerate estimates in the density setting. Their method for choosing
Lasso parameters being analogous to ours, it seems coherent to obtain
worse MSE for $\gamma=0.5$ than for $\gamma=1$ or $\gamma=2$, at least
for \textbf{BO}. The boundary $\gamma=1$ in their simulation study
seems to be a robust choice there, and it seems to be the case here too.

We now provide some reconstructions by using Lasso methods. Figures~\ref{smoothT2} and \ref{smoothT20} give the reconstructions corresponding
to\vadjust{\goodbreak} Experiment~\ref{ex2} ($M=2$) with $K=8$ for $T=2$ and $T=20$, respectively.
The reconstructions are quite satisfying. Of course, the quality
improves when $T$ grows. We also note improvements by using \textbf{BO}
instead of \textbf{B}. For adaptive Lasso, improvements by using the
second step are not significative and this is the reason why we do not
represent reconstructions with \textbf{AO}.
Graphs of the right-hand side of Figure~\ref{smoothT2} illustrate the
difficulties of adaptive Lasso to recover the exact support of
interactions functions, namely $h_2^{(1)}$ and $h_2^{(2)}$ for $T=2$.
%When $T=2$, the reconstruction seem a bit underpenalized in particular
%for \textbf{A}, but larger choice of $\gamma$ tends to kill the
%spontaneous rate estimation (those coefficients are shrinked to 0).
Figure~\ref{8Neurons} provides another illustration in the case of
Experiment~\ref{ex3} ($M=8$) with $K=8$ for $T=20$. For the sake of clarity, we
only represent reconstructions for the first 4 neurons. From the
estimation point of view, this illustration provides a clear hierarchy
between the methods: \textbf{BO} seems to achieve the best results and
\textbf{B} the worst. Finally, Figure~\ref{Neg} shows that even in the
inhibition case, we are able to recover the negative interactions.
%
%Here again one sees a typical phenomenon where there is one additional
%coefficient / function that is predicted, namely on $h_3^{(3)}$
%without changing the independency structures of the groups. The
%corresponding coefficient is still much smaller than the other non
%zero coefficients. This time \textbf{S} and \textbf{SO} seem more
%underpenalized than \textbf{A}. This effect really depends on the
%simulation and for the three settings of Figures \ref{smoothT2},
%first sight between the three methods, except that as expected
%improves \textbf{S} estimation. Note however that we were forced to
%change $\gamma$ for \textbf{A} depending on $T$ whereas the choice of $
%
%f5 #&#
\begin{sidewaysfigure}

\includegraphics{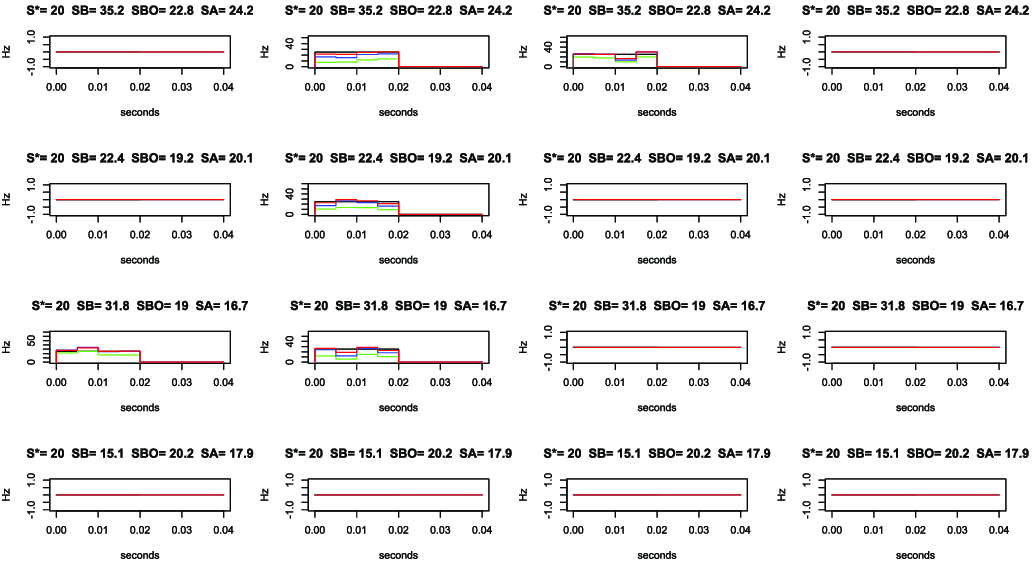}

\caption{Reconstructions corresponding to Experiment~\protect\ref{ex3} with $T=20$ and
$K=8$ and for the first 4 neurons. Each line $m$ represents the
function $h_\ell^{(m)}$, for $\ell=1,2,3,4$. %The spontaneous rates
%associated with each line $m$ are given above the graphs where S$^*$
%denotes the true spontaneous rate and its estimators computed by using
%SBO and SA. The true interactions functions (in black) are
%reconstructed by using \textbf{B}, \textbf{BO} and \textbf{A}
%providing reconstructions in green, red and blue respectively.
Same convention as in Figure \protect\ref{smoothT2}. We use $\gamma=1$
for \textbf{B} and \textbf{BO} and $\gamma=1000$ for \textbf{A}.}\label
{8Neurons}
\end{sidewaysfigure}
%
%f6 #&#
\begin{figure}

\includegraphics{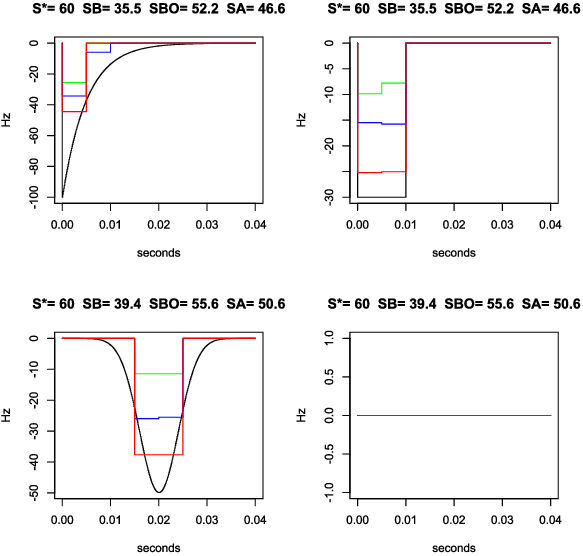}

\caption{Reconstructions corresponding to Experiment~\protect\ref{ex4} with $T=20$ and
$K=8$. Same conventions as in Figure \protect\ref{smoothT2}.
We use $\gamma=1$ for \textbf{B} and \textbf{BO} and $\gamma=1000$ for
\textbf{A}.}
\label{Neg}
\end{figure}

%%%%%%%%%%%%%%%%%%%%%%%%%%%%%%
%%%%%%%%%%%%%%%%%%%%%%%%%%%%%%
%s6.4 #&#
\subsection{Conclusions}
With respect to the problem of tuning our methodology based on
Bernstein type inequalities, our simulation study is coherent with
theoretical aspects since we achieve our best results by taking $\gamma
=1$, which constitutes the limit case of assumptions of Theorem~\ref
{th:oraclewithd}. For practical aspects, we recommend the choice $\gamma
=1$ even if $\gamma=2$ is acceptable. More importantly, this choice is
robust with respect to the duration of recordings, which is not the
case for adaptive Lasso. Implemented with $\gamma=1$, our method
outperforms adaptive Lasso and it is able to recover the dependency
groups, the nonzero spontaneous rates, the nonzero functions and even
the nonzero coefficients as soon as $T$ is large enough. Most of the
time, the two step procedure \textbf{BO} seems to achieve the best
results for parameter estimation.

It is important to note that the question of tuning adaptive Lasso
remains open. Some values of $\gamma$ allow us to obtain very good
results %($\gamma=200$ for $T=2$ and $\gamma=1000$ for $T=20$)
but they are not robust with respect to $T$, which may constitute a
serious problem for practitioners. In the standard regression
setting, this problem may be overcome by using cross-validation on
independent data, which somehow estimates random fluctuations. But in
this multivariate Hawkes setup, independence assumptions on data
cannot be made and this explains the problems for tuning adaptive
Lasso. Our method based on Bernstein type concentration inequalities
takes into account those fluctuations. It also takes into account the
nature of the coefficients and the variability of their estimates
which differ for spontaneous rates on the one hand and coefficients
of interaction functions on the other hand. The shape of weights of
adaptive Lasso does not incorporate this difference, which explains
the contradictions for tuning the method when $T=2$. For instance, in
some cases, adaptive Lasso tends to estimate some spontaneous rates to
zero in order to achieve better performance on the interaction functions.

\section{Proofs}
This section is devoted to the proofs of the results of the paper.
Throughout, $C$ is a constant whose value may change from line to line.
%%%%%%%%%%%%%%%%%%%%%%%%%%%%%%%%%%%
%s7.1 #&#
\subsection{Proof of Theorem \texorpdfstring{\protect\ref{th:oracle}}{1}}\label{sec:proof-oracle}
The proof of Theorem~\ref{th:oracle} is standard (see for instance \cite
{bunea}), but for the sake of completeness, we give it.
We use $\llVert  \cdot  \rrVert _{\ell_2}$ for the Euclidian norm of $\R^{\Phi}$.
Given $a$ recall that
\[
f_{a}=\sum_{\p\in\Phi} a_\p\p.
\]
Then, we have $\hat{f}=f_{\hat a}$,
\[
a' b=\psi(f_a)\bullet N_T
\]
and
\[
a'G a=\llVert f_a \rrVert _T^2=
\bigl\llVert \psi(f_{a}) \bigr\rrVert _{\proc}^2.
\]
Then,
\begin{eqnarray*}
-2 \psi(f_{\hat a})\bullet N_T +\llVert f_{\hat a}
\rrVert _{T}^2+ 2 d'\llvert \hat a \rrvert
\leq-2 \psi(f_{ a})\bullet N_T +\llVert f_{ a}
\rrVert _{T}^2+ 2 d'\llvert a \rrvert .
\end{eqnarray*}
So,
\begin{eqnarray*}
\bigl\llVert \psi(f_{\hat a})-\lambda \bigr\rrVert _{\proc}^2&=&
\bigl\llVert \psi(f_{\hat
a}) \bigr\rrVert _{\proc}^2+
\llVert \lambda \rrVert _{\proc}^2-2\bigl\langle
\psi(f_{\hat a}),\lambda \bigr\rangle _{\proc}
\\
&\leq& \bigl\llVert \psi(f_{a}) \bigr\rrVert _{\proc}^2+
\llVert \lambda \rrVert _{\proc}^2+2\psi(f_{\hat
a}-f_a)
\bullet N_T
\\
&&{} + 2 d' \bigl(\llvert a \rrvert -\llvert \hat a \rrvert \bigr)-2
\bigl\langle \psi (f_{\hat a}),\lambda\bigr\rangle _{\proc}
\\
&=&\bigl\llVert \psi(f_{a})-\lambda \bigr\rrVert
_{\proc}^2 + 2\bigl\langle \psi(f_{a}-f_{\hat
a}),
\lambda\bigr\rangle _{\proc}
\\
&&{} + 2\psi(f_{\hat a}-f_a)\bullet N_T +2
d' \bigl(\llvert a \rrvert -\llvert \hat a \rrvert \bigr)
\\
&=&\bigl\llVert \psi(f_{a})-\lambda \bigr\rrVert
_{\proc}^2 + 2\psi(f_a-f_{\hat a})\bullet
(\Lambda-N)_T+2 d' \bigl(\llvert a \rrvert -\llvert
\hat a \rrvert \bigr)
\\
&=&\bigl\llVert \psi(f_{a})-\lambda \bigr\rrVert
_{\proc}^2+2\sum_{\p\in\Phi}
(a_\p-\hat a_\p)\psi(\p)\bullet(\Lambda-N)_T+2
d' \bigl(\llvert a \rrvert -\llvert \hat a \rrvert \bigr)
\\
&\leq&\bigl\llVert \psi(f_{a})-\lambda \bigr\rrVert
_{\proc}^2+2\sum_{\p\in\Phi} \llvert
a_\p -\hat a_\p \rrvert \times\llvert \bar
b_\p-b_\p \rrvert +2 d' \bigl(\llvert a
\rrvert -\llvert \hat a \rrvert \bigr).
\end{eqnarray*}
Using \eqref{Coeffcontrol}, we obtain:
\begin{eqnarray*}
\bigl\llVert \psi(f_{\hat a})-\lambda \bigr\rrVert _{\proc}^2&
\leq&\bigl\llVert \psi (f_{a})-\lambda \bigr\rrVert
_{\proc}^2+2\sum_{\p\in\Phi}
d_\p\llvert a_\p-\hat a_\p \rrvert +2\sum
_{\p\in\Phi} d_\p \bigl(\llvert
a_\p \rrvert -\llvert \hat a_\p \rrvert \bigr)
\\
&\leq&\bigl\llVert \psi(f_{a})-\lambda \bigr\rrVert
_{\proc}^2+2\sum_{\p\in\Phi}
d_\p \bigl(\llvert a_\p-\hat a_\p \rrvert
+\llvert a_\p \rrvert -\llvert \hat a_\p \rrvert
\bigr).
\end{eqnarray*}
Now, if $\p\notin S(a)$, $\llvert  a_\p-\hat a_\p \rrvert +\llvert  a_\p \rrvert -\llvert  \hat a_\p \rrvert =0$, and
\begin{eqnarray*}
\bigl\llVert \psi(f_{\hat a})-\lambda \bigr\rrVert _{\proc}^2
&\leq&\bigl\llVert \psi(f_{a})-\lambda \bigr\rrVert
_{\proc}^2+2\sum_{\p\in S(a)}d_\p
\bigl(\llvert a_\p-\hat a_\p \rrvert +\llvert
a_\p \rrvert -\llvert \hat a_\p \rrvert \bigr)
\\
&\leq&\bigl\llVert \psi(f_{a})-\lambda \bigr\rrVert
_{\proc}^2+4\sum_{\p\in S(a)}d_\p
\bigl(\llvert a_\p-\hat a_\p \rrvert \bigr)
\\
&\leq&\bigl\llVert \psi(f_{a})-\lambda \bigr\rrVert
_{\proc}^2+4\llVert \hat a-a \rrVert _{\ell
_2} \biggl(
\sum_{\p\in S(a)}d_\p^2
\biggr)^{1/2}.
\end{eqnarray*}
We now use the assumption on the Gram matrix given by \eqref{normlower}
and the triangular inequality for $\llVert \cdot  \rrVert _{T}$, which yields
\begin{eqnarray*}
\llVert \hat a-a \rrVert _{\ell_2}^2&\leq& c^{-1} (
\hat a-a )'G ( \hat a-a )
\\
&=&c^{-1}\llVert f_{\hat a}-f_{a} \rrVert
_{T}^2
\\
&\leq&2c^{-1} \bigl(\bigl\llVert \psi(f_{\hat a})-\lambda \bigr
\rrVert _{\proc}^2+\bigl\llVert \psi (f_{a})-
\lambda \bigr\rrVert _{\proc}^2 \bigr).
\end{eqnarray*}
Let us take $\alpha\in(0,1)$. Since for any $x\in\R$ and any $y\in\R$,
$2xy\leq\alpha x^2+\alpha^{-1}y^2$,
we obtain:
\begin{eqnarray*}
\bigl\llVert \psi(f_{\hat a})-\lambda \bigr\rrVert _{\proc}^2&
\leq &\bigl\llVert \psi(f_{a})-\lambda \bigr\rrVert
_{\proc}^2
\\
&&{} +4\sqrt{2}c^{-1/2}\sqrt {\bigl\llVert \psi(f_{\hat a})-
\lambda \bigr\rrVert _{\proc}^2+\bigl\llVert
\psi(f_{a})-\lambda \bigr\rrVert _{\proc}^2}
\biggl(\sum_{\p\in S(a)}d_\p^2
\biggr)^{1/2}
\\
&\leq&\bigl\llVert \psi(f_{a})-\lambda \bigr\rrVert
_{\proc}^2
\\
&&{} +\alpha \bigl(\bigl\llVert \psi(f_{\hat a})-\lambda \bigr\rrVert
_{\proc}^2+\bigl\llVert \psi(f_{a})-\lambda \bigr
\rrVert _{\proc
}^2 \bigr)+8\alpha^{-1}c^{-1}
\sum_{\p\in S(a)}d_\p^2
\\
&\leq&(1-\alpha)^{-1} \biggl((1+\alpha)\bigl\llVert
\psi(f_{a})-\lambda \bigr\rrVert _{\proc
}^2+8
\alpha^{-1}c^{-1}\sum_{\p\in S(a)}d_\p^2
\biggr).
\end{eqnarray*}
The theorem is proved just by taking an arbitrary absolute value for
$\alpha\in(0,1)$.
%%%%%%%%%%%%%%%%%%%%%%%%%%%
%s7.2 #&#
\subsection{Proof of Theorem \texorpdfstring{\protect\ref{th:oraclewithd}}{2}}
%We place on the event $\Omega_{V,B}\cap\Omega_c$. Then, we apply on
%the one hand Theorem~\ref{th:oracle} and on the other hand Theorem
%$$b_\p-\bar b_\p=\psi(\p)\bullet(N-\Lambda)_T.$$
%Secondly, on $\Omega_{V,B}$, we have for all $\xi\in(0,3)$ and for all
%$t$, if $H=\psi(\p)$
%&\leq& \sum_{m=1}^M \int_0^t \exp(\xi)\lambda_s^{(m)}ds=\exp(\xi)
%Similar computations hold of course with $H=-\psi(\p)$ and with $ \exp(
%satisfied with
%$$w=\frac{B_\p^2x}{\mu-\phi(\mu)}\mbox{ and } v=\frac{B_\p^2x}{\mu-
%which leads to (\ref{Coeffcontrol}) on $\Omega_{V,B}\cap\Omega_c$.

%First we apply Theorem~\ref{th:weakBer} with $\tau=T$ with
%$$ H = \psi(\p) \indic_{\sup_{m,s\leq t}\left| \psi(\p)^{(m)}_s \right|\leq B_\phi,
%which is still predictable. \red{la je dois verifier un truc ....}
%Note that for all $\xi\in(0,3)$ and for all $t$
%$$
%$$
%and similarly for $\exp(\xi H^2/B^2)$.
%
%Next we note that on $\Omega_{V,B}$, $H\bullet(N-\Lambda)_T= b_\p-\bar
%b_\p$. Hence the control of Theorem~\ref{th:weakBer} applies to $b_\p-
%obtain the result.

Let us first define
%$$\mathcal{T}=\{t / \sup_m\left| \psi(\p)^{(m)}_t \right| > B_\p\mbox{ or } (\psi(
%
%e7.1 #&#
\begin{equation}
\label{monT} \mathcal{T}=\Bigl\{t\geq0 \Bigl/ \sup_m\bigl
\llvert \psi^{(m)}_t(\p) \bigr\rrvert > B_\p
\Bigr\}.
\end{equation}
Let us define the stopping time $\tau'=\inf\mathcal{T}$ and the
predictable process $H$ by
\[
H_t^{(m)}=\psi^{(m)}_t(\p)
\indic_{t\leq\tau'}.
\]
Let us apply Theorem~\ref{th:weakBer} to this choice of $H$ with $\tau
=T$ and $B=B_\p$. The choice of $v$ and $w$ will be given later on. To
apply this result, we need to check that
for all $t$ and all $\xi\in(0,3)$, $\sum_m \int_0^t \mathrm{e}^{\xi\sfrac
{H_s^{(m)}}{B_\p}} \lambda^{(m)}_s \mathrm{d}s$ is a.s. finite.
But if $t>\tau'$, then
\[
\int_0^t \mathrm{e}^{\xi\sfrac{H_s^{(m)}}{B_\p}}
\lambda^{(m)}_s \,\mathrm{d}s = \int_0^{\tau'}
\mathrm{e}^{\xi\sfrac{H_s^{(m)}}{B_\p}} \lambda^{(m)}_s \,\mathrm {d}s + \int
_{\tau'}^t \lambda^{(m)}_s
\,\mathrm{d}s,
\]
where the second part is obviously finite (it is just $\Lambda
^{(m)}_t-\Lambda^{(m)}_{\tau'}$).
Hence, it remains to prove that for all $t\leq\tau'$,
\[
\int_0^t \mathrm{e}^{\xi\sfrac{H_s^{(m)}}{B_\p}}
\lambda^{(m)}_s \,\mathrm{d}s
\]
is finite.
But for all $s<t$, $s<\tau'$ and consequently $s\notin\mathcal{T}$.
Therefore, $\llvert  H_s^{(m)} \rrvert \leq B_\p$. Since we are integrating with
respect to the Lebesgue measure, the fact that it eventually does not
hold in $t$ is not a problem and
\[
\int_0^t \mathrm{e}^{\xi\sfrac{H_s^{(m)}}{B_\p}}
\lambda^{(m)}_s \,\mathrm{d}s\leq \mathrm{e}^{\xi}
\Lambda^{(m)}_t,
\]
which is obviously finite a.s.
The same reasoning can be applied to show that a.s. $\exp(\xi
H^2/B^2)\bullet\Lambda_t<\infty$. We can also apply Theorem~\ref
{th:weakBer} to $-H$ in the same way. We obtain at the end that for all
$\e>0$
%
%e7.2 #&#
\begin{eqnarray}
\label{onephi1} &&\hspace*{-30pt}\P \biggl(\bigl\llvert H\bullet(N-\Lambda)_T \bigr
\rrvert \geq\sqrt{2 (1+\e)\hat{V}^\mu x} + \frac{B_\p x}{3} \mbox{
and } w\leq\hat{V}^\mu\leq v \mbox{ and } \sup_{m,t\leq T}
\bigl\llvert H^{(m)}_t \bigr\rrvert \leq B_\p
\biggr)\nonumber
\\[-8pt]\\[-8pt]
&&\hspace*{-30pt}\quad \leq4 \biggl(\frac{\log
(v/w)}{\log(1+\e)}+1 \biggr) \mathrm{e}^{-x}.\nonumber
\end{eqnarray}
But on $\Omega_{V,B}$ it is clear that $\forall t\in[0,T], t\notin
\mathcal{T}$. Therefore, $\tau'\geq T$. Therefore for all $t\leq T$,
one also has $t\leq\tau'$ and $H_t^{(m)}=\psi_t^{(m)}(\p)$.
Consequently, on $\Omega_{V,B}$,
\[
H\bullet(N-\Lambda)_T=b_\p-\bar{b}_\p
\quad \mbox{and}\quad  \hat{V}^\mu= \hat {V}^\mu_\p.
\]
Moreover, on $\Omega_{V,B}$, one has that
\[
\frac{B_\p^2 x}{\mu-\phi(\mu)}\leq\hat{V}^\mu_\p\leq
\frac{\mu}{\mu
-\phi(\mu)} V_\p+\frac{B_\p^2 x}{\mu-\phi(\mu)}.
\]
So, we take $w$ and $v$ as respectively the left- and right-hand side
of the previous inequality. Finally note that on $\Omega_{V,B}$,
\[
\sup_{m,t\leq T} \bigl\llvert H^{(m)}_t \bigr
\rrvert = \sup_{m,t\leq T} \bigl\llvert \psi^{(m)}_t(
\p) \bigr\rrvert \leq B_\p.
\]
Hence, we can rewrite \eqref{onephi1} as follows
%
%e7.3 #&#
\begin{eqnarray}
\label{onephi2}  \hspace*{-15pt}\P \biggl(\llvert b_\p-\bar{b}_\p
\rrvert \geq\sqrt{2 (1+\e)\hat{V}_\p^\mu x} +
\frac
{B_\p x}{3} \mbox{ and } \Omega_{V,B} \biggr)  \leq4 \biggl(
\frac{\log
(1+\sfrac{\mu V_\p}{B_\p^2x} )}{\log(1+\e)}+1 \biggr) \mathrm{e}^{-x}.
\end{eqnarray}
Apply this to all $\p\in\Phi$, we obtain that
\[
\P \bigl(\exists \p\in\Phi\mbox{ s.t. }\llvert b_\p-
\bar{b}_\p \rrvert \geq d_\p\mbox{ and }
\Omega_{V,B} \bigr) \leq4\sum_{\p\in\Phi} \biggl(
\frac{\log
(1+\sfrac{\mu V_\p}{B_\p^2x} )}{\log(1+\e)}+1 \biggr) \mathrm{e}^{-x}.
\]
Now on the event $\Omega_c \cap\Omega_{V,B}\cap\{\forall\p\in\Phi,
\llvert  b_\p-\bar{b}_\p \rrvert \leq d_\p\}$, one can apply Theorem~\ref{th:oracle}.
To obtain Theorem~\ref{th:oraclewithd}, it remains to bound the
probability of the complementary event by
\[
\P\bigl(\Omega_c^c\bigr)+\P\bigl(\Omega_{V,B}^c
\bigr)+\P \bigl(\exists \p\in\Phi\mbox{ s.t. } \llvert b_\p-
\bar{b}_\p \rrvert \geq d_\p\mbox{ and }
\Omega_{V,B} \bigr).
\]

%%%%%%%%%%%%%%%%%%%%%%%%%%%%%%%%%%
%%%%%%%%%%%%%%%%%%%%%%%%%%%%%%%%%%
%s7.3 #&#
\subsection{Proof of Theorem \texorpdfstring{\protect\ref{th:weakBer}}{3}}\label{sec:weakBer}
First, replacing $H$ with $H/B$, we can always assume that $B=1$.
Next, let us fix for the moment $\xi\in(0,3)$.
If one assumes that almost surely for all $t>0$, $\sum_{m=1}^M\int_0^t
\mathrm{e}^{\xi H^{(m)}_s} \lambda^{(m)}_s \mathrm{d}s <\infty$ (i.e., that the process
$\mathrm{e}^{\xi H}\bullet\Lambda$ is well defined) then one can apply Theorem~2 of \cite{Bre}, page 165, stating that the process $(E_t)_{t\geq0}$
defined for all $t$ by
\[
E_t=\exp\bigl(\xi H\bullet(N-\Lambda)_t - \phi(\xi H)
\bullet\Lambda_t\bigr)
\]
is a supermartingale. It is also the case for $E_{t\wedge\tau}$ if
$\tau$ is a bounded stopping time. Hence for any $\xi\in(0,3)$ and for
any $x>0$, one has that
\[
\P\bigl(E_{t\wedge\tau}>\mathrm{e}^x\bigr)\leq \mathrm{e}^{-x}
\E(E_{t\wedge\tau})\leq \mathrm{e}^{-x},
\]
which means that
\[
\P\bigl(\xi H\bullet(N-\Lambda)_{t\wedge\tau} - \phi(\xi H)\bullet\Lambda
_{t\wedge\tau} >x\bigr)\leq \mathrm{e}^{-x}.
\]
Therefore,
\[
\P\Bigl(\xi H\bullet(N-\Lambda)_{t\wedge\tau} - \phi(\xi H)\bullet\Lambda
_{t\wedge\tau}>x \mbox{ and } \sup_{s\leq\tau,m}\bigl\llvert
H^{(m)}_s \bigr\rrvert \leq1\Bigr)\leq \mathrm{e}^{-x}.
\]
But if $ \sup_{s\leq\tau,m}\llvert  H^{(m)}_s \rrvert \leq1$, then for any $\xi>0$ and
any $s$,
\[
\phi\bigl(\xi H^{(m)}_s\bigr) \leq\bigl(H^{(m)}_s
\bigr)^2 \phi(\xi).
\]
So, for every $\xi\in(0,3)$, we obtain:
%
%e7.4 #&#
\begin{equation}
\label{start} \P \Bigl(M_\tau\geq\xi^{-1}\phi(\xi)
H^2\bullet\Lambda_\tau+ \xi^{-1} x\mbox{ and }
\sup_{s\leq\tau,m}\bigl\llvert H^{(m)}_s \bigr
\rrvert \leq1 \Bigr) \leq \mathrm{e}^{-x}.
\end{equation}
%
%So, first let us note that if $H^2\bullet\Lambda_\tau$ is almost
%surely bounded by some deterministic quantity $v$, we have that for
%any $\xi>0$
Now let us focus on the event $H^2\bullet\Lambda_\tau\leq v$ where $v$
is a deterministic quantity. We have that consequently
\[
\P \Bigl(M_\tau\geq\xi^{-1}\phi(\xi)v + \xi^{-1}
x\mbox{ and } H^2\bullet\Lambda_\tau\leq v \mbox{ and }
\sup_{s\leq\tau,m}\bigl\llvert H^{(m)} _s \bigr
\rrvert \leq1 \Bigr) \leq \mathrm{e}^{-x}.
\]
It remains to choose $\xi$ such that $\xi^{-1}\phi(\xi)v + \xi^{-1} x$
is minimal. But this expression has no simple form. However, since
$0<\xi<3$, one can bound $\phi(\xi)$ by $\xi^2(1-\xi/3)^{-1}/2$. Hence,
we can start with
%
%e7.5 #&#
\begin{equation}
\label{start2} \P \biggl(M_\tau\geq\frac{\xi}{2(1-\xi/3)}H^2
\bullet\Lambda_\tau+ \xi ^{-1} x \mbox{ and }\sup
_{s\leq\tau,m}\bigl\llvert H^{(m)}_s \bigr
\rrvert \leq1 \biggr) \leq \mathrm{e}^{-x}
\end{equation}
and also
%its associated upper bound if we assume that $H^2\bullet\Lambda_\tau$
%is almost surely bounded by some deterministic quantity $v$:
%
%e7.6 #&#
\begin{equation}
\label{startdet} \P \biggl(M_\tau\geq\frac{\xi}{2(1-\xi/3)}v +
\xi^{-1} x\mbox{ and } H^2\bullet\Lambda_\tau\leq
v \mbox{ and } \sup_{s\leq\tau,m}\bigl\llvert H^{(m)}
_s \bigr\rrvert \leq1 \biggr) \leq \mathrm{e}^{-x}.
\end{equation}
It remains now to minimize $\xi\longmapsto\frac{\xi}{2(1-\xi/3)}v + \xi
^{-1} x$.

%le2 #&#
\begin{Lemma}\label{Lemme:Bernstein}
Let $a$, $b$ and $x$ be positive constants and let us consider on $(0, 1/b)$,
\[
g(\xi)=\frac{a\xi}{(1-b\xi)}+\frac{x}{\xi}.
\]
Then $\min_{\xi\in(0,1/b)} g(\xi)= 2 \sqrt{ax} + bx$ and the minimum
is achieved in
$\xi(a,b,x)=\frac{xb-\sqrt{ax}}{xb^2-a}$.
\end{Lemma}

\begin{pf}
The limits of $g$ in $0^+$ and $(1/b)^-$ are $+\infty$. The derivative
is given by
\[
g'(\xi)=\frac{a}{(1-b\xi)^2}-\frac{x}{\xi^2}
\]
which is null in $\xi(a,b,x)$ (remark that the other solution of the
polynomial does not lie in $(0,1/b)$). Finally, it remains to evaluate
the quantity in $\xi(a,b,x)$ to obtain the result.
\end{pf}

Now, we apply \eqref{startdet} with $\xi(v/2,1/3,x)$ and we obtain this
well known formula which can be found in \cite{SW} for instance,
%
%e7.7 #&#
\begin{equation}
\label{sho} \P \Bigl(M_\tau\geq\sqrt{2vx} + x/3\mbox{ and }
H^2\bullet\Lambda_\tau \leq v \mbox{ and } \sup
_{s\leq\tau,m}\bigl\llvert H^{(m)}_s \bigr
\rrvert \leq1 \Bigr) \leq \mathrm{e}^{-x}.
\end{equation}
Now we would like first to replace $v$ by its random version $H^2\bullet
\Lambda_\tau$. Let $w,v$ be some positive constants and let us
concentrate on the event
%
%e7.8 #&#
\begin{equation}
\label{hypW} w\leq H^2\bullet\Lambda_\tau\leq v.
\end{equation}
For all $\e>0$ we introduce $K$ a positive integer depending on $\e$,
$v$ and $w$ such that $(1+\e)^K w\geq v$. Note that $K=\lceil\log
(v/w)/\log(1+\e)\rceil$ is a possible choice.
Let us denote $v_0=w$, $v_1=(1+\e)w, \ldots, v_K=(1+\e)^K w$.
For any $0<\xi<3$ and any $k$ in $\{0,\ldots,K-1\}$, one has, by applying
\eqref{start2},
\begin{eqnarray*}
&&\P \biggl(M_\tau\geq\frac{\xi}{2(1-\xi/3)}H^2\bullet
\Lambda_\tau+ \xi ^{-1} x \\
&&\hphantom{\P \biggl(}\mbox{and } v_k \leq
H^2\bullet\Lambda_\tau\leq v_{k+1} \mbox{ and }
\sup_{s\leq
\tau,m}\bigl\llvert H^{(m)}_s \bigr
\rrvert \leq1 \biggr) \leq \mathrm{e}^{-x}.
\end{eqnarray*}
This implies that
\[
\P \biggl(M_\tau\geq\frac{\xi}{2(1-\xi/3)} v_{k+1} +
\xi^{-1} x \mbox{ and } v_k \leq H^2\bullet
\Lambda_\tau\leq v_{k+1}\mbox{ and } \sup
_{s\leq
\tau,m}\bigl\llvert H^{(m)}_s \bigr
\rrvert \leq1 \biggr) \leq \mathrm{e}^{-x}.
\]
Using Lemma~\ref{Lemme:Bernstein}, with $\xi=\xi(v_{k+1}/2,1/3,x)$,
this gives
\[
\P \Bigl(M_\tau\geq\sqrt{2 v_{k+1} x} + x/3\mbox{ and }
v_k \leq H^2\bullet\Lambda_\tau\leq
v_{k+1}\mbox{ and } \sup_{s\leq
\tau,m}\bigl\llvert
H^{(m)}_s \bigr\rrvert \leq1 \Bigr) \leq \mathrm{e}^{-x}.
\]
But if $v_k \leq H^2\bullet\Lambda_\tau$, $v_{k+1}\leq(1+\e) v_k\leq
(1+\e) H^2\bullet\Lambda_\tau$, so
\begin{eqnarray*}
&&\P \Bigl(M_\tau\geq\sqrt{2 (1+\e) \bigl(H^2\bullet
\Lambda_\tau \bigr) x} + x/3\\
&&\hphantom{\P \biggl(}\mbox{and } v_k \leq
H^2\bullet\Lambda_\tau \leq v_{k+1}\mbox{ and }
\sup_{s\leq
\tau,m}\bigl\llvert H^{(m)}_s \bigr
\rrvert \leq1 \Bigr) \leq \mathrm{e}^{-x}.
\end{eqnarray*}
Finally summing on $k$, this gives
%
%e7.9 #&#
\begin{eqnarray}
\label{bracket1} &&\P \Bigl(M_\tau\geq\sqrt{2 (1+\e)
\bigl(H^2\bullet\Lambda_\tau\bigr) x} + x/3 \nonumber\\[-8pt]\\[-8pt]
&&\hphantom{\P \Bigl(}\mbox{and } w
\leq H^2\bullet\Lambda_\tau\leq v
\mbox { and } \sup_{s\leq\tau,m}\bigl\llvert
H^{(m)}_s \bigr\rrvert \leq1 \Bigr) \leq K
\mathrm{e}^{-x}.\nonumber
\end{eqnarray}
This leads to the following result that has interest per se.

%pr6 #&#
\begin{Prop}
\label{weakberLambda}
Let $N=(N^{(m)})_{m=1,\ldots,M}$ be a multivariate counting process with
predictable intensities $\lambda^{(m)}_t$ and corresponding compensator
$\Lambda^{(m)}_t$ with respect to some given filtration.
Let $B>0$.
Let $H=(H^{(m)})_{m=1,\ldots,M}$ be a multivariate predictable process such
that for all $\xi\in(0,3)$, $\mathrm{e}^{\xi H/B}\bullet\Lambda_t<\infty$ a.s.
for all $t$. Let us consider the martingale defined for all $t$ by
\[
M_t=H\bullet(N-\Lambda)_t.
\]
Let $v>w$ be positive constants and let $\tau$ be a bounded stopping time.
Then for any $\e,x>0$
%
%e7.10 #&#
\begin{eqnarray}
\label{bracket}&&\hspace*{-30pt} \P \biggl(M_\tau\geq\sqrt{2 (1+\e)
\bigl(H^2\bullet\Lambda_\tau\bigr) x} + \frac{Bx}{3}
\mbox{ and } w\leq H^2\bullet\Lambda_\tau\leq v
 \mbox{ and } \sup_{m,t\leq\tau} \bigl\llvert
H^{(m)}_t \bigr\rrvert \leq B \biggr) \nonumber\\[-8pt]\\[-8pt]
&&\hspace*{-30pt}\quad \leq \biggl(
\frac{\log(v/w)}{\log(1+\e)}+1 \biggr) \mathrm{e}^{-x}.\nonumber
\end{eqnarray}
\end{Prop}

Next, we would like to replace $H^2\bullet\Lambda_\tau$, the quadratic
characteristic of $M$, with its estimator $H^2\bullet N_\tau$, that is,
the quadratic variation of $M$. For this purpose, let us consider
$W_t=-H^2\bullet(N-\Lambda)_t$ which is still a martingale since the
$-(H^{(m)}_s)^2$'s are still predictable processes. We apply \eqref{start}
with $\mu$ instead of $\xi$, noticing that on the event $\{\sup_{s\leq
\tau,m}\llvert  H^{(m)}_s \rrvert \leq1\}$, one has that $H^4\bullet\Lambda_\tau\leq
H^2\bullet\Lambda_\tau$. This gives that
\[
\P \Bigl(H^2\bullet\Lambda_\tau\geq H^2
\bullet N_\tau+ \bigl\{\phi(\mu)/\mu \bigr\} H^2\bullet
\Lambda_\tau+x/\mu\mbox{ and } \sup_{s\leq\tau,m}\bigl
\llvert H^{(m)} _s \bigr\rrvert \leq1 \Bigr)\leq
\mathrm{e}^{-x},
\]
which means that
%
%e7.11 #&#
\begin{equation}
\label{andV} \P \Bigl(H^2\bullet\Lambda_\tau\geq
\hat{V}^\mu\mbox{ and } \sup_{s\leq\tau,m}\bigl\llvert
H^{(m)}_s \bigr\rrvert \leq1 \Bigr)\leq \mathrm{e}^{-x}.\vadjust{\goodbreak}
\end{equation}
So we use again \eqref{start2} combined with \eqref{andV} to obtain
that for all $\xi\in(0,3)$
\begin{eqnarray*}
&&\P \biggl(M_\tau\geq\frac{\xi}{2(1-\xi/3)}\hat{V}^\mu+
\xi^{-1} x\mbox { and } \sup_{s\leq\tau,m}\bigl\llvert
H^{(m)}_s \bigr\rrvert \leq1 \biggr)
\\
 &&\quad \leq\P \biggl(M_\tau\geq\frac{\xi}{2(1-\xi/3)}\hat{V}^\mu+
\xi^{-1} x\mbox { and } \sup_{s\leq\tau,m}\bigl\llvert
H^{(m)}_s \bigr\rrvert \leq1 \mbox{ and } H^2
\bullet\Lambda _\tau\leq\hat{V}^\mu \biggr)
\\
&&\qquad {}+ \P \Bigl(H^2\bullet\Lambda_\tau\geq\hat{V}^\mu
\mbox{ and } \sup_{s\leq\tau,m}\bigl\llvert H^{(m)}_s
\bigr\rrvert \leq1 \Bigr) \leq 2\mathrm{e}^{-x}.
\end{eqnarray*}
This new inequality replaces \eqref{start2} and it remains to replace
$H^2\bullet\Lambda_\tau$ by $\hat{V}^\mu$ in the peeling arguments to
obtain as before that
%Next, arguments based on \red{peeling} allow to obtain as before that
%
%e7.12 #&#
\begin{equation}
\label{bracket2} \P \Bigl(M_\tau\geq\sqrt{2 (1+\e)
\hat{V}^\mu x} + x/3 \mbox{ and } w\leq\hat{V}^\mu\leq v
\mbox{ and } \sup_{s\leq\tau,m}\bigl\llvert H^{(m)}_s
\bigr\rrvert \leq1 \Bigr) \leq2K \mathrm{e}^{-x}.
\end{equation}

%%%%%%%%%%%%%%%%%%%%%%
%%%%%%%%%%%%%%%%%%%%%%
%s7.4 #&#
\subsection{Proofs of the probabilistic results for Hawkes processes}
%s7.4.1 #&#
\subsubsection{Proof of Lemma \texorpdfstring{\protect\ref{lem:expmoment}}{1}}
Let $K(n)$ denote the vector of the number of descendants in the
$n$th generation from a single ancestral point of type $\ell$,
define $K(0) = \mathbf{e}_{\ell}$ and
let $W(n) = \sum_{k=0}^n K(k)$ denote the total number of
points in the first $n$ generations. Define for $\theta\in\R^M$
\[
\phi_{\ell}(\theta) = \log\E_{\ell} \mathrm{e}^{\theta^T K(1)}.
\]
Thus, $\phi_{\ell}(\theta)$ is the log-Laplace transform of the
distribution of $K(1)$ given that there is a single initial ancestral
point of type $\ell$. We define the vector $\phi(\theta)$ by $\phi
(\theta)'=(\phi_1(\theta),\ldots,\phi_M(\theta))$. Note that $\phi$ only
depends on the law of the number of children per parent, that is, it
only depends on $\Gamma$. Then
\begin{eqnarray*}
\E_{\ell} \mathrm{e}^{\theta^T W(n)} & = & \E_{\ell} \bigl(
\mathrm{e}^{\theta^T W(n-1)} \E \bigl( \mathrm{e}^{\theta^T K(n)} \mid K(n-1) ,\ldots, K(1) \bigr)
\bigr)
\\
& = & \E_{\ell} \bigl( \mathrm{e}^{\theta^T W(n-1)} \mathrm{e}^{ \phi(\theta)^T K(n-1)} \bigr)
\\
& = & \E_{\ell} \mathrm{e}^{(\theta+ \phi(\theta))^T K(n-1) + \theta^T W(n-2)}.
\end{eqnarray*}
Defining $g(\theta) = \theta+ \phi(\theta)$ we arrive by recursion at the
formula
\begin{eqnarray*}
\E_{\ell} \mathrm{e}^{\theta^T W(n)} & = & \E_{\ell} \mathrm{e}^{g^{\circ(n-1)}(\theta)^T
K(1) + \theta^T W(0)}
\\
& = & \mathrm{e}^{\phi(g^{\circ(n-1)}(\theta))_{\ell} + \theta_{\ell}}
\\
& = & \mathrm{e}^{g^{\circ n}(\theta)_{\ell}},
\end{eqnarray*}
where for any $n$, $g^{\circ(n)}=g\circ\cdots\circ g$ $n$ times. Or,
in other words, we have the following representation
\[
\log\E_{\ell} \mathrm{e}^{\theta^T W(n)} = g^{\circ n}(\theta)_{\ell}
\]
of the log-Laplace transform of $W(n)$.

Below we show that $\phi$ is a contraction in a neighborhood
containing 0, that is, for some $r > 0$ and a constant $C < 1$ (and a
suitable norm), $\llVert  \phi(s) \rrVert   \leq C\llVert  s \rrVert  $ for $\llVert  s \rrVert   \leq r$. If
$\theta$ is chosen such that
\[
\frac{\llVert \theta\rrVert }{1-C} \leq r
\]
we have $\llVert  \theta \rrVert   \leq r$, and if we assume that
$g^{\circ k}(\theta) \in B(0, r)$ for $k = 1, \ldots, n-1$ then
\begin{eqnarray*}
\bigl\llVert g^{\circ n}(\theta) \bigr\rrVert & \leq& \llVert \theta
\rrVert + \bigl\llVert \phi\bigl(g^{\circ
(n-1)}(\theta)\bigr) \bigr\rrVert
\\
& \leq& \llVert \theta \rrVert + C \bigl\llVert g^{\circ(n-1)}(\theta) \bigr
\rrVert
\\
& \leq& \llVert \theta \rrVert \bigl( 1 + C + C^2 + \cdots+
C^n \bigr)
\\
& \leq& r.
\end{eqnarray*}
Thus, by induction, $g^{\circ n}(\theta) \in B(0, r)$ for all $n \geq
1$.
Since $n\mapsto W_m(n)$ is increasing and goes to $W_m(\infty)$ for $n
\to\infty$, with
$W_m(\infty)$ the total number of points in a cluster of type $m$, and
since $W = \sum_{m} W_m(\infty) = \mathbf{1}^T W(\infty)$, we have by
monotone convergence that for $\vartheta\in\mathbb{R}$
\[
\log\E_{\ell} \mathrm{e}^{\vartheta W} = \lim_{n \to\infty}
g^{\circ n}(\vartheta\mathbf{1})_{\ell}.
\]
By the previous result, the right-hand side is bounded if
$\llvert  \vartheta \rrvert $ is sufficiently small. This completes the proof up to
proving that $\phi$ is a contraction.

To this end, we note that $\phi$ is continuously differentiable (on
$\R^M$ in fact, but a neighborhood around~0 suffices) with derivative
$D\phi(0) = \Gamma$ at 0. Since the spectral radius of $\Gamma$ is
strictly less than 1 there is a $C < 1$ and, by the Householder
theorem, a norm
$\llVert  \cdot \rrVert  $ on $\R^M$ such that for the induced operator norm
of $\Gamma$ we have
\[
\llVert \Gamma \rrVert = \max_{x: \llVert  x \rrVert   \leq1} \llVert \Gamma x \rrVert
< C.
\]
Since the norm is continuous and $D\phi(s)$ is likewise there is
an $r > 0$ such that
\[
\bigl\llVert D\phi(s) \bigr\rrVert \leq C < 1
\]
for $\llVert  s \rrVert   \leq r$. This, in turn, implies that $\phi$ is
Lipschitz continuous in the ball $B(0,r)$ with Lipschitz constant $C$,
and since $\phi(0) = 0$ we get
\[
\bigl\llVert \phi(s) \bigr\rrVert \leq C \llVert s \rrVert
\]
for $\llVert  s \rrVert   \leq r$.
This ends the proof of the lemma.

Note that we have not at all used the explicit formula for $\phi$
above, which is obtainable and simple since the offspring
distributions are Poisson. The only thing we needed was the fact that
$\phi$ is defined in a neighborhood around 0, thus that the offspring
distributions are sufficiently light-tailed.

%s7.4.2 #&#
\subsubsection{Proof of Proposition \texorpdfstring{\protect\ref{momentsHawkes}}{2}}

We use the cluster representation, and we note that
any cluster with ancestral point in $[-n-1,-n]$ must have at least
$n+1-\lceil A\rceil$
points in the cluster if any of the points are to fall in
$[-A,0)$. This follows from the assumption that all the
$h_{\ell}^{(m)}$-functions have support in $[0,1]$. With
$\tilde{N}_{A, \ell}$ the number of points in $[-A,0)$ from a cluster
with ancestral points of type $\ell$, we thus have the bound\vspace*{1pt}
\[
\tilde{N}_{A, \ell} \leq\sum_{n} \sum
_{k=1}^{A_n} \max\bigl\{W_{n,k} - n
+ \lceil A\rceil, 0\bigr\},
\]
where $A_n$ is the number of ancestral points in $[-n-1,-n]$ of type
$\ell$ and $W_{n,k}$ is the number of points in the respective
clusters. Here
the $A_n$'s and the $W_{n,k}$'s are all independent, the $A_n$'s
are Poisson distributed with mean $\nu_{\ell}$ and the $W_{n,k}$'s are
i.i.d. with the same distribution as $W$ in Lemma~\ref{lem:expmoment}. Moreover,
\[
H_n(\vartheta_\ell) := \E_{\ell}
\mathrm{e}^{\vartheta_\ell\max\{W - n + \lceil
A\rceil, 0\}} \leq \P_{\ell}\bigl(W \leq n - \lceil A\rceil\bigr) +
\mathrm{e}^{-\vartheta_\ell(n-\lceil
A\rceil)} \E_{\ell} \mathrm{e}^{\vartheta_\ell W},
\]
which is finite for $\llvert  \vartheta_\ell \rrvert $ sufficiently small according to
Lemma~\ref{lem:expmoment}. Then we can compute an upper bound on the
Laplace transform of $\tilde{N}_{A,\ell}$:\vspace*{1pt}
\begin{eqnarray*}
\E \mathrm{e}^{\vartheta_\ell\tilde{N}_{A,\ell}} & \leq& \prod_n \E\prod
_{k=1}^{A_n} \E \bigl( \mathrm{e}^{\vartheta_\ell\max\{W_{n,k} - n + \lceil A\rceil, 0\}} \mid
A_n \bigr)
\\
& \leq& \prod_n \E H_n(
\vartheta_\ell)^{A_n}
\\
& = & \prod_n \mathrm{e}^{\nu_{\ell}(H_n(\vartheta_\ell)-1)}
\\
& = & \mathrm{e}^{\nu_{\ell} \sum_{n} (H_n(\vartheta_\ell)-1)}.
\end{eqnarray*}
Since $H_n(\vartheta_\ell)-1 \leq \mathrm{e}^{-\vartheta_\ell(n-\lceil A\rceil
)} \E_{\ell} \mathrm{e}^{\vartheta_\ell W}$
we have $\sum_{n} (H_n(\vartheta_\ell)-1) < \infty$, which shows that the
upper bound is finite. To complete the proof, observe that $N_{[-A,0)}=
\sum_{\ell}
\tilde{N}_{A, \ell}$ where $\tilde{N}_{A, \ell}$ for $\ell= 1,
\ldots, M$ are independent. Since all variables are positive, it is
sufficient to take $\theta=\min_\ell\vartheta_\ell$.

%s7.4.3 #&#
\subsubsection{Proof of Proposition \texorpdfstring{\protect\ref{flow}}{3}}
In this paragraph, the notation $\square$ simply denotes a generic positive
absolute constant that may change from line to line. The notation
$\square_{\theta_1,\theta_2,\ldots}$ denotes a positive constant
depending on $\theta_1,\theta_2,\ldots$ that may change from line to line.

Let
%
%e7.13 #&#
\begin{equation}
\label{u} u=C_1\sigma\log^{3/2}(T)\sqrt{T}
+C_2 b \bigl(\log(T)\bigr)^{2+\eta},
\end{equation}
where the choices of $C_1$ and $C_2$ will be given later. For any
positive integer $k$ such that $x:=T/(2k)>A$, we have by stationarity:
\begin{eqnarray*}
\P \biggl(\int_0^T \bigl[Z\circ
\mathfrak{S}_t(N) - \E(Z)\bigr] \,\mathrm{d}t \geq u \biggr)&=& \P
\Biggl(\sum_{q=0}^{k-1}\int
_{2qx}^{2qx+x} \bigl[Z\circ\mathfrak{S}_t(N)
- \E(Z)\bigr] \,\mathrm{d}t
\\
&& {} + \int_{2qx+x}^{2qx+2x} \bigl[Z\circ
\mathfrak{S}_t(N) - \E(Z)\bigr]\, \mathrm{d}t \geq u \Biggr)
\\
&\leq& 2 \P \Biggl(\sum_{q=0}^{k-1}\int
_{2qx}^{2qx+x} \bigl[Z\circ\mathfrak
{S}_t(N) - \E(Z)\bigr] \,\mathrm{d}t\geq\frac{u}{2} \Biggr).
\end{eqnarray*}
Similarly to \cite{RBRoy}, we introduce $(\tilde M_q^x)_q$ a sequence
of independent Hawkes processes, each being stationary with intensities
per mark given by $\psi_t^{(m)}$. For each $q$, we then introduce
$M_q^x$ the truncated process associated with $\tilde M_q^x$, where
truncation means that we only consider the points lying in $[2qx -A,
2qx+x]$. So, if we set
%
%e7.14 #&#
\begin{eqnarray}\label{split}
F_q&=&\int_{2qx}^{2qx+x} \bigl[Z\circ
\mathfrak{S}_t\bigl(M_q^x\bigr) - \E(Z)
\bigr] \,\mathrm{d}t,
\nonumber\\[-8pt]\\[-8pt]
 \P \biggl(\int_0^T \bigl[Z\circ
\mathfrak{S}_t(N) - \E(Z)\bigr] \,\mathrm{d}t \geq u \biggr)&\leq&2 \P
\Biggl(\sum_{q=0}^{k-1} F_q\geq
\frac{u}{2} \Biggr)+2k\P \biggl(T_e> \frac{T}{2k} -A
\biggr),\nonumber
\end{eqnarray}
where $T_e$ represents the time to extinction of the process. More
precisely, $T_e$ is the last point of the process if in the cluster
representation only ancestral points before 0 are appearing. For more
details, see Section~3 of \cite{RBRoy}. So, denoting $a_{l}$ the
ancestral points with marks $l$ and $H_{a_{l}}^l$ the length of the
corresponding cluster whose origin is $a_{l}$, we have:
\[
T_e=\max_{l\in\{1,\ldots,M\}}\max_{a_{l}}
\bigl\{a_{l}+H_{a_{l}}^l \bigr\}.
\]
But, for any $a>0$,
\begin{eqnarray*}
\P(T_e\leq a)&=&\E \Biggl[\prod_{l=1}^M
\prod_{a_{l}}\E [1_{\{
a_{l}+H_{a_{l}}^l\leq a\}}| a_{l} ]
\Biggr]
\\
&=&\E \Biggl[\prod_{l=1}^M\prod
_{a_{l}}\exp \bigl(\log \bigl(\P\bigl(H_0^l
\leq a-a_l\bigr) \bigr) \bigr) \Biggr]
\\
&=&\E \Biggl[\prod_{l=1}^M\exp \biggl(
\int_{-\infty}^0\log\bigl(\P\bigl(H_0^l
\leq a-x\bigr)\bigr)\,\mathrm{d}\tilde N_x^{(l)} \biggr)
\Biggr],
\end{eqnarray*}
where $\tilde N^{(l)}$ denotes the process associated with the
ancestral points with marks $l$. So,
\begin{eqnarray*}
\P(T_e\leq a) &=&\exp \Biggl(\sum_{l=1}^M
\int_{-\infty}^0 \bigl(\exp\bigl(\log\bigl(\P
\bigl(H_0^l\leq a-x\bigr)\bigr)\bigr)-1 \bigr)
\nu^{(l)}\,\mathrm{d}x \Biggr)
\\
&=&\exp \Biggl(-\sum_{l=1}^M
\nu^{(l)}\int_a^{+\infty}\P
\bigl(H_0^l>u\bigr)\,\mathrm {d}u \Biggr).
\end{eqnarray*}
Now, by Lemma~\ref{lem:expmoment}, there exists some $\vartheta_l>0$,
such that $c_l=\E_\ell(\mathrm{e}^{\vartheta_l W})<+\infty$, where $W$ is the
number of points in the cluster. But if all the interaction functions
have support in $[0,1]$, one always have that $H_0^l<W$. Hence,
\begin{eqnarray*}
\P\bigl(H_0^l>u\bigr)&\leq&\E\bigl[\exp\bigl(
\vartheta_lH_0^l\bigr)\bigr]\exp(-
\vartheta_lu)
\\
&\leq&c_l\exp(-\vartheta_l u).
\end{eqnarray*}
So,
\begin{eqnarray*}
\P(T_e\leq a)&\geq& \exp \Biggl(-\sum
_{l=1}^M \nu^{(l)}\int
_a^{+\infty
}c_l\exp(-
\vartheta_l u)\,\mathrm{d}u \Biggr)
\\
&=& \exp \Biggl(-\sum_{l=1}^M
\nu^{(l)}c_l/\vartheta_l\exp(-
\vartheta_l a) \Biggr)
\\
&\geq& 1-\sum_{l=1}^M
\nu^{(l)}c_l/\vartheta_l\exp(-
\vartheta_l a).
\end{eqnarray*}
So, there exists a constant $C_{\alpha,f_0,A}$ depending on $\alpha,
A$, and $f_0$ such that if we take $k=\lfloor C_{\alpha,A,f_0} T/\log
(T)\rfloor$, then
\[
k\P \biggl(T_e> \frac{T}{2k} -A \biggr)\leq
T^{-\alpha}.
\]
In this case $x=\frac{T}{2k} \approx\log(T)$ is larger than $A$ for $T$
large enough (depending on $A,\alpha, f_0$).

Now, let us focus on the first term $B$ of \eqref{split}, where
\[
B=\P \Biggl(\sum_{q=0}^{k-1}
F_q\geq\frac{u}{2} \Biggr).
\]
Let us consider some $\tilde{\mathcal{N}}$ where $\tilde{\mathcal{N}}$
will be fixed later and let us define the measurable events
\[
\Omega_q= \Bigl\{\sup_t \bigl
\{M_q^x|_{[t-A,t)}\bigr\} \leq\tilde{\mathcal
{N}} \Bigr\},
\]
where $M_q^x|_{[t-A,t)}$ represents the set of points of $M_q^x$ lying
in $[t-A,t)$.
Let us also consider $\Omega=\bigcap_{1\leq q\leq k} \Omega_q$.
Then
\[
B\leq\P \biggl(\sum_q F_q \geq
u/2 \mbox{ and } \Omega \biggr) + \P\bigl(\Omega^c\bigr).
\]
We have $\P(\Omega^c)\leq\sum_q\P(\Omega_q^c)$.
Each $\Omega_q$ can also be easily controlled. Indeed it is sufficient
to split $[2qx-A,2qx+x]$ in intervals of size $A$ (there are about
$\square_{\alpha,A,f_0} \log(T)$ of those) and require that the number
of points in each subinterval is smaller than $\tilde{\mathcal{N}}/2$.
By stationarity, we obtain that
\[
\P\bigl(\Omega_q^c\bigr) \leq\square_{\alpha,A,f_0}
\log(T)\P( N_{[-A,0)}>\tilde {\mathcal{N}}/2).
\]
%
%Since Lemma~\ref{momentsHawkes} shows that $N_{[-A,0)}$ has an
%exponential moment, we obtain:
Using Proposition~\ref{momentsHawkes} with $u=\lceil\tilde{\mathcal
{N}}/2\rceil+1/2$, we obtain:
%
%e7.15 #&#
\begin{eqnarray}
\label{controlN} \P\bigl(\Omega_q^c\bigr)&\leq&
\square_{\alpha,A,f_0} \log(T) \exp(-\square_{\alpha
,A,f_0} \tilde{\mathcal{N}})
\quad \mbox{and}\nonumber
\\[-8pt]
\\[-8pt]
  \P\bigl(\Omega^c\bigr)&\leq&\square _{\alpha,A,f_0} T
\exp(-\square_{\alpha,A, f_0} \tilde{\mathcal{N}}).
\nonumber
\end{eqnarray}
Note that this control holds for any positive choice of $\tilde{\mathcal
{N}}$. Hence, this gives also the following lemma that will be used later.

%le3 #&#
\begin{Lemma}
\label{withN}
For any $\mathcal R>0$,
\[
\P \bigl(\mbox{there exists } t \in[0,T] \mid M_q^x
| _{[t-A,t)} > \mathcal {R}\bigr)\leq\square_{\alpha,A,f_0} T
\exp(-\square_{\alpha,A,f_0} \mathcal{R}).
\]
\end{Lemma}

Hence by taking $\tilde{\mathcal{N}}=C_3\log(T)$ for $C_3$ large enough
this is smaller than $\square_{\alpha,A,f_0} T^{-\alpha'}$, where
$\alpha'=\max(\alpha,2)$.

It remains to obtain the rate of $D:=\P(\sum_q F_q \geq u/2 \mbox{ and
} \Omega)$.
For any positive constant $\theta$ that will be chosen later, we have:
%
%e7.16 #&#
\begin{eqnarray}
\label{majoD} D&\leq& \mathrm{e}^{-\sfrac{\theta u}{2}} \E \biggl(\mathrm{e}^{\theta\sum_q F_q} \prod
_q \indic_{\Omega_q} \biggr)
\nonumber
\\[-8pt]
\\[-8pt]
&\leq& \mathrm{e}^{-\sfrac{\theta u}{2}} \prod_q \E
\bigl(\mathrm{e}^{\theta F_q} \indic _{\Omega_q} \bigr)
\nonumber
\end{eqnarray}
since the variables $(M_q^x)_q$ are independent.
But
\[
\E \bigl(\mathrm{e}^{\theta F_q} \indic_{\Omega_q} \bigr)=1+\theta
\E(F_q \indic _{\Omega_q})+\sum_{j\geq2}
\frac{\theta^j}{j!}\E\bigl(F_q^j \indic_{\Omega_q}
\bigr)
\]
and $\E(F_q \indic_{\Omega_q})=\E(F_q)-\E(F_q \indic_{\Omega_q^c})=-\E
(F_q \indic_{\Omega_q^c})$.

Next note that
if for any integer $l$,
\[
l \tilde{\mathcal{N}} < \sup_t M_q^x|_{[t-A,t)}
\leq(l+1) \tilde {\mathcal{N}}
\]
then
\[
\llvert F_q \rrvert \leq x b \bigl[(l+1)^{\eta} \tilde{
\mathcal{N}}^\eta+1\bigr] + x \E(Z).
\]
Hence, cutting $\Omega_q^c$ in slices of the type $\{l \tilde{\mathcal
{N}} < \sup_t M_q^x{} _{[t-A,t)} \leq(l+1) \tilde{\mathcal{N}}\}$ and
using Lemma~\ref{withN}, we obtain by taking $C_3$ large enough,
\begin{eqnarray*}
\bigl\llvert \E(F_q \indic_{\Omega_q}) \bigr\rrvert =\bigl
\llvert \E(F_q \indic_{\Omega_q^c}) \bigr\rrvert &\leq& \sum
_{l=1}^{+\infty} x\bigl(b \bigl[(l+1)^{\eta}
\tilde{\mathcal{N}}^\eta+1\bigr] + \bigl\llvert \E(Z) \bigr\rrvert
\bigr)
\\
&&{}\times \P\bigl(\mbox{there exists } t \in[0,T] \mid \bigl\{ M_q^x|_{[t-A,t)}
\bigr\} > \ell\tilde{\mathcal{N}}\bigr)
\\
&\leq& \square_{\alpha,A,f_0}\sum_{l=1}^{+\infty}
x\bigl(b \bigl[(l+1)^{\eta} \tilde{\mathcal{N}}^\eta+1\bigr] +
\bigl|\E(Z)\bigr|\bigr) \log(T) \mathrm{e}^{-\square_{\alpha
,A,f_0} l \tilde{\mathcal{N}}}
\\
&\leq&\square_{\alpha,A,f_0}\sum_{l=1}^{+\infty}
x\bigl(b \tilde{\mathcal {N}}^\eta+ \bigl|\E(Z)\bigr|\bigr) \log(T)
2^{l\eta}\mathrm{e}^{-\square_{\alpha,A,f_0} l
\tilde{\mathcal{N}}}
\\
&\leq& \square_{\alpha,\eta,A,f_0} \log^2(T) b \tilde{
\mathcal{N}}^\eta \frac
{\mathrm{e}^{-\square_{\alpha,A,f_0} \tilde{\mathcal{N}}}}{1-2^\eta \mathrm{e}^{-\square
_{\alpha,A,f_0} \tilde{\mathcal{N}}}}
\\
&\leq& z_1:=\square_{\alpha,\eta,A,f_0} b T^{-\alpha'}.
\end{eqnarray*}
Note that in the previous inequalities, we have bounded $\llvert  \E(Z) \rrvert $ by
$b\E[N_{[-A,0)}^\eta]$. In the same way, one can bound
\[
\E\bigl(F_q^j \indic_{\Omega_q}\bigr)\leq\E
\bigl(F_q^2 \indic_{\Omega_q}\bigr)
z_b^{j-2},
\]
with $z_b:=xb[\tilde{\mathcal{N}}^\eta+1]+x\E(Z)=\square_{\alpha,\eta
,A,f_0} b \log(T)^{1+\eta}$.
One can also note that by stationarity,
\begin{eqnarray*}
\E\bigl(F_q^2 \indic_{\Omega_q}\bigr)& \leq& x \E
\biggl[ \int_{2qx}^{2qx+x} \bigl[Z\circ
\theta_s\bigl(M_q^x\bigr)-\E(Z)
\bigr]^2 \indic_{\{\mathrm{for\ all\ } t,
M_q^x|_{[t-A,t)} \leq\tilde{\mathcal{N}}\}} \,\mathrm{d}s \biggr]
\\
&\leq& x \E \biggl[ \int_{2qx}^{2qx+x} \bigl[Z\circ
\theta_s\bigl(M_q^x\bigr)-\E(Z)
\bigr]^2 \indic_{\{ M_q^x|_{[s-A,s)} \leq\tilde{\mathcal{N}}\}} \,\mathrm {d}s \biggr]
\\
&\leq& x^2 \E\bigl(\bigl[Z(N)-\E(Z)\bigr]^2
\indic_{N_{[-A,0)} \leq\tilde{\mathcal
{N}}}\bigr)
\\
& \leq& z_v:= \square_{\alpha,\eta,A,f_0}\bigl(\log(T)
\bigr)^2 \sigma^2.
\end{eqnarray*}
Now let us go back to \eqref{majoD}. We have that
\begin{eqnarray*}
D&\leq& \exp \biggl[-\frac{\theta u}{2} + k \ln \biggl(1+\theta z_1
+ \sum_{j\geq2} z_v z_b^{j-2}
\frac{\theta^j}{j!} \biggr) \biggr]
\\
&\leq& \exp \biggl[-\theta \biggl(\frac{ u}{2}-kz_1 \biggr) +
k \sum_{j\geq2} z_v z_b^{j-2}
\frac{\theta^j}{j!} \biggr],
\end{eqnarray*}
using that $\ln(1+u)\leq u$.
It is sufficient now to recognize a step of the proof of the Bernstein
inequality (weak version see \cite{stflour}, page 25). Since $kz_1=\square
_{\alpha,\eta,s} b T^{1-\alpha'}/(\log(T))$, one can choose $\alpha'>1,
C_1$ and $C_2$ in the definition (\ref{u}) of $u$ (not depending on
$b$) such that $u/2-kz_1 \geq\sqrt{2kz_v z}+\frac{1}3z_b z$ for some
$z=C_4\log(T)$, where $C_4$ is a constant. Hence,
\[
D\leq\exp \biggl[-\theta\biggl(\sqrt{2kz_v z}+
\frac{1}3z_b z\biggr) + k \sum
_{j\geq
2} z_v z_b^{j-2}
\frac{\theta^j}{j!} \biggr].
\]
One can choose accordingly $\theta$ (as for the proof of the Bernstein
inequality) to obtain a bound in $\mathrm{e}^{-z}$. It remains to choose $C_4$
large enough and only depending on $\alpha,\eta,A$ and $f_0$ to
guarantee that $D\leq \mathrm{e}^{-z} \leq\square_{\alpha,\eta,A,f_0} T^{-\alpha
}$. This concludes the proof of the proposition.

%s7.4.4 #&#
\subsubsection{Proof of Proposition \texorpdfstring{\protect\ref{eqnormH}}{4}}
Let $\Q$ denote a measure such that under $\Q$ the distribution of the
full point process restricted to $(-\infty, 0]$ is identical to the
distribution under $\P$ and such that on $(0, \infty)$ the process
consists of
independent components each being a homogeneous Poisson process with
rate 1. Furthermore, the Poisson processes should be
independent of the process on $(-\infty,0]$. %We let $N_t^{(m)} =
%N^{(m)}((0,t])$ denote the $m$'th counting process on $(0, \infty)$
%and $N_t = \sum_{m} N_t^{(m)}$.
From Corollary~5.1.2 in
\cite{Jacobsen:2006}, the likelihood process is given by
\[
\mathcal{L}_t = \exp \biggl(Mt - \sum_{m}
\int_0^t \lambda_u^{(m)}\,
\mathrm{d}u + \sum_{m}\int_0^t
\log\lambda_u^{(m)} \,\mathrm{d}N_u^{(m)}
\biggr)
\]
and we have for $t \geq0$ the relation
%
%e7.17 #&#
\begin{equation}
\label{eq:ChangeOfMeasure} \E_{\P} \kappa_t(\mathbf{f})^2
= \E_{\Q} \kappa_t(\mathbf{f})^2
\mathcal{L}_t,
\end{equation}
where $\E_{\P}$ and $\E_{\Q}$ denote the expectation with respect to $\P
$ and $\Q$, respectively. Let, furthermore, $\tilde{N}_1 = N_{[-1,0)}$
denote the
total number of points on $[-1,0)$.
Proposition~\ref{eqnormH} will be an easy consequence of the following lemma.

%le4 #&#
\begin{Lemma} \label{lem:measureChangeBound} If the point process is
stationary under $\mathbb{P}$, if
\[
\mathrm{e}^{d} \leq\lambda_t^{(m)} \leq a
(N_1 + \tilde{N}_1) + b
\]
for $t \in[0, 1]$ and for constants $d \in\mathbb{R}$ and $a, b >
0$, and if $\E_{\mathbb{P}} (1+\varepsilon)^{\tilde{N}_1} < \infty$ for
some $\varepsilon> 0$ then for any $\mathbf{f}$,
%
%e7.18 #&#
\begin{equation}
\label{eq:measureChangeBound} Q(\mathbf{f},\mathbf{f}) \geq\zeta\llVert \mathbf{f} \rrVert
^2
\end{equation}
for some constant $\zeta> 0$.
\end{Lemma}

\begin{pf} %By stationarity
%$$Q(\mathbf{f},\mathbf{f}) = \E_{\P} \kappa_1(\mathbf{f})^2.$$
We use H\"{o}lders inequality on $\kappa_1(\mathbf{f})^{\sfrac{2}{p}}
\mathcal{L}_1^{\sfrac{1}{p}}$ and $\kappa_1(\mathbf{f})^{\sfrac{2}{q}}
\mathcal{L}_1^{-\sfrac{1}{p}}$ to get
%
%e7.19 #&#
\begin{equation}
\label{eq:CSbound} \E_{\Q} \kappa_1(\mathbf{f})^2
\leq \bigl(\E_{\Q} \kappa_1(\mathbf{f})^2
\mathcal{L}_1 \bigr)^{\sfrac{1}{p}} \bigl(\E_{\Q}
\kappa_1(\mathbf{f})^2\mathcal{L}_1^{-\sfrac{q}{p}}
\bigr)^{\sfrac
{1}{q}} = Q(\mathbf{f},\mathbf{f})^{\sfrac{1}{p}} \bigl(
\E_{\Q} \kappa_1(\mathbf{f})^2
\mathcal{L}_1^{1-q} \bigr)^{\sfrac{1}{q}},
\end{equation}
where $\frac{1}{p} + \frac{1}{q} = 1$. We choose $q \geq1$ (and thus
$p$) below to
make $q-1$ sufficiently small.
For the left-hand side, we have by independence of the homogeneous Poisson
processes that if $\mathbf{f}=(\mu,(g_\ell)_{\ell=1,\ldots,M})$,
\begin{eqnarray*}
\E_{\Q} \kappa_1(\mathbf{f})^2 & = & \bigl(
\E_{\Q} \kappa_1(\mathbf{f})\bigr)^2 +
\mathbb{V}_{\Q} \kappa_1(\mathbf{f})
\\
& = & \biggl(\mu+ \sum_{\ell} \int
_0^{1} g_{\ell}(u)\, \mathrm{d} u
\biggr)^2 + \sum_{\ell} \int
_0^{1} g_{\ell}(u)^2
\,\mathrm{d} u.
\end{eqnarray*}
Exactly as on page 32 in \cite{RBS} there exists $c'>0$ such that
%
%e7.20 #&#
\begin{equation}
\label{eq:lower1} \E_{\Q} \kappa_1(\mathbf{f})^2
\geq c' \biggl(\mu^2 + \sum
_{\ell} \int_0^{1}
g_{\ell}^2(u)\, \mathrm{d} u \biggr) = c'
\llVert \mathbf{f} \rrVert ^2.
\end{equation}
To bound the second factor on the right-hand side in (\ref{eq:CSbound}) we
observe, by assumption, that we have the lower bound
\[
\mathcal{L}_1 \geq \mathrm{e}^{M(1-b)}\mathrm{e}^{(d - aM)N_1}\mathrm{e}^{-aM \tilde{N}_1}
\]
on the likelihood process. Under $\Q$ we have that $(\kappa_1(\mathbf
{f}), N_1)$ and
$\tilde{N}_1$ are independent, and with $\rho=
\mathrm{e}^{(q-1)(aM -d)}$ and $\tilde{\rho} = \mathrm{e}^{(q-1)(aM)}$ we get that
\[
\E_{\Q} \kappa_1(\mathbf{f})^2
\mathcal{L}_1^{1-q} \leq \mathrm{e}^{(q-1)M(b-1)}
\E_{\Q} \tilde{\rho}^{\tilde{N}_1} \E_{\Q}
\kappa_1(\mathbf{f})^2\rho^{N_1}.
\]
Here we choose $q$ such that $\tilde{\rho}$ is sufficiently close to 1
to make sure
that $\E_{\Q} \tilde{\rho}^{\tilde{N}_1} = \E_{\P} \tilde{\rho}^{\tilde
{N}_1} < \infty$ (see Proposition~\ref{momentsHawkes}).
Moreover, by Cauchy--Schwarz' inequality
%
%e7.21 #&#
\begin{equation}
\label{eq:CSineq} \kappa_1^2(\mathbf{f}) \leq \biggl(
\mu^2 + \sum_\ell\int
_{0}^{1-} g_\ell^2(1-u)
\,\mathrm{d}N^{(\ell)}_u \biggr) (1+N_1).
\end{equation}
Under $\Q$ the point processes on $(0, \infty)$ are homogeneous Poisson
processes with
rate 1 and $N_1$, the total number of points, is Poisson.
% with mean $M$
This implies that conditionally on
$(N_1^{(1)},\ldots,N_1^{(M)})=(n^{(1)},\ldots,n^{(M)})$ the
$n^{(m)}$-points for the $m$th process are uniformly
distributed on $[0,1]$, hence
%
%e7.22 #&#
\begin{eqnarray}
\label{eq:lower2} \E_{\Q} \kappa_1(\mathbf{f})^2
\mathcal{L}_1^{1-q} &\leq& \biggl(\mu^2 + \sum
_\ell\int_{0}^{1}
g_{\ell}^2(u) \,\mathrm{d} u \biggr) \underbrace{\mathrm{e}^{(q-1)M(b-1)}
\E_{\Q} \tilde{\rho}^{\tilde{N}_1} \E_{\Q}
(1+N_1)^2\rho^{N_1}}_{c''}\nonumber
\\[-8pt]
\\[-8pt] &=&
c'' \llVert \mathbf{f} \rrVert ^2.
\nonumber
\end{eqnarray}
Combining (\ref{eq:lower1}) and (\ref{eq:lower2}) with
(\ref{eq:CSbound}) we get that
\[
c' \llVert \mathbf{f} \rrVert ^2 \leq
\bigl(c''\bigr)^{\sfrac{1}{q}}\llVert \mathbf{f}
\rrVert ^{\sfrac
{2}{q}} Q(\mathbf{f},\mathbf{f})^{\sfrac{1}{p}}
\]
or by rearranging that
\[
Q(\mathbf{f},\mathbf{f}) \geq\zeta\llVert \mathbf{f} \rrVert ^2
\]
with $\zeta= (c')^p/(c'')^{p-1}$.
\end{pf}

For the Hawkes process, it follows that if $\nu^{(m)} > 0$ and if
\[
\sup_{t \in[0,1]} h^{(m)}_{\ell}(t) < \infty
\]
for $l,m = 1, \ldots, M$ then for $t \in[0,1]$ we have $\mathrm{e}^{d} \leq
\lambda_t^{(m)} \leq a (N_1
+ \tilde{N}_1) + b$
with
\[
d = \log\nu^{(m)},\qquad  a = \max_{l} \sup
_{t \in[0,1]} h^{(m)}_{\ell}(t),\qquad  b =
\nu^{(m)}.
\]
Proposition~\ref{momentsHawkes} proves that there exists $\e>0$ such
that $\E_{\mathbb{P}} (1+\varepsilon)^{\tilde{N}_1} < \infty$. This
completes the proof of Proposition~\ref{eqnormH}.
%%%%%%%%%%%%%%%%%%%%%%%%%%%%%%%
%s7.5 #&#
\subsection{Proofs of the results of Sections \texorpdfstring{\protect\ref
{sec:aalenvrai}}{4.2} and \texorpdfstring{\protect\ref{LassoH}}{5.2}}

%s7.5.1 #&#
\subsubsection{Proof of Propositions \texorpdfstring{\protect\ref{Qnorm}}{5} 
and \texorpdfstring{\protect\ref{Prop:Aalen}}{1}}\label{2proofs}
We first prove Proposition~\ref{Qnorm}.
As in the proof of Proposition~\ref{flow}, we use the notation $\square$.
Note that for any $\p_1$ and any $\p_2$ belonging to $\Phi$,
\[
G_{\p_1,\p_2}=\sum_{m=1}^M\int
_0^T\kappa_t\bigl(\mathbf{
\p_1}^{(m)}\bigr)\kappa _t\bigl(\mathbf{
\p_2}^{(m)}\bigr)\,\mathrm{d}t
\]
and $\E(G_{\p_1,\p_2})=T\sum_{m=1}^M Q(\mathbf{\p_1}^{(m)},\mathbf{\p
_2}^{(m)})$ by using (\ref{Q}).
This implies that
\[
\E\bigl(a'Ga\bigr)=a'\E(G)a=T\sum
_m Q\bigl(\mathbf{f}_a^{(m)},
\mathbf{f}_a^{(m)}\bigr).
\]
Hence by Proposition~\ref{eqnormH}, $\E(a'Ga)\geq T\zeta \sum_m \llVert  \mathbf{f}_a^{(m)} \rrVert ^2=T\zeta\llVert  f_a  \rrVert ^2$ by definition of the norm on
$\mathcal{H}$. Since $\Phi$ is an orthonormal system, this implies that
$\E(a'Ga)\geq T\zeta\llVert  a  \rrVert _{\ell_2}$. Hence, to show that $\Omega_c$
is a large event for some $c>0$, it is sufficient to show that for some
$0<\epsilon<\zeta$, with high probability, for any $a \in\R^\Phi$,
%
%e7.23 #&#
\begin{equation}
\label{cequonveut} \bigl\llvert a'Ga-a'\E(G) a \bigr
\rrvert \leq T\epsilon\llVert a \rrVert _{\ell_2}^2.
\end{equation}
Indeed, (\ref{cequonveut}) implies that, with high probability, for any
$a \in\R^\Phi$,
\[
a'Ga\geq a'\E(G) a- T \epsilon\llVert a \rrVert
_{\ell_2} \geq T(\zeta-\epsilon) \llVert a \rrVert _{\ell_2},
\]
and the choice $c=T(\zeta-\epsilon)$ is convenient. So, first one has
to control all the coefficients of $G-\E(G)$. For all $\p,\rho\in\Phi
$, we apply Proposition~\ref{flow} to
\[
Z(N)=\sum_m \psi_0^{(m)}(
\p) \psi_0^{(m)}(\rho).
\]
Note that $Z$ only depends on points lying in $[-1,0)$. % For all $f$
%in $\mathcal{H}$, let us denote by $$\left\Vert f  \right\Vert_\infty=\max\left\{
Therefore, $\llvert  Z(N) \rrvert \leq2M\llVert  \p  \rrVert _\infty\llVert  \rho  \rrVert _\infty
(1+N_{[-1,0)}^2 )$. This leads to
\[
\P \biggl(\frac{1}{T} \bigl|G_{\p,\rho}-\E(G_{\p,\rho}) \bigr|\geq
x_{\p
,\rho} \biggr)\leq\square_{\beta,f_0} T^{-\beta}
\]
with
\[
x_{\p,\rho}=\square_{\beta,f_0,M}\bigl[\sigma_{\p,\rho}\log
^{3/2}(T)T^{-1/2}+\llVert \p \rrVert _\infty\llVert
\rho \rrVert _\infty\log^4(T)T^{-1}\bigr]
\]
and
\[
\sigma_{\p,\rho}^2=\E \biggl[ \biggl[\sum
_m \psi_0^{(m)}(\p)
\psi_0^{(m)}(\rho)-\E \biggl(\sum
_m \psi_0^{(m)}(\p)
\psi_0^{(m)}(\rho) \biggr) \biggr]^2\indic
_{N_{[-1,0)}\leq\tilde{\mathcal{N}}} \biggr].
\]
Hence, with probability larger than $1- \square_{\beta,f_0} \llvert  \Phi \rrvert ^2
T^{-\beta}$ one has that
\[
\bigl\llvert a'Ga-a'\E(G) a \bigr\rrvert \leq
\square_{\beta,f_0} \biggl(\sum_{\p,\rho\in\Phi} \llvert
a_\p \rrvert \llvert a_\rho \rrvert \bigl[
\sigma_{\p,\rho}\log^{3/2}(T)T^{1/2}+\llVert \p \rrVert
_\infty \llVert \rho \rrVert _\infty\log^4(T)
\bigr] \biggr).
\]
Hence, for any positive constant $\delta$ chosen later,
%
%e7.24 #&#
\begin{eqnarray}
\label{mamajo} &&\bigl\llvert a'Ga-a'\E(G) a \bigr
\rrvert \nonumber\\
&&\quad \leq \square_{\beta,f_0} \biggl[T\sum_{\p,\rho\in\Phi}
\llvert a_\p \rrvert \llvert a_\rho \rrvert \biggl[\delta
\frac{\sigma_{\p,\rho}^2}{\llVert  \p  \rrVert _\infty\llVert  \rho
\rrVert _\infty}
\\
&&\hphantom{  \square_{\beta,f_0} \biggl[T\sum_{\p,\rho\in\Phi}
\llvert a_\p \rrvert \llvert a_\rho \rrvert \biggl[}\qquad {} + \biggl[\frac{1}{\delta\log(T)} + 1 \biggr]\llVert \p \rrVert
_\infty\llVert \rho \rrVert _\infty\frac{\log^4(T)}{T} \biggr]
\biggr].\nonumber
\end{eqnarray}
Now let us focus on
$E:=\sum_{\p,\rho\in\Phi} \llvert  a_\p \rrvert  \llvert  a_\rho \rrvert  \frac{\sigma_{\p,\rho
}^2}{\llVert  \p  \rrVert _\infty\llVert  \rho  \rrVert _\infty}$. First, we have:
\[
E\leq2 \sum_{\p,\rho\in\Phi} \llvert a_\p \rrvert
\llvert a_\rho \rrvert \frac{\E([\sum_m \psi
_0^{(m)}(\p) \psi_0^{(m)}(\rho)]^2\indic_{N_{[-1,0)}\leq\tilde{\mathcal{N}}}
)+(\E[\sum_m \psi_0^{(m)}(\p) \psi_0^{(m)}(\rho)])^2}{ \llVert  \p  \rrVert _\infty\llVert  \rho \rrVert _\infty}
\]
with $\tilde{\mathcal{N}}:=\square_{\beta,f_0} \log(T)$.
Next,
\[
\sum_m \psi_0^{(m)}(\p)
\psi_0^{(m)}(\rho) \leq 2M\llVert \p \rrVert
_\infty\llVert \rho \rrVert _\infty\bigl(1+N_{[-1,0)}^2
\bigr).
\]
Hence, if $N_{[-1,0)}\leq\tilde{\mathcal{N}}=\square_{\beta,f_0} \log
(T)$, for $T$ large enough,
\[
\sum_m \psi_0^{(m)}(\p)
\psi_0^{(m)}(\rho) \leq\square_{\beta,M,f_0}\llVert \p
\rrVert _\infty\llVert \rho \rrVert _\infty\log^2(T)
\]
and
\[
\E\biggl(\sum_m \psi_0^{(m)}(
\p) \psi_0^{(m)}(\rho)\biggr)\leq\square_{\beta,M,f_0}
\llVert \p \rrVert _\infty\llVert \rho \rrVert _\infty
\log^2(T).
\]
Hence,
\[
E\leq \square_{\beta,M,f_0} \log^2(T) \sum
_{\p,\rho\in\Phi} \llvert a_\p \rrvert \llvert
a_\rho \rrvert \E \biggl( \biggl|\sum_m
\psi_0^{(m)}(\p) \psi_0^{(m)}(\rho) \biggr|
\biggr).
\]
But note that for any $f$, $\llvert  \psi_0^{(m)}(f) \rrvert \leq\psi_0^{(m)}(\llvert  f \rrvert )$ where
$\llvert  f \rrvert =((\llvert  \mu^{(m)} \rrvert ,(\llvert  g_\ell^{(m)} \rrvert )_{\ell=1,\ldots,M})_{m=1,\ldots,M})$. Therefore,
\begin{eqnarray*}
E&\leq& \square_{\beta,M,f_0} \log^2(T) \sum
_{\p,\rho\in\Phi} \llvert a_\p \rrvert \llvert
a_\rho \rrvert \E \biggl(\sum_m
\psi_0^{(m)}\bigl(\llvert \p \rrvert \bigr)
\psi_0^{(m)}\bigl(\llvert \rho \rrvert \bigr)   \biggr)
\\
&\leq& \square_{\beta,M,f_0} \log^2(T) \sum
_m\E \biggl( \biggl[\sum_{\p\in
\Phi}
\llvert a_{\p} \rrvert \psi_0^{(m)}\bigl(\llvert
\p \rrvert \bigr) \biggr]^2 \biggr)
\\
&\leq& \square_{\beta,M,f_0} \log^2(T) \sum
_m\E \biggl( \biggl[ \psi_0^{(m)}
\biggl(\sum_{\p\in\Phi}\llvert a_{\p} \rrvert
\llvert \p \rrvert \biggr) \biggr]^2 \biggr).
\end{eqnarray*}
But if $\p=(\mu^{(m)}_\p, ((g_\p)_\ell^{(m)})_\ell)_m$, then
\[
\biggl[ \psi_0^{(m)} \biggl(\sum
_{\p\in\Phi}\llvert a_{\p} \rrvert \llvert \p \rrvert
\biggr) \biggr]^2= \Biggl[\sum_\p
\llvert a_\p \rrvert \mu^{(m)}_\p+ \sum
_{\ell=1}^M\int_{-1}^{0-}
\sum_\p\llvert a_\p \rrvert \bigl
\llvert (g_\p)_\ell^{(m)} \bigr\rrvert (-u)
\,\mathrm{d}N^{(\ell)}_u \Biggr]^2.
\]
If one creates artificially a process $N^{(0)}$ with only one point and
if we decide that $(g_\p)_0^{(m)}$ is the constant function equal to $\mu^{(m)}
_\p$, this can also be rewritten as
\[
\biggl[ \psi_0^{(m)} \biggl(\sum
_{\p\in\Phi}\llvert a_{\p} \rrvert \llvert \p \rrvert
\biggr) \biggr]^2= \Biggl[ \sum_{\ell=0}^M
\int_{-1}^{0-} \sum_\p
\llvert a_\p \rrvert \bigl\llvert (g_\p)_\ell^{(m)}
\bigr\rrvert (-u) \,\mathrm{d}N^{(\ell)}_u
\Biggr]^2.
\]
Now we apply the Cauchy--Schwarz inequality for the measure $\sum_\ell
\mathrm{d}N^{(\ell)}$, which gives
\[
\biggl[ \psi_0^{(m)} \biggl(\sum
_{\p\in\Phi}\llvert a_{\p} \rrvert \llvert \p \rrvert
\biggr) \biggr]^2\leq (N_{[-1,0)}+1)\sum
_{\ell=0}^M \int_{-1}^{0-}
\biggl[\sum_\p\llvert a_\p \rrvert
\bigl\llvert (g_\p)_\ell^{(m)} \bigr\rrvert (-u)
\biggr]^2 \,\mathrm{d}N^{(\ell)}_u.
\]
Consequently,
\begin{eqnarray*}
E&\leq& \square_{\beta,M,f_0} \log^2(T)\sum
_{m=1}^M\sum_{\ell=0}^M
\E \biggl((N_{[-1,0)}+1) \int_{-1}^{0-}
\biggl[\sum_\p\llvert a_\p \rrvert
\bigl\llvert (g_\p)_\ell^{(m)} \bigr\rrvert (-u)
\biggr]^2 \,\mathrm{d}N^{(\ell
)}_u \biggr)
\\
&\leq& \square_{\beta,M,f_0} \log^2(T)\sum
_{m=1}^M\sum_{\ell=0}^M
\sum_{\p,\rho\in
\Phi}\llvert a_\p \rrvert \llvert
a_\rho \rrvert
\\
&&{} \times\E \biggl(\int_{-1}^{0-} (N_{[-1,0)}+1)
\bigl\llvert (g_\p)_\ell^{(m)} \bigr\rrvert (-u)
\bigl\llvert (g_\rho)_\ell^{(m)} \bigr\rrvert (-u)
\,\mathrm{d}N^{(\ell)}_u \biggr).
\end{eqnarray*}
Now let us use the fact that for every $x,y\geq0$, $\eta,\theta>0$
that will be chosen later,
\[
xy - \eta \mathrm{e}^{\theta x} \leq\frac{y}{\theta} \bigl[\log(y)-\log(\eta
\theta)-1 \bigr],
\]
with the convention that $y\log(y)=0$ if $y=0$.
Let us apply this to $x=N_{[-1,0)}+1$ and $y=\llvert  (g_\p)_\ell^{(m)} \rrvert (-u)\llvert  (g_\rho
)_\ell^{(m)} \rrvert (-u)$. We obtain that
\begin{eqnarray*}
E&\leq&\square_{\beta,M,f_0} \eta\log^2(T)\sum
_{m=1}^M \sum_{\p,\rho
\in\Phi}
\llvert a_\p \rrvert \llvert a_\rho \rrvert \E
\bigl((N_{[-1,0)}+1) \mathrm{e}^{\theta
(N_{[-1,0)}+1)} \bigr)
\\
&&{}+\square_{\beta,M,f_0} \theta^{-1} \log^2(T)\sum
_{m=1}^M\sum_{\ell=0}^M
\sum_{\p,\rho\in
\Phi}\llvert a_\p \rrvert \llvert
a_\rho \rrvert
\\
&&\hphantom{{}+{}}{}\times\E \biggl(\int_{-1}^{0^-}\bigl\llvert
(g_\p)_\ell^{(m)} \bigr\rrvert \bigl\llvert
(g_\rho)_\ell^{(m)} \bigr\rrvert (-u) \bigl[\log
\bigl(\bigl\llvert (g_\p)_\ell^{(m)} \bigr\rrvert
\bigl\llvert (g_\rho)_\ell^{(m)} \bigr\rrvert (-u)
\bigr)-\log(\eta\theta)-1 \bigr] \,\mathrm{d}N^\ell_u
\biggr).
\end{eqnarray*}
Since for $\ell>0$, $\mathrm{d}N^{(\ell)}_u$ is stationary, one can
replace $\E(\mathrm{d}N^{(\ell)}_u)$ by $\square_{f_0} \,\mathrm{d}u$. Moreover,
since by Proposition~\ref{momentsHawkes}, $N_{[-1,0)}$ has some
exponential moments there exists $\theta=\square_{f_0}$ such that $\E
((N_{[-1,0)}+1) \mathrm{e}^{\theta(N_{[-1,0)}+1)} )=\square_{f_0}$.
With $\llvert  \Phi \rrvert $ the size of the dictionary, this leads to
\begin{eqnarray*}
E&\leq&\square_{\beta,M,f_0} \eta\llvert \Phi \rrvert \log^2(T)
\llVert a \rrVert _{\ell_2} ^2
\\
&&{}+\square_{\beta,M,f_0} \log^2(T)\\
&&\hphantom{{}+{}}{}\times\sum_{m=1}^M
\Biggl[\sum_{\p,\rho\in\Phi
}\llvert a_\p \rrvert
\llvert a_\rho \rrvert \bigl\llvert \mu^{(m)}_\p
\bigr\rrvert \bigl\llvert \mu^{(m)}_\rho \bigr\rrvert \bigl[
\log\bigl(\bigl\llvert \mu^{(m)}_\p \bigr\rrvert \bigl
\llvert \mu_\rho^{(m)} \bigr\rrvert \bigr)-\log(\eta\theta)-1
\bigr]
\\
&&\hphantom{{}+{}\times\sum_{m=1}^M
\Biggl[}{} +\sum_{\ell=1}^M\sum
_{\p,\rho\in\Phi}\llvert a_\p \rrvert \llvert
a_\rho \rrvert \int_{0}^{1}\bigl
\llvert (g_\p)_\ell^{(m)} \bigr\rrvert \bigl
\llvert (g_\rho)_\ell^{(m)} \bigr\rrvert (u) \\
&&\hspace*{42.5pt}\hphantom{{}+\square_{\beta,M,f_0} \log^2(T)\sum_{m=1}^M
\Biggl[}{}\times\bigl[
\log\bigl(\bigl\llvert (g_\p)_\ell^{(m)} \bigr
\rrvert \bigl\llvert (g_\rho)_\ell^{(m)} \bigr
\rrvert (u)\bigr)-\log(\eta\theta)-1 \bigr] \,\mathrm{d}u \Biggr].
\end{eqnarray*}
Consequently, using $\llVert  \Phi  \rrVert _\infty$ and $r_\Phi$,
\[
E\leq\square_{\beta,M,f_0} \eta\llvert \Phi \rrvert \log^2(T)
\llVert a \rrVert _{\ell_2}^2 + \square_{\beta,M,f_0}
\log^2(T) r_\Phi\bigl[2\log\bigl(\llVert \Phi \rrVert
_\infty\bigr)-\log (\eta\theta)-1\bigr]\llVert a \rrVert
_{\ell_2} ^2.
\]
We choose $\eta=\llvert  \Phi \rrvert ^{-1}$ and obtain that
\[
E\leq\square_{\beta,M,f_0} \log^2(T) r_\Phi\bigl[\log
\bigl(\llVert \Phi \rrVert _\infty \bigr)+\log\bigl(\llvert \Phi \rrvert
\bigr)\bigr]\llVert a \rrVert _{\ell_2}^2.
\]
Now, let us choose $\delta=\omega/(\log^2(T) r_\Phi[\log(\llVert  \Phi
\rrVert _\infty)+\log(\llvert  \Phi \rrvert )])$ where $\omega$ depends only on $\beta, M $
and $f_0$ and will be chosen later and let us go back to \eqref{mamajo}:
\begin{eqnarray*}
\frac{1}{T}\bigl\llvert a'Ga-a'\E(G) a \bigr
\rrvert & \leq& \square_{\beta,M, f_0}\omega\llVert a \rrVert
_{\ell_2}^2\\
&&{} + \square_{\beta,
f_0, \omega} r_\Phi\bigl[
\log\bigl(\llVert \Phi \rrVert _\infty\bigr)
 +\log\bigl(\llvert \Phi \rrvert \bigr)\bigr] \\
&&\hphantom{   {}+{} }{}\times\sum
_{\p,\rho\in\Phi} \llvert a_\p \rrvert \llvert
a_\rho \rrvert \llVert \p \rrVert _\infty \llVert \rho
\rrVert _\infty\frac{\log^5(T)}{T}
\\
& \leq&\square_{\beta,M, f_0} \omega\llVert a \rrVert _{\ell_2}^2
+\square _{\beta, f_0, \omega} \llVert a \rrVert _{\ell_2}^2
A_\Phi(T).
\end{eqnarray*}
Under assumptions of Proposition~\ref{Qnorm}, for $T_0$ large enough
and $T\geq T_0$,
\[
\frac{1}{T}\bigl\llvert a'Ga-a'\E(G) a \bigr
\rrvert \leq\square_{\beta,M, f_0} \omega\llVert a \rrVert _{\ell_2}^2.
\]
It is now sufficient to take $\omega$ small enough and then $T_0$ large
enough to obtain \eqref{cequonveut} with $\epsilon<\zeta$ and
Proposition~\ref{Qnorm} is proved.

Arguments for the proof of Proposition~\ref{Prop:Aalen} are similar. So
we just give a brief sketch of the proof. Now,
\[
G_{\p_1,\p_2}=\sum_{m=1}^M\int
_0^1 \bigl(Y^{(m)}_t
\bigr)^2 \p_1\bigl(t,X^{(m)}\bigr) \p
_2\bigl(t,X^{(m)}\bigr) \,\mathrm{d}t.
\]
Let $\beta>0$. With probability larger than $1-2M^{-\beta}$,
\[
\frac{1}{M}\bigl\llvert G_{\p_1,\p_2}-\E[G_{\p_1,\p_2}] \bigr
\rrvert \leq\sqrt{\frac{2\beta
v_{\p_1,\p_2}\log M}{M}}+\frac{\beta b_{\p_1,\p_2}\log M}{3M},
\]
with
\begin{eqnarray*}
b_{\p_1,\p_2}&=&\| \p_1 \|_{\infty}\|\p_2
\|_{\infty},
\\
v_{\p_1,\p_2}&=&\E \biggl(\int_0^1
\bigl(Y^{(m)}_t \bigr)^2\p_1
\bigl(t,X^{(m)}\bigr) \p_2\bigl(t,X^{(m)} \bigr)\,
\mathrm{d}t \biggr)^2\leq D\|\p_1\|_{\infty}\|
\p_2\|_{\infty}\bigl\langle \llvert \p_1 \rrvert ,
\llvert \p_2 \rrvert \bigr\rangle,
\end{eqnarray*}
where $\langle\cdot,\cdot\rangle$ denotes the standard $\mathbb
{L}_2$-scalar product. We have just used the classical Bernstein
inequality combined with (\ref{condY}). So, with probability larger
than $1-2\llvert  \Phi \rrvert ^2M^{-\beta}$,
for any vector $a$ and any $\delta>0$,
\begin{eqnarray*}
\bigl\llvert a'Ga-\E\bigl[a'Ga\bigr] \bigr\rrvert &
\leq&\square_{D,\beta}\sum_{\p_1,\p_2}\llvert
a_{\p_1} \rrvert \llvert a_{\p
_2} \rrvert \bigl[\delta M \bigl
\langle\llvert \p_1 \rrvert ,\llvert \p_2 \rrvert \bigr
\rangle+\delta^{-1}\log M\|\p_1\| _\infty\|
\p_2\|_\infty\bigr]
\\
&\leq&\square_{D,\beta}\bigl(\delta M r_\Phi+
\delta^{-1}\|\Phi\|_{\infty
}^2\llvert \Phi \rrvert
\log M\bigr)\|a\|_{\ell_2}^2.
\end{eqnarray*}
We choose $\delta=\sqrt{\frac{\|\Phi\|_\infty^2\llvert  \Phi \rrvert \log M}{M r_\Phi
}}$, so that with probability larger than $1-2\llvert  \Phi \rrvert ^2M^{-\beta}$,
\[
\frac{1}{M}\bigl\llvert a'Ga-\E\bigl[a'Ga
\bigr] \bigr\rrvert \leq\square_{D,\beta}\sqrt{\frac{\|\Phi\|
_\infty^2r_\Phi\llvert  \Phi \rrvert \log M}{M}}
\|a\|_{\ell_2}^2.
\]
We use (\ref{condr}) and (\ref{cond:Aalen}) to conclude as for
Proposition~\ref{Qnorm} and we obtain Proposition~\ref{Prop:Aalen}.
%%%%%%%%%%%%%%%%%%%%%%%

%s7.5.2 #&#
\subsubsection{Proof of Corollary \texorpdfstring{\protect\ref{OmegaN}}{3}}
First, let us cut $[-1,T]$ in $\lfloor T \rfloor+2$ intervals $I$'s of
the type $[a,b)$ such that the first $\lfloor T \rfloor+1$ intervals
are of length 1 and the last one is of length strictly smaller than 1
(eventually it is just a singleton). Then, any interval of the type
$[t-1,t]$ for $t$ in $[0,T]$ is included into the union of two such
intervals. Therefore, the event where all the $N_I$'s are smaller than
$u=\mathcal{N}/2$ is included into $\Omega_\mathcal{N}$. It remains to
control the probability of the complementary of this event.
By stationarity, all the first $N_I$'s have the same distribution and
satisfy Proposition~\ref{momentsHawkes}. The last one can also be
viewed as the truncation of a stationary point process to an interval
of length smaller than 1. Therefore, the exponential inequality of
Proposition~\ref{momentsHawkes} also applies to the last interval.
It remains to apply $\lfloor T \rfloor+2$ times this exponential
inequality and to use a union bound.
%%%%%%%%%%%%%%%%%%%%%%%%%%%%%%%

%s7.5.3 #&#
\subsubsection{Proof of Corollary \texorpdfstring{\protect\ref{etpourH}}{4}}
As in the proof of Proposition~\ref{flow}, we use the notation $\square$.
The nonasymptotic part of the result is just a pure application of
Theorem~\ref{th:oraclewithd}, with the choices of $B_\p$ and $V_\p$
given by \eqref{Bphi} and \eqref{Vphi}.
The next step consists in controlling the martingale $\psi(\p)^2\bullet
(N-\Lambda)_T$ on $\Omega_{V,B}$.
To do so, let us apply \eqref{sho} to $H$ such that for any $m$,
\[
H^{(m)}_t=\psi^{(m)}_t(
\p)^2\indic_{t\leq\tau'},
\]
with $B=B_\p^2$ and $\tau=T$ and where $\tau'$ is defined in \eqref
{monT} (see the proof of Theorem~\ref{th:oraclewithd}). The assumption
to be fulfilled is checked as in the proof of Theorem~\ref{th:oraclewithd}.
But as previously, on $\Omega_{V,B}$, $H\bullet(N-\Lambda)_T=\psi(\p
)^2\bullet(N-\Lambda)_T$ and also
$H^2\bullet\Lambda_T=\psi(\p)^4\bullet\Lambda_T$. Moreover, on $\Omega
_\mathcal{N}\subset\Omega_{V,B}$
\[
H^2\bullet\Lambda_T=\psi(\p)^4\bullet
\Lambda_T \leq v:=TM\Bigl(\max_m
\nu^{(m)} +\mathcal{N}\max_{m,\ell} h_\ell^{(m)}
\Bigr) B_\p^4.
\]
Recall that $x=\alpha\log(T)$.
So on $\Omega_{V,B}$, with probability larger than
$1-(M+KM^2)\mathrm{e}^{-x}=1-(M+KM^2)T^{-\alpha}$, one has that for all $\p\in
\Phi$,
\[
\psi(\p)^2\bullet N_T \leq\psi(\p)^2\bullet
\Lambda_T +\sqrt{2vx}+\frac
{B_\p^2x}{3}.
\]
So that for all $\p\in\Phi$,
\[
\psi(\p)^2\bullet N_T \leq\square_{M,f_0} \bigl[
\mathcal{N}\llVert \p \rrVert ^2_T+\llVert \Phi \rrVert
_\infty^2\mathcal{N}^2\sqrt{T\mathcal{N}\log(T)}
\bigr].
\]
Also, since $\mathcal{N}=\log^2(T)$, one can apply Corollary~\ref
{OmegaN}, with $\beta=\alpha$. We finally choose $c$ as in Proposition~\ref{Qnorm}. This leads to the result.

\section*{Acknowledgements}
We are very grateful to Christine Tuleau-Malot who allowed us to use
her R programs simulating Hawkes processes. The research of Patricia
Reynaud-Bouret and Vincent Rivoirard is partly supported by the french
Agence Nationale de la Recherche (ANR 2011 BS01 010 01 projet
Calibration). The authors would like to thank the anonymous Associate
Editor and Referees for helpful comments and suggestions.

%suskaldyti doi

% imsref loaded by jurgita.kaciuliene, 2014-01-19 12:12:37
% imsref loaded by jurgita.kaciuliene, 2014-01-19 15:14:05

\printhistory

\end{document}